\documentclass[preprint,12pt]{elsarticle}

\usepackage{amssymb}
\usepackage{amsmath}
\usepackage{mathtools}
\usepackage{amsfonts}
\usepackage{amsthm}
\usepackage{float}
\usepackage{times}
\usepackage{tikz-cd}
\usepackage{mathabx}
\usepackage{stmaryrd}
\usepackage{multicol}
\usetikzlibrary{arrows}

\floatstyle{boxed}
\newtheorem{theorem}{Theorem}[section]

\newtheorem{lemma}[theorem]{Lemma}
\newtheorem{corollary}[theorem]{Corollary}
\theoremstyle{definition}
\newtheorem{example}[theorem]{Example}
\newtheorem{remark}[theorem]{Remark}

\newcommand{\m}{\mathbf} %fontshape for math structures and algebras

\newcommand{\ra}{\mathbin{\rightarrow}}
\newcommand{\jn}{\vee}
\newcommand{\mt}{\wedge}

\newcommand{\ls}{\setbox0\hbox{$-$}
\mathbin{\hbox{$-$\kern-\wd0\raise2\dp0\hbox{$\cdot$}\kern.3\wd0\lower2\dp0\hbox{$\cdot$}}}}
\newcommand{\rs}{\setbox0\hbox{$-$}
\mathbin{\hbox{$-$\kern-\wd0\lower2\dp0\hbox{$\cdot$}\kern.3\wd0 \raise2\dp0\hbox{$\cdot$}}}}

\newcommand{\lsc}{\setbox0\hbox{$-$}
\mathbin{\hbox{$-$\kern-\wd0\raise2\dp0\hbox{$\cdot$}\kern.3\wd0\lower2\dp0\hbox{\phantom{$\cdot$}}}}}

\newcommand{\rsc}{\setbox0\hbox{$-$}
\mathop{\hbox{$-$\kern-\wd0\lower2\dp0\hbox{\phantom{$\cdot$}}\kern.3\wd0\raise2\dp0\hbox{$\cdot$}}}}

\newcommand{\lrsc}{\setbox0\hbox{$-$}
\mathop{\hbox{$-$\kern-\wd0\raise2\dp0\hbox{$\cdot$}\kern.3\wd0\raise2\dp0\hbox{$\cdot$}}}}

\renewcommand{\ln}{{\sim}}
\newcommand{\rn}{{-}}

\newcommand{\CHAT}[2]{{\small\bf\color{red} {#1: }}{\color{blue}#2}}

\renewcommand{\rm}{\cdot} % right ring multiplication
\newcommand{\diagcov}{\Yleft}

\journal{Journal of Algebra}

\begin{document}

\begin{frontmatter}

\title{Generation and decidability for periodic $\ell$-pregroups}
\author[inst1]{Nikolaos Galatos}
\ead{ngalatos@du.edu}

\affiliation[inst1]{organization= {Department of Mathematics, University of Denver},
            addressline={2390 S. York St.}, 
            city={Denver},
            postcode={80208}, 
            state={CO},
            country={USA}}

\author[inst1]{Isis A. Gallardo}
\ead{isis.gallardo@du.edu}

\begin{abstract}

In  \cite{GG} it is shown that the variety $\mathsf{DLP}$ of distributive $\ell$-pregroups is generated by a single algebra, the functional algebra $\m{F}(\mathbb{Z})$ over the integers.
Here, we show that $\mathsf{DLP}$ is equal to the join of its subvarieties $\mathsf{LP_n}$, for $n \in \mathbb{Z}^+$, consisting of $n$-periodic $\ell$-pregroups. We also prove that every algebra in $\mathsf{LP_n}$ embeds into the subalgebra $\m{F}_n(\m \Omega)$ of $n$-periodic elements of $\m{F}(\m \Omega)$, for some integral chain $\m \Omega$; we use this representation to show that for every $n$, the variety $\mathsf{LP_n}$ is generated by the single algebra  $\m{F}_n(\mathbb{Q} \overrightarrow{\times}\mathbb{Z})$, noting that the  chain $\mathbb{Q} \overrightarrow{\times}\mathbb{Z}$
%$\m \Omega =\mathbb{Q} \overrightarrow{\times}\mathbb{Z}$ 
is independent of $n$.

We further establish a second representation theorem: every algebra in $\mathsf{LP_n}$ embeds into  the wreath product of an $\ell$-group and $\m{F}_n(\mathbb{Z})$, showcasing the prominent role of the simple $n$-periodic $\ell$-pregroup $\m{F}_n(\mathbb{Z})$.
%Here, we prove that for every $n$, the  $n$-periodic variety of lattice-ordered pregroups ($\mathsf{LP_n}$) is generated by a single algebra, the subalgebra $\m{F}_n(\m \Omega)$ of $n$-periodic elements of $\m{F}(\m \Omega)$, where the chain $\m \Omega =\mathbb{Q} \overrightarrow{\times}\mathbb{Z}$ is independent of $n$. 
 Moreover, we prove that 
the join of the varieties $\mathsf{V}(\m{F}_n(\mathbb{Z}))$ is also equal to $\mathsf{DLP}$, hence equal to the join of the varieties $\mathsf{LP_n}$, even though $\mathsf{V}(\m{F}_n(\mathbb{Z}))\not =\mathsf{LP_n}$ for every single $n$. In this sense  $\mathsf{DLP}$ has two different  well-behaved approximations.
%can be approximated by the the varieties $\mathsf{LP_n}$, as well as $\mathsf{V}(\m{F}_n(\mathbb{Z}))$, can be seen as for. 

We further prove that, for every $n$, the equational theory of $\m{F}_n(\mathbb{Z})$ is decidable and, using the wreath product decomposition, we show that  the equational theory of $\mathsf{LP_n}$ is decidable, as well.
\end{abstract}
\end{frontmatter}

\section{Introduction}

 A \emph{lattice-ordered pregroup} (\emph{$\ell$-pregroup}) is an algebra $(A,\wedge,\vee,\cdot,^{\ell} ,^{r},1)$, where $(A,\wedge,\vee)$ is a lattice, $(A,\cdot,1)$ is a monoid, mutilication preserves the lattice order $\leq$, and for all $x$,  
 $$
x^{\ell} x\leq1\leq xx^{\ell} \text{ and }xx^{r}\leq1\leq x^{r}x.
$$ We often refer to $x^\ell$ and to $x^r$ as the \emph{left} and \emph{right inverse} of $x$, respectively.
 
The $\ell$-pregroups that satisfy $x^\ell=x^r$ are exactly the \emph{lattice-ordered groups} (\emph{$\ell$-groups}): algebras $(A,\wedge,\vee,\cdot,{}^{-1},1)$, where $(A,\wedge,\vee)$ is a lattice $(A,\cdot,{}^{-1}, 1)$ is a group, and mutilication preserves the order. Lattice-ordered groups have been studied extensively (\cite{AF}, \cite{Da}, \cite{GHo}, \cite{KM}) and they admit a Calyey-style representation theorem due to Holland \cite{Ho-em}: every $\ell$-group  can be embedded into the \emph{symmetric} $\ell$-group of order-preserving permutations on a totally-ordered set.

 \emph{Pregroups} are defined in a similar way as $\ell$-pregroups, but without the stipulation that the underlying order yields a lattice. Pregroups were introduced and studied in the context of mathematical linguistics, both in the theoretical realm (in connection to context-free grammars and automata) and in applications (for studying the sentence structure of various natural languages), see \cite{La}, \cite{Bu}, \cite{Ba}. Pregroups where the order is discrete (or, equivalently, that satisfy $x^\ell=x^r$) are exactly groups. 
 
 It turns out that $\ell$-pregroups are precisely the \emph{residuated lattices} that satisfy $(xy)^\ell =y^\ell x^\ell$ and $x^{r\ell}=x=x^{\ell r}$ (see \cite{GJKO}), hence they can be axiomatized by equations and they form a variety of algebras which we denote by $\mathsf{LP}$. Since other examples of residuated lattices include the lattice of ideals of a ring with unit, Boolean algebras, and the algebra of binary relations over a set among others, $\ell$-pregroups enjoy  common properties  with these structures (for example, congruences correspond to `normal' subalgebras). Furthermore, since residuated lattices form algebraic semantics for \emph{substructural logics} \cite{GJKO} (including intuitionistic, linear, relevance, and many-valued logic), the study of $\ell$-pregroups relates to the study of these non-classical logical systems, as well.

  \medskip

An $\ell$-pregroup is called \emph{distributive}, if its underlying lattice is distributive. It is shown in \cite{GH} that distributive $\ell$-pregroups, like $\ell$-groups,  also enjoy a Cayley/Holland embedding theorem: every distributive $\ell$-pregroup can be embedded in the \emph{functional/symmetric} $\ell$-pregroup $\m F(\m \Omega)$ over a chain $\m \Omega$. Here $\m F(\m \Omega)$ consists of all functions on $\m \Omega$ that have \emph{residuals} and \emph{dual residuals} of all orders (as defined in Section~\ref{s: F(Omega)}) under composition and pointwise order. In \cite{GG} we showed that this representation theorem can be improved in the sense that the chain $\m \Omega$ can be assumed to be integral; a chain is called \emph{integral} if every point is contained in an interval isomorphic to $\mathbb{Z}$. This tighter representation is used in \cite{GG} to show that the variety of distributive $\ell$-pregroups is generated by a single algebra: the functional algebra $\m{F}(\mathbb{Z})$. Here, we establish similar generation results for all periodic varieties of $\ell$-pregroups.

An  $\ell$-pregroup is called \emph{$n$-periodic}, for $n \in \mathbb{Z}^+$, if it satisfies $x^{\ell^n}=x^{r^n}$; we denote the corresponding variety by $\mathsf{LP_n}$. For example, $2$-periodic $\ell$-pregroups satisfy $x^{\ell \ell}= x^{rr}$, while $1$-periodic $\ell$-pregroups are precisely the $\ell$-groups; we say that an $\ell$-pregroup is \emph{periodic} if it is $n$-periodic for some $n \in \mathbb{Z}^+$. As we mention right below, periodicity is related to distributivity, so there is hope of using a similar approach as in \cite{GG} to obtain a representation theorem and a generation result.

In particular, even though it is still 
an open problem whether every $\ell$-pregroup is distributive (we only know that  the underlying lattice is semi-distributive by \cite{GJKP}),  $\ell$-groups are known to be distributive and in \cite{GJ} it is shown that actually all periodic $\ell$-pregroups are distributive, i.e.,  $\mathsf{LP_n}$ is a subvariety of $\mathsf{DLP}$ for all $n \in \mathbb{Z}^+$.

In Section~\ref{s: elements of F_N} , we make use of this fact to show that the representation theorem of \cite{GG} for distrubutive $\ell$-pregroups restricts nicely to $n$-periodic $\ell$-pregroups, for every $n$. In particular, first we observe that if $\m \Omega$ is a chain, then the elements of $\m F(\m \Omega)$ that satisfy $x^{\ell^n}=x^{r^n}$ form a subalgebra $\m F_n (\m \Omega)$ of $\m F(\m \Omega)$ and we provide an alternative characterization for these elements; then we show that every $n$-periodic $\ell$-pregroup embeds in $\m F_n (\m \Omega)$, for some integral chain $\m \Omega$ (first representation for $n$-periodic $\ell$-pregroups). In other
words the operator $\m F_n$ does for $\mathsf{LP_n}$ what the operator $\m F$ does for $\mathsf{DLP}$, at least in terms of an embedding theorem.
Moreover, utilizing that every integral chain $\m \Omega$ is locally isomorphic to $\mathbb{Z}$, we show that every function on $\m \Omega$ decomposes into a global component and many local components, where the latter are elements of $\m F_n (\mathbb{Z})$. This shows the important role that $\m F_n (\mathbb{Z})$ plays; we give an acessible description of the elements of $\m F_n (\mathbb{Z})$ and show that the latter is a simple algebra. Finally, we extend the notion of wreath product (for groups, monoids, and $\ell$-groups) to $\ell$-pregroups and prove that  every $n$-periodic $\ell$-pregroup embeds into the wreath product $\m H \wr \m F_n (\mathbb{Z})$ of an $\ell$-group $\m H$ and  $\m F_n (\mathbb{Z})$ (second representation for $n$-periodic $\ell$-pregroups).

 As mentioned earlier, $\ell$-pregroups are involutive residuated lattices and it turns out that the notion of $n$-periodicity for $\ell$-pregroups is just a specialization of the one for involutive residuated lattices.    In \cite{GJframes} it is shown that the join of all of the $n$-periodic varieties of involutive residuated lattices, for  $n \in \mathbb{Z}^+$, is equal to the whole variety of involutive residuated lattices; the proof is quite involved using both algebraic and proof-theoretic arguments. It is natural to ask what is the join of all of the $\mathsf{LP_n}$'s and in particular whether the join is all of  $\mathsf{DLP}$, given that each $\mathsf{LP_n}$ is a subvariety of $\mathsf{DLP}$. Since we do not have an analytic proof-theoretic calculus for $\ell$-pregroups we cannot use the methods of \cite{GJframes}. 
% In Section~\ref{}, we  prove that indeed $\bigvee \mathsf{LP_n}=\mathsf{DLP}$, by relying on the notion of diagram that is considered in \cite{GG}.    \CHAT{N}{Consider the presentation order: decidability vs joins. Also, revisit after all sections have been (re)written}

In Section~\ref{s: joinperiodic} we prove that indeed  $\mathsf{DLP}$ is equal to the join of all of the varieties  $\mathsf{LP_n}$. In more detail, using the method of diagrams of \cite{GG} we show that if an equation fails in $\mathsf{DLP}$, then it fails in  $\m F_n (\mathbb{Z})$, for some $n$. As a result we obtain that $\mathsf{DLP}$ is also equal to the join of the varieties $\mathsf{V}(\m F_n (\mathbb{Z}))$, thus complementing nicely the generation result for  $\mathsf{DLP}$ given in \cite{GG}. With the goal of converting the representation theorem for $\mathsf{LP_n}$, for every $n$, into a generation result, we search for  chains $\m \Omega_n$ such that the variety  $\mathsf{LP_n}$ is generated by $\m F_n (\m \Omega_n)$. Since $\bigvee \mathsf{LP_n}=\bigvee \mathsf{V}(\m F_n (\mathbb{Z}))$, we investigate whether we can take $\m \Omega_n=\mathbb{Z}$ for all $n$. It is easy to see that this cannot work for $n=1$, as $\m F_n (\mathbb{Z})$ is isomorphic to the $\ell$-group on the integers, hence it generates the variety of abelian $\ell$-groups, not the variety of all $\ell$-groups. We further prove that, unfortunately, $\mathsf{LP_n}\not = \mathsf{V}(\m F_n (\mathbb{Z}))$ for every single $n$, contrasting with $\bigvee \mathsf{LP_n}=\bigvee \mathsf{V}(\m F_n (\mathbb{Z}))$.
%  Starting with the case $n=1$, since the variety $\mathsf{DLP}$ is generated by $\m F (\mathbb{Z})$, the first thought would be to take $\m \Omega_1$ to be $\mathbb{Z}$.
Since  $\m F_1 (\mathbb{Q})$ ends up being the $\ell$-group or order-preserving permutations on $\mathbb{Q}$, Holland's generation theorem shows that
we can take $\m \Omega_1$ to be $\mathbb{Q}$, i.e.,  $\mathsf{LP_1}=\mathsf{V}(\m F_1 (\mathbb{Q}))$.
%However, $\mathsf{LP_1}$ is the variety of $\ell$-groups and it is not generated by $\m F_1 (\mathbb{Z})$ since the latter contains only the translations $x \mapsto x+k$ and it is isomorphic to the $\ell$-group of the integers, which generates the variety of abelian $\ell$-groups, not the variety of all $\ell$-groups. Actually, Holland's generation theorem states that the variety of $\ell$-groups is generated by the symmetric $\ell$-group (of order-preserving permutations) on $\mathbb{Q}$ (or, alternatively $\mathbb{R}$). Since it turns out that $\m F_1 (\mathbb{Q})$ is actually precisely this symmetric $\ell$-group (due to the abundance of limit points in $\mathbb{Q}$ the only residuated and dually residuated maps of all orders are the order-preserving permutations), we see that we can take  $\m \Omega_1$ to be $\mathbb{Q}$, i.e.,  $\mathsf{LP_1}$ is generated by $\m F_1 (\mathbb{Q})$.
We prove, however, that $\mathsf{LP_n} \not =\mathsf{V}(\m F_n (\mathbb{Q}))$, for every $n>1$,
%Moving to the case $n=2$, we first explore whether we can take  $\m \Omega_2$ to also be $\mathbb{Q}$.However, $\mathsf{LP_2}$ is not generated by $\m F_2 (\mathbb{Q})$,  as the latter turns out to be also equal to the symmetric $\ell$-group on $\mathbb{Q}$ (see Lemma~\ref{}), which generates $\mathsf{LP_1}$, not $\mathsf{LP_2}$. 
%Therefore, we know that even though $\mathbb{Q}$ can serve as  $\m \Omega_1$, it cannot serve as all of the $\m \Omega_n$'s. 
leaving unsettled the question of whether there is a choice of  $\m \Omega_n$ for each $n$; or whether we can select a single/uniform chain $\m \Omega$ that will serve as $\m \Omega_n$ for all $n$, i.e.,  $\mathsf{LP_n}= \mathsf{V}(\m F_n (\m \Omega))$. Later, in Section~\ref{s: From F_N(Omega) to an N-Diagram}, we prove that such a uniform choice of a chain does exist and can be taken to be $\mathbb{Q} \overrightarrow{\times} \mathbb{Z}$, i.e.,   $\mathsf{LP_n}$ is generated by $\m F_n (\mathbb{Q} \overrightarrow{\times} \mathbb{Z})$, for all $n$ (generation theorem for $\mathsf{LP_n}$).
%We do this in two steps . % and~\ref{s: diagram to F_N}.

It follows from the wreath product representation $\m H \wr \m F_n (\mathbb{Z})$ that, even though $\m F_n(\mathbb{Z})$ does not generate the variety $\mathsf{LP_n}$, it plays an important role for $n$-periodic $\ell$-pregroups. It turns out that a good understanding of this algebra is needed for the goal of Section~\ref{s: From F_N(Omega) to an N-Diagram}, so 
Section~\ref{s: F_nZ} is devoted to $\m F_n(\mathbb{Z})$ and, among other things, to proving that its equational theory is decidable. In particular, we observe that each element $f$ of $\m F_n(\mathbb{Z})$ admits a useful decomposition $f=f^\circ \circ f^*$, where $f^\circ$ is an automorphism of $\mathbb{Z}$ ($f^\circ \in F_1(\mathbb{Z})$) and $f^*$ is a \emph{short} element of $\m F_n(\mathbb{Z})$. Furthermore, we consider diagrams (of \cite{GG}) that are $n$-periodic (a new notion) and show that they are suitable for studying $\m F_n(\mathbb{Z})$: we can go from a failure of an equation on $\m F_n(\mathbb{Z})$ to an $n$-periodic diagram and back. As the $n$-periodic diagrams obtained in this way can have unbounded \emph{heights} and a decision algorithm would need a known bound of this height to terminate, we focus on controlling the height of the automorphism part  $f^\circ$ of a function (the $f^*$ part is short by construction). We rephrase the height-control problem in the language or linear algebra and, working with linear systems, we prove that such an upper bound can indeed be computed, thus leading to decidability. 

In Section~\ref{s: From F_N(Omega) to an N-Diagram} we rely on the wreath product decomposition of the second representation theorem. In particular, each function in $\m F_n(\m \Omega)$, where $\m \Omega$ is an integral chain, has a global component and many local components, which are elements of $\m F_n(\mathbb{Z})$ and thus can be controlled by results in Section~\ref{s: F_nZ}. Since every integral chain has a lexicographic product structure  $\m J \overrightarrow{\times} \mathbb{Z}$, where $\m J$ is an arbitrary chain, we introduce a notion of diagram suited for $\m F_n(\m \Omega)$, called $n$-periodic partition diagram, that captures the natural partitioning induced by the lexicographic product on the diagram and has local parts that are $n$-periodic diagrams. We prove that a failure of an equation in $\m F_n(\m \Omega)$ yields a failure in an $n$-periodic partition diagram, we show that the latter can be taken to be short per local component, and we further  prove that any failure in an $n$-periodic partition diagram can be materialized in $\m F_n (\mathbb{Q} \overrightarrow{\times} \mathbb{Z})$. This way, we  obtain both the generation result $\mathsf{LP_n}=\mathsf{V}(\m F_n (\mathbb{Q} \overrightarrow{\times} \mathbb{Z}))$ and that the latter has a decidabile equational theory.

\section{Characterizing the elements of $\m F_n(\m{\Omega})$.}\label{s: elements of F_N} 

In this section, for every positive integer $n$,  we define the $n$-periodic $\ell$-pregroup of the form $\m F_n(\m \Omega)$ over a given chain $\m \Omega$. These algebras will play an important role in the representation and generation theorems for  $n$-periodic $\ell$-pregroups.
%For example, the generating algebra $\m F_n (\mathbb{Q} \overrightarrow{\times} \mathbb{Z})$ for the variety $\mathsf{LP_n}$ is of of this form.  $\m F_n(\m \Omega)$ is itself defined as a particular  subalgebra of $\m F(\m \Omega)$.

\subsection{The functional $\ell$-pregroup $\m F(\m \Omega)$.}\label{s: F(Omega)}

We first describe the functional/symmetric $\ell$-pregroup $\m F(\m \Omega)$ over a chain $\m \Omega$.
Given functions $f:\mathbf{P}\rightarrow\mathbf{Q}$ and $g:\mathbf{Q}\rightarrow\mathbf{P}$ between posets, we say that $g$ is a \emph{residual} for $f$, or that $f$ is a \emph{dual residual} for $g$, or that $(f,g)$ forms a \emph{residuated pair} if 
\[
f(a)\leq b\Leftrightarrow a\leq g(b)\text{, for all }a\in P,b\in Q.
\]

The residual and the dual residual of $f$ are unique when they exists\, and we denote them by $f^r$ and $f^{\ell}$, respectively; if they exist, they are given by
\[
f^{r}(b)=\max \{ a\in P :f(a)\leq b \}\quad \text{ and } \quad  f^{\ell}(a)=\min \{ b\in Q :a\leq f (b) \}.
\]
%Also, the dual residual of $f$ is unique, when it exists, and we denote it by $f^{\ell}$; in this case, we have
%\[
%f^{\ell}(a)=\min \{ b\in Q :a\leq f (b) \}.
%\]
%If $f$ has a residual, then $f$ is called \emph{residuated} and if it has a dual residual, it is called \emph{dually residuated}.

If $f$ has a residual, i.e., if $f^r$ exists,  then $f$ is called \emph{residuated} and if it has a dual residual, i.e., if $f^\ell$ exists, , it is called \emph{dually residuated}.  If the residual of $f^{r}$ exists, we denote it by $f^{rr}$ or $f^{(-2)}$ and we call it the \emph{second-order residual} of $f$. If the dual residual of $f^{\ell}$ exists, we denote it by $f^{\ell\ell}$ or $f^{(2)}$ and we call it the \emph{second-order dual residual} of $f$. More generally, $f^{r^n}$ or $f^{(-n)}$ is the \emph{$n$th-order residual} of $f$, if it exists, and $f^{\ell^n}$ or $f^{(n)}$ is the \emph{$n$th-order dual residual} of $f$, if it exists. 

Given a chain $\m \Omega$, we denote by $F(\m{\Omega})$ the set all maps on $\mathbf{\Omega}$ that have residuals and dual residuals of all orders. This set supports an $\ell$-pregroup $\m F(\m{\Omega})$, under composition, identity, the pointwise order, ${}^r$ and ${}^\ell$. Since the elements are functions with a chain a codomain and the order is pointwise,  $\m F(\m{\Omega})$ is a \emph{distributive} $\ell$-pregroup, i.e., its underlying lattice is distrubutive; we denote by $\mathsf{DLP}$ the variety of distrubutive $\ell$-pregroups. In \cite{GH} it is shown that every distrubutive $\ell$-pregroup can be embedded in $\m F(\m{\Omega})$  for some  chain $\m \Omega$.

 We write $a \prec b$, when $a$ is covered by $b$ (i.e., when $b$ is a cover of $a$) and also $a+1=b$ and $a=b-1$.
%Given a chain $\mathbf{\Omega}$, we say that a point $a\in\Omega$  \emph{has a $k$-cover} (or \emph{has a cover of order $k$}), if $k \in \mathbb{Z}^+$ and there exist  $a_{1},a_{2}\ldots,a_{k}\in\Omega $ with  $a \prec a_1 \prec \cdots \prec a_{k}$, or if $k \in \mathbb{Z}^-$ and there exist $a_{-1},a_{-2}\dots,a_{k}\in\Omega $ with  $a_k \prec \cdots \prec a_{-1} \prec  a$. Note that, since $\mathbf{\Omega}$ is a chain, these elements are unique.  In this case, we call $a_k$ the \emph{$k$-cover} of $a$ and denote it by $a + k$; we also define the \emph{$0$-cover} of $a$ to be $a + 0=0a = a$. An element in a chain $\m \Omega$ is said to be \emph{integral} if it has a $k$-cover for all $k \in \mathbb{Z}$; equivalently, it belongs to an interval of $\m \Omega$ that is isomorphic to the integers.
A chain $\mathbf{\Omega}$ is said to be \emph{integral} if %all of its elements are integral, i.e., if
%;  observe that, applying the axiom of choice, we can show that any integral chain 
it is isomorphic to the lexicographic product $\m{J} \overrightarrow{\times}\mathbb{Z}$ for some chain $\m 
 J$.  
 
 In \cite{GG} it is shown that every $\m F(\m{\Omega})$ embeds in $\m F(\overline{\m{\Omega}})$ for some integral chain $\overline{\m{\Omega}}$ containing $\m \Omega$. This results in an improved representation theorem for $\mathsf{DLP}$: every distrubutive $\ell$-pregroup can be embedded in $\m F(\m{\Omega})$  for some integral chain $\m \Omega$.

 \subsection{Functions in $\m F(\m \Omega)$, where $\m \Omega$ is integral}\label{s: elements of F}

We now take a closer look at the elements in $\m F(\m \Omega)$, where $\m \Omega$ is integral---in particular on how they behave under the application of iterated inverses.
 
% The first results are established for general elements of $\m{F(\m{\Omega})}$ and will be useful in detailing the properties of the elements of $\m{F_N(\m{\Omega})}$.

\begin{lemma}\textnormal{\cite{GG}} \label{l: ffellf}
     If $f$ is an order-preserving map on a chain $\mathbf{\Omega}$ with residual $f^r$ and dual residual $f^{\ell}$, and $a,b \in \Omega$, then:
    \begin{enumerate}
         \item $f^{\ell}ff^{\ell}=f^{\ell}$, $f^{r}ff^{r}=f^{r}$, 
        $ff^{\ell}f=f$ and $ff^{r}f=f$.
        \item We have $b<f(a)$ iff $f^r(b)<a$. Also, $f(b)<a$ iff $b<f^{\ell}(a)$.
    \end{enumerate}
\end{lemma}

 It is well known that a function $f:\mathbf{P}\rightarrow\mathbf{Q}$ is residuated iff
it is order-preserving 
and, for all $b \in Q$, the set $\{ a\in P :f(a)\leq b \} $ has a maximum in $\mathbf{Q}$; also, it is dually residuated iff
it is order-preserving 
and, for all $b \in Q$, the set $\{ b\in P :a\leq f (b) \}$ has a minimum. If a function is residuated, then it preserves existing joins, and if it is dually residuated, it preserves  existing meets.

We start by studying the effect of $f \mapsto f^{\ell \ell}$ on functions of $\m F(\m \Omega)$, where $\m \Omega$ is integral; this process is an $\ell$-pregroup automorphism (see \cite{GG}, for example).

\begin{lemma}\label{l: fellell=f+1}
Assume that $\m{\Omega}$ is an integral chain, $f:\Omega\rightarrow\Omega$ is an order-preserving map and $x\in\Omega$.
\begin{enumerate}
    \item If $f^{\ell\ell}$ exists, then $f^{\ell\ell}(x)=f(x-1)+1$.
    \item If $f^{rr}$ exists, then $f^{rr}(x)=f(x+1)-1$.
\end{enumerate}
\end{lemma}
\begin{proof}
   (1) We first show the claim for the case where $f^{\ell}f(x)=x$. %By the order preservation of $f$ we have $f(x-1)\leq f(x)$ and 
   Since $x=f^{\ell}f(x)=\min \{ b\in \Omega :f(x)\leq f (b) \}$, we get $x-1\notin\{ b\in \Omega :f(x)\leq f (b) \}$, hence $f(x)\nleq f(x-1)$; since $\m \Omega$ is a chain we get $f(x-1)<f(x)$.
   Therefore,  
    $f(x-1)<f(x-1)+1\leq f(x)$, which yields $f^{\ell}(f(x-1)+1)=x$, by the definition of $f^{\ell}$. On the other hand, Lemma~\ref{l: ffellf}(1) yields $ff^{\ell}f(x-1)=f(x-1)<f(x-1)+1$, so by Lemma~\ref{l: ffellf}(2) we get $f^{\ell}f(x-1)< f^{\ell}(f(x-1)+1)=x$. Therefore, $f^{\ell}(f(x-1))<x$ and $f^{\ell}(f(x-1)+1)\leq x$; taking into account that $f(x-1)\prec f(x-1)+1$ we obtain $f^{\ell\ell}(x)=\min \{ b\in \Omega :x\leq f^{\ell} (b) \}=f(x-1)+1$. %, proving the claim in this specific situation.
    
    Now, we consider the case $ f^{\ell}f(x)\neq x$. Since  $f^{\ell}f(x)\leq x$ by residuation, we get $f^{\ell}f(x)<x$. Also, since $f(x)<f(x)+1$,  Lemma~\ref{l: ffellf}(2) yields $x<f^{\ell}(f(x)+1)$; so $f^{\ell}f(x)<x<f^{\ell}(f(x)+1)$. Hence, by the definition of $f^{\ell\ell}$, we have $f^{\ell\ell}(x)=f(x)+1=f^{\ell\ell}(f^{\ell}(f(x)+1))$. 
    So, for $y:=f^{\ell}(f(x)+1)$ we have $f^{\ell\ell}(x)=f^{\ell\ell}(y)$.
    Consequently, using Lemma~\ref{l: ffellf}(1) we obtain $f^{\ell}f(y)=f^{\ell}ff^{\ell}(f(x)+1)=f^{\ell}(f(x)+1)=y$. In particular, the element $y$ satisfies the conditions of the previous case, so  $f^{\ell\ell}(x)=f^{\ell\ell}(y)=f(y-1)+1$. Below we prove that $f(y-1)=f(x-1)$, which  yields $f^{\ell\ell}(x)=f(y-1)+1=f(x-1)+1$, as desired.

   Since $y-1<y=f^{\ell}(f(x)+1)$,  Lemma~\ref{l: ffellf}(2) gives $f(y-1)<f(x)+1$, so $f(y-1)\leq f(x)$. %Moreover, since  $x\neq f^{\ell}f(x)$ and also $f^{\ell}f(x)\leq x$ by residuation, 
   As mentioned above, we have $f^{\ell}f(x)<x$, hence $f^{\ell}f(x)\leq x-1$. Therefore, $f(x)=ff^{\ell}f(x)\leq f(x-1)$ by using Lemma~\ref{l: ffellf}(1). In summary, we have $f(y-1)\leq f(x)\leq f(x-1)$. Since $x<y$, we also get the other inequality: $f(x-1)\leq f(y-1)$.     
    %On the one hand, since $x<y$, $x\leq y-1$, by order preservation $f(x)\leq f(y-1)$. On the other hand, by residuation, $f^{\ell}f(x)\leq x$ and since $x\neq f^{\ell}f(x)$, we have that $f^{\ell}f(x)<x$. Then, $f^{\ell}f(x)\leq x-1<x$ and applying $f$ and Lemma~\ref{l: ffellf}(1), we obtain $f(x)=ff^{\ell}f(x)\leq f(x-1)\leq f(x)$, so $f(x-1)=f(x)$. Since $y-1<y=f^{\ell}(f(x)+1)$, by Lemma~\ref{l: ffellf}(2), we know that $f(y-1)<f(x)+1$, meaning that $f(y-1)\leq f(x)$. Hence $f(x-1)=f(x)=f(y-1)$.
    Statement (2) is the dual of (1).
\end{proof}

Using induction and Lemma~\ref{l: fellell=f+1} we get the following formula.

\begin{lemma}\label{l:2N}
  If $\m{\Omega}$ is an integral chain, $f\in F(\m{\Omega})$ and $n\in\mathbb{Z}$, then $f^{(2n)}(x)=f(x-n)+n$ for all $x\in\Omega$.   
\end{lemma}
%\begin{proof}    We will show the claim for $n\geq 0$ (by induction), as the other case is dual. The base case is trivial since $f^{(2\cdot 0)}(x)=f(x)=f(x-0)+0$. We assume that the claim is true for $n$ and we will show it holds for $n+1$. Using Lemma~\ref{l: fellell=f+1} and the inductive hypothesis we get $f^{(2n+2)}(x)=(f^{(2n)})^{\ell\ell}(x)=f^{2n}(x-1)+1=[f((x-1)-n)+n]+1=f(x-(n+1))+(n+1)$.    \end{proof}

As usual, we denote the inverse image of a set $X\subseteq\Omega$ via a function $f$ by $f^{-1}[X]$; we write $f^{-1}[a]$ for $f^{-1}[\{a\}]$.

%The following lemma can be found in \cite{GG} and it will be useful to characterize $\m F_N(\m\Omega)$.

\begin{lemma}\textnormal{\cite{GG}} \label{l: bounded preimage}
Given a chain $\mathbf{\Omega}$, a map $f$  on $\Omega$ is residuated and dually residuated iff $f$ is order-preserving and for all $a\in \Omega$: 
\begin{enumerate}
    \item If $a\in f[\Omega]$, then $f^{-1}[ a]=[b,c]$, for some $b \leq c$ in $\Omega$.
   \item  If $a\notin f[\Omega]$, there exists $b,c\in\Omega$ such that, $b\prec c$ and $ \ensuremath{a\in(f(b),f(c))}$.
\end{enumerate}
In the case (1), $f^\ell(a)=b$ and $f^r(a)=c$; in the case (2), $f^\ell(a)=c$ and $f^r(a)=b$.
Moreover, if $\m \Omega$ is integral then this is further equivalent to $f \in F(\m \Omega)$.
\end{lemma}

A simple characterization of the elements of $\m{F(\mathbb{Z})}$ is given by the next lemma.

\begin{lemma}\textnormal{\cite{GG}}\label{l:F(Z)}
 If $f$ is a function on $\mathbb{Z}$, then $f\in F(\mathbb{Z})$ iff $f$ is order-preserving and  $f^{-1}[a]$ is a bounded interval for all $a\in f[\mathbb{Z}]$ iff $f$ is an order-preserving and $f^{-1}[a]$ is a finite set for all $a\in f[\mathbb{Z}]$.
 %   If $f:\mathbb{Z}\rightarrow\mathbb{Z}$ is an order-preserving function and, for all $a\in f[\mathbb{Z}]$, $f^{-1}[a]$ is bounded, then $f\in F(\mathbb{Z})$ 
\end{lemma}

 \subsection{The $n$-periodic subalgebra $\m F_n(\m \Omega)$ of $\m F(\m \Omega)$.}\label{s: FnOmega}

  \medskip
In the language of $\ell$-pregroups, we define $x^{(n)}$ for all $n\in\mathbb{Z}$, by $x^{(0)}:=x$, $x^{(n+1)}:=(x^{(n)})^{\ell}$ if $n>0$ and $x^{(n-1)}:=(x^{(n)})^{r}$ if $n<0$. Note that this notation agrees with the notation for $f^{(n)}$, when this function exists. 
Given $n \in \mathbb{Z}^+$, we say that an element $x$ is $n$-periodic if $x^{\ell^n}=x^{r^n}$, i.e., $x^{(n)}=x^{(-n)}$, or equivalently $x^{(2n)}=x$; an  $\ell$-pregroup is \emph{$n$-periodic} if all of its elements are, and it is called \emph{periodic} if it is $n$-periodic for some $n$.
 In \cite{GJ} it is shown that every periodic $\ell$-pregroup is in fact distributive.

%The importance of $\mathbf{F(\m{\Omega})}$ in the theory of distributive $\ell$-pregroups  has been established in the Cayley/Holland style embedding theorem proved on \cite{GH}, where it is shown that distributive $\ell$-pregroups can be embedded into $\mathbf{F(\m{\Omega})}$ for some chain $\mathbf{\Omega}$. Moreover, we can restrict to $\mathbf{F(\m{\Omega})}$ where $\mathbf{\Omega}$ is integral (i.e. a chain in which every element is contained in an interval of isomorphic to $\mathbb{Z}$), as shown in \cite{GG}.

%We refer to the maps $^\ell$ and $^r$ as the \emph{left inverse} and \emph{right inverse} operations. 
For every chain $\m \Omega$, $F_n(\m \Omega)$ denotes the set of all $n$-periodic elements of $\m F(\m \Omega)$; the following lemma shows that this set supports a subalgebra $\m F_n(\m \Omega)$  of $\m F(\m \Omega)$.

\begin{lemma} \label{l: F_N(Omega)}
Let $n$ be a positive integer.
\begin{enumerate} 
\item %For every $\ell$-pregroup $\m A$, the $N$-periodic elements of $\m A$ form a $N$-periodic subalgebra of $\m A$.
The $n$-periodic elements of any $\ell$-pregroup form a $n$-periodic subalgebra.
\item For every chain $\m \Omega$, $\m F_n(\m \Omega)$ is an $n$-periodic subalgebra of $\m F(\m \Omega)$.
\end{enumerate}
\end{lemma}

\begin{proof}
(1) As mentioned in the introduction, $\ell$-pregroups are involutive residuated lattices, so they satisfy involutivity $x^{\ell r}=x=x^{r \ell}$, the De Morgan laws 
\begin{center}
 $(x \jn y)^\ell=x^\ell \mt y^\ell$,  $(x \mt y)^\ell=x^\ell \jn y^\ell$,
$(x \jn y)^r=x^r \mt y^r$,  $(x \mt y)^r=x^r \jn y^r$,   
\end{center}
 and, since $x+y=xy$, the inverses are monoid antihomomorphisms 
 \begin{center}
     $(xy)^\ell =y^\ell x^\ell$,  $(xy)^r =y^r x^r$, and $1^\ell=1^r=1$; 
 \end{center}
 see \cite{GJ}.
 Therefore, the map $x \mapsto x^{\ell \ell}$ is an endomorphism of the $\ell$-pregroup $\m L$ (actually an automorphism), hence so are all of its (positive and negative) powers. In particular, $x \mapsto x^{(2n)}$ is an endomorphism, hence the set of its fixed points is a subalgebra (for example, if $x^{(2n)}=x$ and $y^{(2n)}=y$, then $(xy)^{(2n)}=x^{(2n)}y^{(2n)}=xy$). Clearly, an element of $\m L$ is $n$-periodic iff it is a fixed point of $x \mapsto x^{(2n)}$.  As a result, the set of all $n$-periodic elements of $\m L$ forms a subalgebra of $\m L$ and all of its elements are $n$-periodic.
 (2) follows by applying (1) to $\m F(\m \Omega)$. 
 \end{proof}

We now establish our first representation theorem for  $n$-periodic $\ell$-pregroups, which will be useful in obtaining the generation and decidability results for $n$-periodic $\ell$-pregroups; the second representation theorem will be Theorem~\ref{t: repnper}.

\begin{theorem} \label{t: emb to F_N(Omega)}
Given  $n \in \mathbb{Z}^+$, every $n$-periodic $\ell$-pregroup can be embedded in  $\m F_n(\m \Omega)$, for some integral chain $\m{\Omega}$, i.e., in  $\m F_n(\m{J}\overrightarrow{\times}\mathbb{Z})$, for some  chain $\m{J}$.
%\item Every $N$-periodic $\ell$-pregroup can be embedded in  $\m F_N(\m{J}\overrightarrow{\times}\mathbb{Z})$, for some  chain $\m{J}$.
\end{theorem}

\begin{proof}
Let $\m L$ be an $n$-periodic $\ell$-pregroup. By \cite{GJ} $\m L$ is a distributive, so by \cite{GG} $\m L$ embeds in $\m F(\m \Omega)$, for some integral chain $\m{\Omega}$. Since $\m L$ satisfies the equation $x^{(2n)}=x$, the image of the embedding also satisfies the equation; so the image is contained in $\m F_n(\m \Omega)$. The last part of the theorem follows from the fact that every integral chain is isomorphic to $\m{J} \overrightarrow{\times}\mathbb{Z}$ for some chain $\m 
 J$. 
\end{proof}

 We will dedicate the remainder of this subsection to characterizing  the functions in $\m{F}_n(\m{\Omega})$, where $\m{\Omega}$ is an integral chain.

%  \begin{lemma}\label{l:N-periodic}%\marginpar{N: used only in next lemma?}
%    If $\mathbf{\Omega}$ is an integral chain, $f\in F(\m\Omega)$ and $n \in \mathbb{Z}^+$, each of the following  is equivalent to $f\in F_n(\m\Omega)$.
%    \begin{enumerate}
        %\item .
%        \item For all $x\in\Omega$, $f(x)=f(x-n)+n$. 
%        \item For all $x\in\Omega$  and $k\in\mathbb{Z}$, $f(x)=f(x-kn)+kn$. 
       %  \item For all $x,y\in\Omega$:  
      %       \begin{enumerate}
     %            \item         $x\leq y+ n \Rightarrow f(x)\leq f(y)+n$ and 
    %             \item $x+ n\leq y \Rightarrow f(x)+n\leq f(y)$.
   %         \end{enumerate}
  %      \item For all $x,y\in\Omega$ and $k\in\mathbb{Z}$:  $x\leq y+ kn \Rightarrow f(x)\leq f(y)+kn$. 
        %\item For all $x,y\in\Omega$ and $k\in\mathbb{Z}$, if $\overline{y}=y+kn$ such that $\overline{y}\in [x,x+n)$, then $f(x)\leq f(y)+kn$
 %   \end{enumerate}
%  \end{lemma}

 \begin{lemma}\label{l:N-periodic}%\marginpar{N: used only in next lemma?}
    If $\mathbf{\Omega}$ is an integral chain, $f\in F(\m\Omega)$ and $n \in \mathbb{Z}^+$, each of the following are equivalent to $f\in F_n(\m\Omega)$:
    \begin{enumerate}
       % \item $f\in F_n(\m\Omega)$.
        \item For all $x\in\Omega$, $f(x)=f(x-n)+n$. 
        \item For all $x\in\Omega$  and $k\in\mathbb{Z}$, $f(x)=f(x-kn)+kn$. 
         \item For all $x,y\in\Omega$:  
             \begin{enumerate}
                 \item         $x\leq y+ n \Rightarrow f(x)\leq f(y)+n$ and 
                 \item $x+ n\leq y \Rightarrow f(x)+n\leq f(y)$.
            \end{enumerate}
        \item For all $x,y\in\Omega$ and $k\in\mathbb{Z}$:  $x\leq y+ kn \Rightarrow f(x)\leq f(y)+kn$. 
        %\item For all $x,y\in\Omega$ and $k\in\mathbb{Z}$, if $\overline{y}=y+kn$ such that $\overline{y}\in [x,x+n)$, then $f(x)\leq f(y)+kn$
    \end{enumerate}
  \end{lemma}
  \begin{proof}
       The equivalence to (1) follows from $n$-periodicity and Lemma~\ref{l:2N}. Also, $(1)\Rightarrow(2)$ follows by induction and the converse is obvious.  $(4)\Rightarrow(3)$ is obvious, as well.
  
      For $(2)\Rightarrow(4)$, if  $x,y\in\Omega$, $k\in\mathbb{Z}$ and $x\leq y +kn$, then by order preservation and (2) we have $f(x)\leq f(y+kn)=f(y+kn-kn)+kn=%f^{(-2kn)}(y)+kn
      f(y)+kn$.
    
      For $(3)\Rightarrow(1)$, if $x\in \Omega$, then $x\leq(x-n)+n$, so  $f(x)\leq f(x-n)+n$ by (3a). Since $x-n+n\leq x$, (3b) gives $f(x-n)+n\leq f(x)$; so $f(x-n)+n= f(x)$. %By Lemma~\ref{l:2N} we get  $f^{(2n)}(x)=f(x-n)+n=f(x)$.  Therefore $f\in F_n(\m{\Omega})$.
     % $(2)\Rightarrow(3)$ Suppose that $x,y,k\in\mathbb{Z}$, where $\overline{y}=y+kn$ is such that $\overline{y}\in [x,x+n)$, then $x\leq y+kn$. So, $x\leq y+kn$ and by (2), then $f(x)\leq f(y)+kn$.
      %$(3)\Rightarrow(2)$ Suppose that $x,y,k\in\mathbb{Z}$, and $x\leq y+ kn$; consider $k_0\in\mathbb{Z}$ the smallest integer satisfying the condition $x-y\leq k_0n$, then $\overline{y}=y+k_0 n\in[x,x+n)$ and by (3), $f(x)\leq f(y)+k_0 n\leq f(y)+kn$.
\end{proof}

%  \begin{proof}
%    \CHAT{I}{The items and the proof are out of order, was 1 before f is periodic? I think the proof could be organized easier that way}
%       $(1)\Rightarrow(2)$ follows from $n$-periodicity and Lemma~\ref{l:2N}. Also, $(2)\Rightarrow(3)$ follows by induction and the converse is obvious.  $(5)\Rightarrow(4)$ is obvious, as well.
  
%      For $(3)\Rightarrow(5)$, if  $x,y\in\Omega$, $k\in\mathbb{Z}$ and $x\leq y +kn$, then by order preservation and (3) we have $f(x)\leq f(y+kn)=f(y+kn-kn)+kn=%f^{(-2kn)}(y)+kn f(y)+kn$.
    
%      For $(4)\Rightarrow(2)$, if $x\in \Omega$, then $x\leq(x-n)+n$, so  $f(x)\leq f(x-n)+n$ by (4a). Since $x-n+n\leq x$, (4b) gives $f(x-n)+n\leq f(x)$; so $f(x-n)+n= f(x)$. %By Lemma~\ref{l:2N} we get  $f^{(2n)}(x)=f(x-n)+n=f(x)$.  Therefore $f\in F_n(\m{\Omega})$.
     % $(2)\Rightarrow(3)$ Suppose that $x,y,k\in\mathbb{Z}$, where $\overline{y}=y+kn$ is such that $\overline{y}\in [x,x+n)$, then $x\leq y+kn$. So, $x\leq y+kn$ and by (2), then $f(x)\leq f(y)+kn$.
      %$(3)\Rightarrow(2)$ Suppose that $x,y,k\in\mathbb{Z}$, and $x\leq y+ kn$; consider $k_0\in\mathbb{Z}$ the smallest integer satisfying the condition $x-y\leq k_0n$, then $\overline{y}=y+k_0 n\in[x,x+n)$ and by (3), $f(x)\leq f(y)+k_0 n\leq f(y)+kn$.
%\end{proof}

In particular, for $\m F_{n}(\mathbb{Z})$ we make use of the characterization for $\m F(\mathbb{Z})$, as well.

\begin{lemma}\label{l: F_N(Z)}
    For all $f:\mathbb{Z}\rightarrow\mathbb{Z}$, and $n \in \mathbb{Z}^+$, we have  $f\in F_{n}(\mathbb{Z})$ iff
    \begin{enumerate}
        \item for all $a\in f[\mathbb{Z}]$, $f^{-1}[a]$ is finite  (equivalently, a bounded interval) and,
        \item for all $x,y,k\in\mathbb{Z}$: $x\leq y+kn \Rightarrow f(x)\leq f(y)+kn$.
  \end{enumerate}
\end{lemma}  
\begin{proof}
    In the forward direction, observe that by Lemma~\ref{l: bounded preimage}, we know that for all $a\in f[\mathbb{Z}]$, there exist $b,c\in\mathbb{Z}$ such that $f^{-1}[a]=[b,c]$, so $f^{-1}[a]$ is bounded. Also, by Lemma~\ref{l:N-periodic}, for all  $x,y,k\in\mathbb{Z}$, if $x\leq y+kn$, then $f(x)\leq f(y)+kn$.
    
    For the converse direction, %assume that $f$ satisfies (1) and (2). Observe that
    if $x,y\in\mathbb{Z}$ and $x\leq y$, by (2) we have $f(x)\leq f(y)$, so $f$ is order preserving. Since (1) also holds, Lemma~\ref{l:F(Z)} implies  $f\in F(\mathbb{Z})$; by (2) and Lemma~\ref{l:N-periodic}, we get $f\in F_n(\mathbb{Z})$
 \end{proof}

 \begin{lemma}\label{l: simple}
For every $n$, the $\ell$-pregroup $\m F_n(\mathbb{Z})$ is simple.
\end{lemma}

\begin{proof}
The invertible elements of $\m F_n(\mathbb{Z})$ are the order-preserving bijections on $\mathbb{Z}$, so they are precisely the translations $t_k: x \mapsto x-k$, for $k \in \mathbb{Z}$. 
Using the results in \cite{GJKO} about the correspondence between congruences and convex normal submonoids of the negative cone,  to prove that  $\m F_n(\mathbb{Z})$ is simple it suffices to prove that: if $M$ is a non-trivial convex normal submonoid of the negative cone of $\m F_n(\mathbb{Z})$ then $M$ is the whole negative cone.

Given a strictly negative $f \in M$, the element
$g:=f \mt f^{\ell \ell} \mt \cdots \mt  f^{\ell^{2n-2}}$ is invertible, since
$(f \mt f^{\ell \ell} \mt \cdots \mt   f^{\ell^{2n-4}} \mt   f^{\ell^{2n-2}})^{\ell\ell}=f^{\ell\ell} \mt  f^{\ell^4} \mt  \cdots \mt   f^{\ell^{2n-2}} \mt   f^{\ell^{2n}}
=f^{\ell\ell} \mt  f^{\ell^4} \mt  \cdots \mt   f^{\ell^{2n-2}} \mt   f$,
and $g \leq f<1$ so $g$ is a strictly decreasing translation and $M$  contains $g$ and all of its powers; therefore $M$ contains translations $t_k$ for arbitrarily large $k$. We now argue that every negative element $h$ of $F_n(\mathbb{Z})$ is above some power of $g$ and thus belongs to $M$. Since $h$ is negative, for all $m \in \mathbb{Z}$ there exists $k_m \geq 0$ such that $h(m)=m-k_m$. Since $h$ is $n$-periodic, there are only $n$-many distinct such $k_m$'s; let $k'$ be the largest. We have $t_k \leq h$ for $k \geq k'$, so $h \in M$.
\end{proof}

The following two lemmas provide an even more transparent characterization of the elements of $\m F_n(\m{J}\overrightarrow{\times}\mathbb{Z})$ where  $\m{J}$ is a chain. 
Note that if $f\in F_n(\m\Omega)$ then not only $f^{(2n)}=f$, but moreover $f^{(2kn)}=f$ for all $k\in\mathbb{Z}$. %As usual, we denote by $\pi_1$ and $\pi_2$ the first and second projection maps. 
For $x,y \in J \times \mathbb{Z}$, we write $x \equiv y$ iff they have the same first coordinate,
%$\pi_1(x)=\pi_1(y)$,
i.e., if there exists $k \in \mathbb{Z}$ such that $x=y+k$ ($x$ and $y$ belong to the same component of the lexicographic product). The following lemma shows that maps in $\m F_n(\m{J} \overrightarrow{\times}\mathbb{Z})$  preserve this equivalence relation.

\begin{lemma}\label{l:pre f^tilde}
    If $\m J$ is a chain, $\m{\Omega}=\m J \overrightarrow{\times}\mathbb{Z}$, $n \in \mathbb{Z}$, $f\in F_n(\m{\Omega})$,  $j,i\in J$ and $m\in\mathbb{Z}$, then the following hold.
    \begin{enumerate}
        \item For all $x,y \in \Omega$,  $x \equiv y \Rightarrow f(x)\equiv f(y)$.\\
        In other words, if $f(j,0)=(i,m)$, then for all $r\in\mathbb{Z}$, there exists $k\in\mathbb{Z}$, such that $f(j,r)=(i,k)$.
        \item If $x \notin f[\Omega]$, then there exists $b\in\Omega$ such that $f(b)<x<f(b+1)$.\\ 
        In other words, if $(i,m)\notin f[\Omega]$, then there exist $b\in\Omega$, and $r,k\in \mathbb{Z}$ such that $f(b)=(i,r)$, $f(b+1)=(i,k)$ and  $m\in (r,k)$.
        %\textcolor{red}{if $(i,m)\notin f[\Omega]$, then there exist $b\in\Omega$, $j\in J$, and $n,k\in \mathbb{Z}$ such that $f(b)=(j,n)$, $f(b+1)=(j,k)$ and  $m\in (n,k)$.}
    \end{enumerate}
   % \CHAT{N}{This lemma and the next one use $N$ instead of $n$. Since $n$ is also used, we would need to find a different letter to replace the $n$ that is used.}
\end{lemma}
\begin{proof}
          (1) Suppose $f(j,0)=(i,m)$ and let $r \leq 0$; the proof for $r>0$ is dual. We have  $f(j,r)=(i',k)$, where $i'\in J$ and $k\in\mathbb{Z}$; we will show that $i'=i$. Since $(j,r)\leq (j,0)$, we get $(i',k)=f(j,r)\leq f(j,0)=(i,m)$ by the order-preservation of $f$, so $i'\leq i$. 
          Using the order-preservation of $f$, Lemma~\ref{l:2N} and the $n$-periodicity of $f$, we get $(i,m)= f(j,0)\leq f(j,r-rn)=f^{(2rn)}((j,r-rn)+rn)-rn=f^{(2rn)}(j,r)-rn=f(j,r)-rn=(i',k)-rn=(i',k-rn)$, which implies $i\leq i'$; therefore $i=i'$.%\CHAT{N}{I just wanted to make sure that we really want to write $r-rn$ and this is not because we change the letters we use for the variables.} 
          % Now observe that $0\leq n-nN$, so $(j,0)\leq (j,n-nN)$, then again by order-preservation of $f$, $ f(j,0)\leq f(j,n-nN)$. Hence, applying Lemma~\ref{l:2N}, we obtain $ f(j,0)\leq f(j,n-nN)=f^{(2nN)}((j,n-nN)+nN)-nN=f^{(2nN)}(j,n)-nN$ and given that $f$ is $N$-periodic, $f(j,0)\leq f(j,n-nN)=f^{(2nN)}(j,n)-nN=f(j,n)-nN$, so $(i,m)\leq (i',k)-nN=(i',k-nN)$, meaning that $i\leq i'$, therefore $i=i'$. Then $f(j,n)=(i,k)$.
          
          (2) If $(i,m)\notin f[\Omega]$ then, since $f\in F(\m\Omega)$, by Lemma~\ref{l: bounded preimage} there exist $b,c\in \Omega$ with $b\prec c$ such that $(i,m)\in (f(b),f(c))$. Then, by (1), there exist $j\in J$ and $r,k\in\mathbb{Z}$ such that $f(b)=(j,r)$ and $f(c)=f(b+1)=(j,k)$; so $(i,m)\in ((j,r),(j,k))$. Therefore, $i=j$ and $m\in (r,k)$.
\end{proof}

The following theorem shows that maps in $\m{F}_n(\m{J} \overrightarrow{\times}\mathbb{Z})$ consist of a global component map on $\m J$ and many local component maps, one for each $j \in J$. The ability to work on the global level and on the local levels separately will be crucial for the results in the next sections.
%The lemma will be useful in Section~\ref{s: diagram to F_N}, and it will provide relevant insight into the structure of the elements of $\m{F_N(\m{\Omega})}$ for an integral chain $\m{\Omega}$.

\begin{theorem}\label{t:F(JxZ)}
    Given a chain $\mathbf{J}$, then %$\mathbf{\Omega}=\m{J} \overrightarrow{\times}\mathbb{Z}$, 
    $f\in F_n(\m{J} \overrightarrow{\times}\mathbb{Z})$ if and only if there exists an order-preserving bijection $\widetilde{f}:J\rightarrow J$ and maps $\overline{f}_j\in F_n(\mathbb{Z})$, for all $j\in J$, such that $f(j,r)=(\widetilde{f}(j),\overline{f}_{j}(r))$
  for all $(j,r)\in J \times\mathbb{Z}$. 
  \end{theorem}
  \begin{proof}
      We first show that if $\widetilde{f}:J\rightarrow J$ is an order-preserving bijection and  $\overline{f}_j\in F_n(\mathbb{Z})$, for all $j\in J$, then $f\in F(\m{J} \overrightarrow{\times}\mathbb{Z})$, where $f(j,r)=(\widetilde{f}(j),\overline{f}_{j}(r))$ for all $(j,r)\in J \times\mathbb{Z}$, by verifying the conditions of Lemma~\ref{l: bounded preimage}; we set $\mathbf{\Omega}:=\m{J} \overrightarrow{\times}\mathbb{Z}$  for brevity. If $(j_1,r_1)\leq(j_2,r_2)$ in $\m \Omega$, then $j_1<j_2$ or ($j_1=j_2$ and $r_1\leq r_2$). In the first case, since $\widetilde{f}$ is an order preserving bijection, we have $\widetilde{f}(j_1)<\widetilde{f}(j_2)$ and in the second case, $\widetilde{f}(j_1)=\widetilde{f}(j_2)$ and  $\overline{f}_{j_1}(r_1)\leq \overline{f}_{j_1}(r_2)=\overline{f}_{j_2}(r_2)$, since $\overline{f}_{j_1}$ is order preserving. Hence, in both cases $f(j_1,r_1)=(\widetilde{f}(j_1),\overline{f}_{j_1}(r_1))\leq (\widetilde{f}(j_2),\overline{f}_{j_2}(r_2))=f(j_2,r_2)$, i.e., $f$ is order-preserving. 
      If $(i,m)\in f[\Omega]$, then $(i,m)=f(j,\ell)=(\widetilde{f}(j), \overline{f}_j(\ell))$, for some $(j,\ell)\in J \times \mathbb{Z}$. 
      Actually, since $\widetilde{f}$ is a bijection, $j\in J$ is unique with the property $\widetilde{f}(j)=i$. Also, since $m\in \overline{f}_{j}[\mathbb{Z}]$ and  $\overline{f}_{j}\in F(\mathbb{Z})$, by Lemma~\ref{l: bounded preimage} there exist $r,k\in\mathbb{Z}$ such that $(\overline{f}_{j})^{-1}[m]=[r,k]$. Therefore, $f^{-1}[(i,m)]=[(j,r),(j,k)]$.
      %If $a\in f^{-1}[\Omega]$, then since $\widetilde{f}$ is a bijection, there exist a unique $j\in J$ such that $\widetilde{f}(j)=\pi_1(a)$. Now since $\pi_2(a)\in \overline{f}_{j}[\mathbb{Z}]$, and since $\overline{f}_{j}\in F(\mathbb{Z})$, by Lemma~\ref{l: bounded preimage}, there exists $n,k\in\mathbb{Z}$ such that $(\overline{f}_{j})^{-1}[\pi_2(a)]=[n,k]$. Therefore, $f^{-1}[a]=[(j,n),(j,k)]$. 
        If $(i,m) \not\in f[\Omega]$, since $\widetilde{f}$ is a bijection, there exist $j\in J$ such that $\widetilde{f}(j)=i$, hence $m\not \in (\overline{f}_{j})^{-1}[\mathbb{Z}]$. Since $\overline{f}_{j}\in F(\mathbb{Z})$, by Lemma~\ref{l: bounded preimage} there exist $r,k\in\mathbb{Z}$ such that $r\prec k$ and $m\in (\overline{f}_{j}(r),\overline{f}_{j}(k))$. Hence, $(j,r)\prec(j,k)$ and $(i,m)\in (f(j,r),f(j,k))$. 
      %  Similarly, if $a\in f^{-1}[\Omega]$, since $\widetilde{f}$ is a bijection, there exist $j\in J$ such that $\widetilde{f}(j)=\pi_1(a)$. Now since $\pi_2(a)\in (\overline{f}_{j})^{-1}[\mathbb{Z}]$, and since $\overline{f}_{j}\in F(\mathbb{Z})$, by Lemma~\ref{l: bounded preimage}, there exists $n,k\in\mathbb{Z}$ such that $n\prec k$ and $\pi_2(a)\in (\overline{f}_{j}(n),\overline{f}_{j}(k))$. Hence, $(j,n)\prec(j,k)$ and $a\in (f(j,n),f(j,k))$. 
      Having shown that $f\in F(\m{\Omega})$, to prove
      $f\in F_n(\m\Omega)$ we verify that $f$ is $n$-periodic. For all $(j,r)\in\Omega$, by applying Lemma~\ref{l:2N} to $f$ and to $\overline{f}_j$ and using $\overline{f}_{j}\in F_n(\mathbb{Z})$, we get  $f^{(2n)}(j,r)=f((j,r)-n)+n=f(j,r-n)+n=
      (\widetilde{f}(j),\overline{f}_{j}(r-n))+n=
      (\widetilde{f}(j),\overline{f}_{j}(r-n)+n)=(\widetilde{f}(j),\overline{f}_{j}^{(2n)}(r))=(\widetilde{f}(j),\overline{f}_{j}(r))=f(j,r)$. 
      %suppose that $a=(j,n)\in\Omega$, then by Lemma~\ref{l:2n}, $f^{(2N)}(a)=f(a-N)+N=(\widetilde{f}(j),\overline{f}_{j}(n-N)+N)$, and applying again Lemma~\ref{l:2N}, $f^{(2N)}(a)=(\widetilde{f}(j),\overline{f}_{j}(n-N)+N)=(\widetilde{f}(j),\overline{f}_{j}^{(2N)}(n))$. And since $\overline{f}_{j}\in F_N(\mathbb{Z})$, we have that $f^{(2N)}(a)=(\widetilde{f}(j),\overline{f}_{j}^{(2N)}(n))=(\widetilde{f}(j),\overline{f}_{j}(n))=f(a)$. Thus, $f\in F_N(\m\Omega)$. 

      Conversely, if $f\in F_n(\m{\Omega})$, we define the function $\widetilde{f}:J\rightarrow J$ by $\widetilde{f}(j)=i$ iff $f(j,0)=(i,m)$ for some $m\in \mathbb{Z}$; by Lemma~\ref{l:pre f^tilde}(1) $\widetilde{f}$ is well defined. To show that $\widetilde{f}$ is one-to-one, we assume that $\widetilde{f}(j_1)=\widetilde{f}(j_2)$ for some $j_1,j_2\in J$; since $\m{J}$ is a chain, we may further assume that $j_1\leq j_2$. Then, there exist $r_1,r_2\in\mathbb{Z}$ such that $f(j_1,0)=(\widetilde{f}(j_1),r_1)$ and $f(j_2,0)=(\widetilde{f}(j_1),r_2)$; since $j_1\leq j_2$, by the order-preservation of $f$ we have $(\widetilde{f}(j_1),r_1)=f(j_1,0)\leq f(j_2,0)=(\widetilde{f}(j_2),r_2)=(\widetilde{f}(j_1),r_2)$, so $r_1\leq r_2$. Let $k\in\mathbb{Z}$ be such that $r_2<r_1+kn$. 
      Using this fact, Lemma~\ref{l:2N} and the periodicity of $f$, we have 
      $f(j_2,0)=(\widetilde{f}(j_1),r_2)<
      (\widetilde{f}(j_1),r_1+kn)=
      (\widetilde{f}(j_1),r_1)+kn
      =f(j_1,0)+kn
       =f(j_1,kn-kn)+kn
      =f((j_1,kn)-kn)+kn
      =f^{(2kn)}(j_1,kn)
      =f(j_1,kn)$; hence $f(j_1,kn)\nleq f(j_2,0)$.
      By order preservation we get $(j_1,kn)\not\leq(j_2,0)$ and since $\Omega$ is a chain we obtain $(j_2,0)<(j_1,kn)$; thus $j_2\leq j_1$. In conclusion $j_1=j_2$.
      %By periodicity of $f$, $f(j_1,kN)=f^{(2kN)}(j_1,kN)$ and by Lemma~\ref{l:2N}  $f(j_1,kN)=f^{(2kN)}(j_1,kN)=f((j_1,kN)-kN)+kN=f(j_1,0)+kN$. So $f(j_2,0)=(\widetilde{f}(j_1),n_2)<(\widetilde{f}(j_1),n_1)+kN=f(j_1,kN)$, so $f(j_1,kN)\nleq f(j_2,0)$ and then by order preservation $(j_1,kN)\not\leq(j_2,0)$ and since $\Omega$ is a chain, we must have  $(j_2,0)<(j_1,kN)$, then $j_2\leq j_1$. So $j_1=j_2$, meaning that $\widetilde{f}$ is one-to-one.

      To show that $\widetilde{f}$ is onto, let $i\in J$ and consider $(i,0)\in  \Omega$. If $(i,0)\in f[\Omega]$, there exists $(j,r)\in\Omega$ such that $f(j,r)=(i,0)$; by Lemma \ref{l:pre f^tilde}(1) we have $f(j,0)=(i,k)$ for some $k\in\mathbb{Z}$, so $i\in \widetilde{f}[J]$. If $(i,0)\notin f[ \Omega]$, by Lemma~\ref{l:pre f^tilde}(2) there exist $j\in J$ and $r,m_1,m_2\in \mathbb{Z}$, such that $f(j,r)=(i,m_1)$, $f(j,r+1)=(i,m_2)$, and $0\in (m_1,m_2$); then, by Lemma~\ref{l:pre f^tilde}(1), $f(j,0)=(i,m_3)$ for some $m_3\in\mathbb{Z}$, so $\widetilde{f}(j)=i$, i.e., $i\in\widetilde{f}[J]$.

      Now, for $j \in J$, we define $\overline{f}_{j}:\mathbb{Z}\rightarrow\mathbb{Z}$ where $\overline{f}_{j}(r)=m$ iff $f(j,r)=(\widetilde{f}(j),m)$ for some $m\in \mathbb{Z}$; by Lemma~\ref{l:pre f^tilde}(1) $\overline{f}_i$ is well-defined.
      %  Let us concentrate now on the second coordinate of $f$. By Lemma~\ref{l:pre f^tilde}~(1) for all $(j,n)\in\Omega$, $f(j,n)=(\widetilde{f}(j),m)$ for some $m\in\mathbb{Z}$, then define $\overline{f}_{j}:\mathbb{Z}\rightarrow\mathbb{Z}$ where $\overline{f}_{j}(n)=m$ iff $f(j,n)=(\widetilde{f}(j),m)$ for some $m\in \mathbb{Z}$. By our previous observation, $\overline{f}_i$ is well-defined. 
      We will show that $\overline{f}_{j}\in F(\mathbb{Z})$ by verifying the condition in Lemma~\ref{l:F(Z)}.
      %Let us show that, for all $j\in J$, $\overline{f}_{j}$ is order-preserving. Suppose that $n_1,n_2\in\mathbb{Z}$ and $n_1\leq n_2$,  with $f(j,n_1)=(\widetilde{f}(j),m_1)$ and $f(j,n_2)=(\widetilde{f}(j),m_2)$ for some $m_1,m_2\in\mathbb{Z}$. Then, by order-preservation of $f$, we have that $f(j,n_1)\leq f(j,n_2)$, so $(\widetilde{f}(j),m_1)\leq (\widetilde{f}(j),m_2)$, meaning that $m_1\leq m_2$. Hence $\overline{f}_{j}(n_1)=m_1\leq m_2=\overline{f}_{j}(n_2)$. Then $\overline{f}_{j}$ is order-preserving. We want to show now that, $\overline{f}_{j}\in F(\m{\Omega})$, in order to do so, we will apply Lemma~\ref{l:F(Z)}.     
      If $m\in \overline{f}_{j}[\mathbb{Z}]$, then there exist $r\in\mathbb{Z}$ such that $\overline{f}_{j}(r)=m$, hence, by the definition of $\overline{f}_{j}$, there exist $i\in J$ such that $f(j,r)=(i,m)$. Since $f\in F(\mathbf{\Omega})$, by Lemma~\ref{l: bounded preimage}(1) there exist $(j_1,r_1),(j_2,r_2)\in \Omega$ such that $f^{-1}[(i,m)]=[(j_1,r_1),(j_2,r_2)]$. So, by Lemma~\ref{l:pre f^tilde}(1) $\widetilde{f}(j_1)=i=\widetilde{f}(j_2)$ and since $\widetilde{f}$ is a bijection, $j_1=j_2:=j$. Then $f^{-1}[(i,m)]=[(j,r_1),(j,r_2)]$; thus $\overline{f}_j^{-1}[m]=[r_1,r_2]$. % Then by Lemma~\ref{l:F(Z)} we conclude that $\overline{f}_{j}\in F(\mathbb{Z})$.  
      Finally, we prove that $\overline{f}_{j}\in F_n(\mathbb{Z})$, by showing that  $\overline{f}_{j}$ is $n$-periodic. For all $j\in J$ and $m\in\mathbb{Z}$, by the $n$-periodicity of $f$ and Lemma \ref{l:2N}, we have 
      $(\widetilde{f}(j),\overline{f}_{j}(m))=f(j,m)=f^{(2n)}(j,m)=f((j,m)-n)+n=f(j,m-n)+n
      =(\widetilde{f}(j),\overline{f}_{j}(m-n))+n
      =(\widetilde{f}(j),\overline{f}_{j}(m-n)+n)$; hence $\overline{f}_{j}(m)=\overline{f}_{j}(m-n)+n=(\overline{f}_{j})^{(2n)}(m)$, by Lemma~\ref{l:2N}, we have that $\overline{f}_{j}$ is $n$-periodic. 
      %Let us show now that $\overline{f}_{j}$ is N-periodic. Let $m\in\mathbb{Z}$, by Lemma \ref{l:2N}, $f^{(2N)}(j,m)=f((j,m)-N)+N=f(j,m-N)+N=(\widetilde{f}(j),\overline{f}_{j}(m-N)+N)$. And since $f$ is N-periodic, $f=f^{(2N)}$, so $(\widetilde{f}(j),\overline{f}_{j}(m))=(\widetilde{f}(j),\overline{f}_{j}(m-N)+N)$, then by Lemma~\ref{l:2N} $\overline{f}_j(m)=\overline{f}_{j}(m-N)+N=(\overline{f}_{j})^{(2N)}(m)$. Hence, $\overline{f}_{j}$ is N-periodic, meaning that $\overline{f}_{j}\in F_N(\mathbb{Z})$.
  \end{proof}
  %{\color{blue}
  \begin{example}
  In Figure~\ref{f: figure 1} we can see a function $f$ of $\m F_2(\mathbb{Q}\overrightarrow{\times}\mathbb{Z})$ and its $0$-th component $\overline{f}_0$ which is an element of $\m F_2(\mathbb{Z})$.
  \end{example}
%}

\begin{figure}[H]\def\eq{=}
\begin{center}
{\scriptsize
\begin{tikzpicture}
[scale=0.4]
%--------
\node[fill,draw,circle,scale=0.3,left](9) at (0,9){};
\node[right](9.1,0) at (-.9,9.25){$3$};
\node[fill,draw,circle,scale=0.3,right](9r) at (2,9){};
\node[right](9.1) at (2.1,9.25){$3$};

\node[fill,draw,circle,scale=0.3,left](8) at (0,8){};
\node[right](8.1,0) at (-.9,8.25){$2$};
\node[fill,draw,circle,scale=0.3,right](8r) at (2,8){};
\node[right](8.1) at (2.1,8.25){$2$};

\node[fill,draw,circle,scale=0.3,left](7) at (0,7){};
\node[right](7.1,0) at (-.9,7.25){$1$};
\node[fill,draw,circle,scale=0.3,right](7r) at (2,7){};
\node[right](7.1) at (2.1,7.25){$1$};

\node[fill,draw,circle,scale=0.3,left](6) at (0,6){};
\node[right](6.1,0) at (-.9,6.25){$0$};
\node[fill,draw,circle,scale=0.3,right](6r) at (2,6){};
\node[right](6.1) at (2.1,6.25){$0$};

%---------
\node[fill,draw,circle,scale=0.3,left](4) at (0,4){};
\node[right](9.1,0) at (-.9,4.25){$7$};
\node[fill,draw,circle,scale=0.3,right](4r) at (2,4){};
\node[right](4.1) at (2.1,4.25){$7$};

\node[fill,draw,circle,scale=0.3,left](3) at (0,3){};
\node[right](3.1,0) at (-.9,3.25){$6$};
\node[fill,draw,circle,scale=0.3,right](3r) at (2,3){};
\node[right](3.1) at (2.1,3.25){$6$};

\node[fill,draw,circle,scale=0.3,left](2) at (0,2){};
\node[right](2.1,0) at (-.9,2.25){$5$};
\node[fill,draw,circle,scale=0.3,right](2r) at (2,2){};
\node[right](2.1) at (2.1,2.25){$5$};

\node[fill,draw,circle,scale=0.3,left](1) at (0,1){};
\node[right](1.1,0) at (-.9,1.25){$4$};
\node[fill,draw,circle,scale=0.3,right](1r) at (2,1){};
\node[right](1.1) at (2.1,1.25){$4$};

%--------

\node[fill,draw,circle,scale=0.3,left](-1) at (0,-1){};
\node[right](-1.1,0) at (-.9,-1.25){$3$};
\node[fill,draw,circle,scale=0.3,right](-1r) at (2,-1){};
\node[right](9.1) at (2.1,-1.25){$3$};

\node[fill,draw,circle,scale=0.3,left](-2) at (0,-2){};
\node[right](-2.1,0) at (-.9,-2.25){$2$};
\node[fill,draw,circle,scale=0.3,right](-2r) at (2,-2){};
\node[right](-2.1) at (2.1,-2.25){$2$};

\node[fill,draw,circle,scale=0.3,left](-3) at (0,-3){};
\node[right](-3.1,0) at (-.9,-3.25){$1$};
\node[fill,draw,circle,scale=0.3,right](-3r) at (2,-3){};
\node[right](-3.1) at (2.1,-3.25){$1$};

\node[fill,draw,circle,scale=0.3,left](-4) at (0,-4){};
\node[right](-4.1,0) at (-.9,-4.25){$0$};
\node[fill,draw,circle,scale=0.3,right](-4r) at (2,-4){};
\node[right](-4.1) at (2.1,-4.25){$0$};
%--------
\node at (1)[below= 1pt]{$\vdots$};
\node at (1r)[below= 1pt]{$\vdots$};
\node at (6)[below=1pt]{$\vdots$};
\node at (6r)[below=1pt]{$\vdots$};
%\node at (1)[above=3pt]{$\vdots$};
%\node at (1r)[above=3pt]{$\vdots$};
\node at (9)[above=3pt]{$\vdots$};
\node at (9r)[above=3pt]{$\vdots$};
\node at (-4)[below=-1pt]{$\vdots$};
\node at (-4r)[below=-1pt]{$\vdots$};
%--------

\draw[-](1)--(7r);
\draw[-](2)--(7r);
\draw[-](3)--(9r);
\draw[-](4)--(9r);

\draw[-](-1)--(3r);
\draw[-](-2)--(3r);
\draw[-](-3)--(1r);
\draw[-](-4)--(1r);

\node at (1,-5){\normalsize$f$};
\end{tikzpicture}
\qquad \qquad \qquad
\begin{tikzpicture}
[scale=0.4]
%--------
\node[fill,draw,circle,scale=0.3,left](9) at (0,9){};
\node[right](9.1,0) at (-.9,9.25){$7$};
\node[fill,draw,circle,scale=0.3,right](9r) at (2,9){};
\node[right](9.1) at (2.1,9.25){$7$};

\node[fill,draw,circle,scale=0.3,left](8) at (0,8){};
\node[right](8.1,0) at (-.9,8.25){$6$};
\node[fill,draw,circle,scale=0.3,right](8r) at (2,8){};
\node[right](8.1) at (2.1,8.25){$6$};

\node[fill,draw,circle,scale=0.3,left](7) at (0,7){};
\node[right](7.1,0) at (-.9,7.25){$5$};
\node[fill,draw,circle,scale=0.3,right](7r) at (2,7){};
\node[right](7.1) at (2.1,7.25){$5$};

\node[fill,draw,circle,scale=0.3,left](6) at (0,6){};
\node[right](6.1,0) at (-.9,6.25){$4$};
\node[fill,draw,circle,scale=0.3,right](6r) at (2,6){};
\node[right](6.1) at (2.1,6.25){$4$};

\node[fill,draw,circle,scale=0.3,left](5) at (0,5){};
\node[right](9.1,0) at (-.9,5.25){$3$};
\node[fill,draw,circle,scale=0.3,right](5r) at (2,5){};
\node[right](5.1) at (2.1,5.25){$3$};

\node[fill,draw,circle,scale=0.3,left](4) at (0,4){};
\node[right](8.1,0) at (-.9,4.25){$2$};
\node[fill,draw,circle,scale=0.3,right](4r) at (2,4){};
\node[right](4.1) at (2.1,4.25){$2$};

\node[fill,draw,circle,scale=0.3,left](3) at (0,3){};
\node[right](3.1,0) at (-.9,3.25){$1$};
\node[fill,draw,circle,scale=0.3,right](3r) at (2,3){};
\node[right](3.1) at (2.1,3.25){$1$};

\node[fill,draw,circle,scale=0.3,left](2) at (0,2){};
\node[right](2.1,0) at (-.9,2.25){$0$};
\node[fill,draw,circle,scale=0.3,right](2r) at (2,2){};
\node[right](2.1) at (2.1,2.25){$0$};
%---------
\node at (9)[above=3pt]{$\vdots$};
\node at (9r)[above=3pt]{$\vdots$};
\node at (2)[below=-1pt]{$\vdots$};
\node at (2r)[below=-1pt]{$\vdots$};
%--------
\draw[-](2)--(6r);
\draw[-](3)--(6r);
\draw[-](4)--(8r);
\draw[-](5)--(8r);

%--------
\node at (1,1.25){\normalsize $\overline{f}_0$};
\end{tikzpicture}
}
\caption{An element of $\m F_n(\mathbb{Q}\overrightarrow{\times} \mathbb{Z})$ and one of its components in $\m F_n(\mathbb{Z})$}
\label{f: figure 1}
\end{center}
%\label{f:omega bar}

\end{figure}
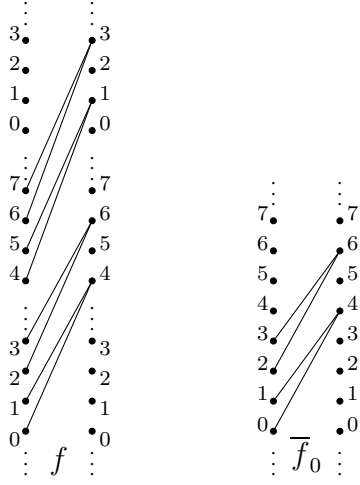

  \subsection{Wreath products}

  In this section we show that $\m F_n(\m{J} \overrightarrow{\times}\mathbb{Z})$ is actually a wreath product and that every $n$-periodic $\ell$-pregroup can be embedded in the wreath product of an $\ell$-group and an  $n$-periodic simple $\ell$-pregroup (our second representation theorem for $n$-periodic $\ell$-pregroups). 

  \medskip

  First recall that if $\m H=(H, \overline{\circ}, 1_H)$ and $\m N=(N, \odot, 1_N)$ are monoids and $\m H$ acts on the monoid $\m N$ on the right  via $\otimes: \m H \curvearrowright \m N$, 
  (i.e., $n \otimes 1_H=n$ $(n \otimes h_1)\otimes h_2=n \otimes (h_1 \overline{\circ} h_2)$, $(n_1 \odot n_2)\otimes h=(n_1 \otimes h) \odot (n_2 \otimes h)$ and $1_N \otimes h=1_N$),
  then the \emph{semidirect product} $\m H \ltimes_{\otimes} \m N$ is defined by $$(h_1, n_1)(h_2, n_2)=(h_1 \overline{\circ} h_2, (n_1 \otimes h_2) \odot n_2)$$ and it forms a monoid with unit $(1_H, 1_N)$. 
  
  If we take $\m N=\m F^J$, where  $J$ is a set and $\m F=(F, \ocirc, 1_F)$ is a monoid, i.e., the operation of $\m N$ is $(n_1 \odot n_2)(j):=n_1(j) \ocirc n_2(j)$, and if $\m H$ acts on the set $J$ on the left by $*: \m H \curvearrowright J$, then $\m H$ acts on $\m F^J$ on the right
via $\otimes: \m H \curvearrowright \m F^J$, where
\begin{center}
      $n \otimes h:=n \circ \lambda_h$ and $\lambda_h(j):=h *j$,  
\end{center}
hence 
$(n\otimes h)(j)=(n \circ \lambda_h)(j)=n(\lambda_h(j))=
n(h*j)$.
Indeed, %$\m N$ is a monoid (the direct power of $\m F$, 
%for every $n_1, n_2, n_3 \in \m F^J$, $(n_1 \odot (n_2 \odot n_3))(j)=n_1(j) \ocirc (n_2(j) \ocirc n_3(j))=(n_1(j) \ocirc n_2(j)) \ocirc n_3(j)=((n_1 \odot n_2) \odot n_3)(j)$, 
$\otimes$ is a right action because 
$((n_1 \odot n_2)\otimes h)(j)=
((n_1 \odot n_2)\circ \lambda_h)(j)=
(n_1 \odot n_2)(\lambda_h(j))
=n_1(\lambda_h(j)) \ocirc  n_2(\lambda_h(j))
=n_1 \circ \lambda_h(j) \ocirc  n_2 \circ \lambda_h(j)
=(n_1 \otimes h)(j) \ocirc  (n_2 \otimes h)(j)
=((n_1 \otimes h) \odot (n_1 \otimes h))(j)$ and $(1_N \otimes h)(j)=(1_N \circ \lambda_h)(j)=1_N(\lambda_h(j))=1_F=1_N(j)$.
 Also, for every $h_1, h_2 \in H$ and $j \in J$, we have $(\lambda_{h_1} \circ \lambda_{h_2}) (j)=h_1 *(h_2*j)=(h_1 \overline{\circ} h_2)*j=\lambda_{h_1 \overline{\circ} h_2}(j)$, so 
$(n \otimes h_1)\otimes h_2=(n \circ \lambda_{h_1} \circ \lambda_{h_2})=n \circ \lambda_{h_1\overline{\circ} h_2}= n \otimes (h_1 \overline{\circ} h_2)$. The semidirect product  $\m H \ltimes_{\otimes} \m F^J$ is known as the \emph{wreath product}  $\m H \wr_{J, *} \m F$, or simply  $\m H \wr \m F$ when  $\m H$ is a submonoid of $\m {End}(J)$ and $h*j:=h(j)$.
%; then $n(h(j))=n(h*j)=n(\lambda_h(j))=(n \circ \lambda_h)(j)=(n\otimes h)(j)$. 
One such case  is when we take $\m H:=\m {Aut}(\m J)$ to be the order-preserving automorphisms of a chain $\m J$, %(i.e.,  $\m H:=\m {F}_1(\m J)$), 
resulting in the wreath product $\m {Aut}(\m J) \wr \m F= \m {Aut}(\m J) \ltimes \m F^J$. 

%In analogy to the above definition, given a chain $\m J=(J, \leq)$ and an $\ell$-pregroup $\m F$, we denote by $\m F^{\m J}$ the poset of all order-preserving functions from $\m J$ to $\m F$. Then, this is a subalgebra of the $\ell$-pregroup $\m F ^J$, the direct power of $\m F$.....

\medskip

Now, we expand the monoid structure of the wreath product  $\m H \wr_{J, *} \m F$ relative to $J$ and $*$, when $\m H$ is an $\ell$-group,  $\m J$ is a chain, $\m F$ is an $\ell$-pregroup, and $*: \m H \curvearrowright \m J$ is a left action of $\m H$  on the chain $\m J$, by defining an order and inversion operations. We say that a monoid $\m H$ acts on a chain $\m J$ via  $*: \m H \curvearrowright \m J$, if $\m H$ acts on the set $J$ and $*$ is monotone on both coordinates.

We define 
\begin{center}
  $(h_1,n_1) \leq (h_2, n_2)$ by: $h_1 \leq h_2$ and $h_1*j=h_2*j \Rightarrow n_1(j) \leq n_2(j)$, for all $j$.   
\end{center}
Also, $(h_1,n_1) \jn (h_2, n_2)=(h_1\jn h_2, n_\jn)$ and $(h_1,n_1) \mt (h_2, n_2)=(h_1 \mt h_2, n_\mt)$, where for every $j$,
$$  n_\jn(j)= 
\begin{cases}
 n_2(j)  \\
 n_1(j) \\
 n_1(j) \jn n_2(j)
 \end{cases}
  n_\mt(j)= 
\begin{cases}
 n_1(j) & \text{ if } h_1*j<h_2*j \\
 n_2(j) &  \text{ if } h_2*j<h_1*j\\
 n_1(j) \mt n_2(j)&  \text{ if } h_1*j=h_2*j
\end{cases}$$
Also, we define: 
\begin{center}
    $(h,n)^\ell:=(h^{-1}, n^\ell \otimes h^{-1})$ and $(h,n)^r:=(h^{-1}, n^r \otimes h^{-1})$
,
\end{center} where $n \mapsto n^\ell$ and $n \mapsto n^r$ are the operations of the direct product $\m F^J$. We denote by $\m H \wr_{\m {J}, *} \m F$ the resulting structure, which we call the \emph{($\ell$-pregroup) wreath product}, and we prove that it is an $\ell$-pregroup.

\begin{theorem}
If $*: \m H \curvearrowright \m J$ is a left action of the $\ell$-group $\m H$  on the chain $\m J$ and $\m F$ is an $\ell$-pregroup,  %the induced right action $\otimes: \m H \curvearrowright \m F^J$  is order-preserving in both coordinates, and $^{\ell}$ and $^{r}$ distribute over $\otimes$, 
then  $\m H \wr_{\m J, *} \m F$ is an $\ell$-pregroup.
Furthermore, 
\begin{enumerate}
    \item for every $k$, $\m H \wr_{\m J, *} \m F$ is $k$-periodic iff $\m F$ is $k$-periodic.
    \item $\m H \wr_{\m J, *} \m F$ is distributive iff %\m H$ and $\m F$ are distributive. 
    $\m F$ is distributive.
    %\CHAT{N}{Is the converse true?}
    %\CHAT{N}{Check if lattice distributivity is  also preserved.} %$\m H \wr_{\m J, *} \m F$ is abelian iff $\m F$ is abelian.
\end{enumerate}

\end{theorem}

\begin{proof}
From the above calculations we know that  $\m H \wr_{J, *} \m F$  supports a monoid structure.     

We show that $\leq$ is a partial order; it is clearly reflexive. For transitivity, we assume  $(h_1,n_1) \leq (h_2, n_2) \leq (h_3, n_3)$, so $h_1 \leq h_2 \leq h_3$; in particular $h_1*j\leq h_2*j\leq h_3*j$ for all $j$, by the monotonicity of $*$ on its left coordinate. So, if $h_1*j=h_3*j$, then $h_1*j= h_2*j= h_3*j$, so $ n_1(j) \leq n_2(j) \leq n_3(j)$. For antisymmetry,  we assume  $(h_1,n_1) \leq (h_2, n_2) \leq (h_1, n_1)$, so $h_1 \leq h_2 \leq h_1$; in particular $h_1*j= h_2*j$, for all $j$. Therefore, by assumption we get  $ n_1(j) \leq n_2(j) \leq n_1(j)$, for all $j$.  

We will use the fact that $*$ distributes over joins on its left coordinate, which we establish now. 
%Since $\m H$ is a group  and $*$ is order-preserving on the left, we get that $h \mapsto h*j$ is an order-preserving bijection; therefore, it preserves joins and meets.
%Here is one possible way to get preservation, by assuming the action is order-preserving on the right coordinate, as well.
First note that the map $j \mapsto h*j$ is an order-preserving bijection: $h*i=h*j\Rightarrow h^{-1}*(h*i)=h^{-1}*(h*j) \Rightarrow i=j$ and  the action is order-preserving on the right coordinate; so, $j \mapsto h*j$ distributes over joins and meets.
Now, to prove $(h_1 \jn h_2)*j=h_1*j \jn h_2*j$, by order-preservation we only need to check $(h_1 \jn h_2)*j\leq h_1*j \jn h_2*j$, which is equivalent to $j\leq (h_1 \jn h_2)^{-1}*(h_1*j \jn h_2*j)$, i.e., to $j\leq (h_1 \jn h_2)^{-1}*(h_1*j) \jn (h_1 \jn h_2)^{-1}*(h_2*j)$, which is equivalent to $j\leq (1 \mt h_2^{-1}\overline{\circ} h_1)*j \jn(1 \mt h_1^{-1}\overline{\circ} h_2)*j$, which in turn holds by order preservation on the left coordinate.

We show that $\leq$ is a lattice-order and that  join and meet are given by the formulas. First note that 
 $(h_1,n_1) \leq (h_1 \jn h_2, n_\jn)$, because  $h_1 \leq h_1 \jn h_2$ and, if $h_1 *j  = (h_1 \jn h_2)*j$, %=(h_1*j) \jn (h_2 *j)$,  then  $h_2 *j \leq h_1 *j$; 
 then  $h_2 *j \leq  (h_1 \jn h_2)*j= h_1 *j$;
 by definition, we get that then $n_\jn(j)\in \{n_1(j), n_1(j)\jn n_2(j)\}$, thus $n_1(j) \leq n_\jn(j)$. Also, if
 $(h_1,n_1) \leq (h, n)$ and $(h_2,n_2) \leq (h, n)$, then $h_1, h_2 \leq h$, so  $h_1 \jn h_2 \leq h$. If $ (h_1 \jn h_2)*j=h*j$, for some $j$, then $(h_1*j) \jn (h_2 *j)=h*j$, so because $\m J$ is a chain we get $h*j= h_1*j$ or $h*j= h_2*j$. In either case, by the definition of $n_\jn$,  we get $n_\jn(j)= n_1(j)\jn n_2(j)$. Therefore, 
 $(h_1 \jn h_2,n_\jn ) \leq (h, n)$. Likewise, we show that the meet is given by the formula. 

 We will show that multiplication is order-preservating on the right. We assume that $(h_1, n_1) \leq (h_2, n_2)$, i.e.,  $h_1 \leq h_2$ and $h_1*j=h_2*j \Rightarrow n_1(j) \leq n_2(j)$, for all $j$. We will show that
$(h_1, n_1)(h,n) \leq (h_2, n_2)(h,n)$, i.e., that 
$(h_1 \overline{\circ} h, (n_1 \otimes h) \odot n) \leq (h_2 \overline{\circ} h, (n_2 \otimes h) \odot n)$.
Since $h_1 \leq h_2$, by monotonicity of $\overline{\circ}$, we get $h_1\overline{\circ} h \leq h_2\overline{\circ} h$. For each $j$, if $(h_1\overline{\circ} h)*j = (h_2\overline{\circ} h)*j$, then  $h_1*(h*j) = h_2 *(h*j)$, so 
$n_1(h*j) \leq n_2(h*j)$, by hypothesis. By the order-preservation of $\ocirc$ in $\m F$, we get 
$ n_1(h*j) \ocirc n(j) \leq n_2(h*j) \ocirc n(j)$. Since $n_i(h*j)= (n_i \circ \lambda_h)(j)=(n_i \otimes h)(j)$, we get
%$ (n_1 \circ \lambda_h)(j) \ocirc n(j) \leq (n_2 \circ \lambda_h)(j) \ocirc n(j)$
$ (n_1 \otimes h)(j) \ocirc n(j) \leq (n_2 \otimes h)(j) \ocirc n(j)$, i.e., 
$ ((n_1 \otimes h) \odot n)(j) \leq ((n_2 \otimes h) \odot n)(j)$. 

 For the left side, we will show that
$(h,n)(h_1, n_1) \leq (h,n)(h_2, n_2)$, i.e., that 
$(h \overline{\circ} h_1, (n \otimes h_1) \odot n_1) \leq (h \overline{\circ} h_2, (n \otimes h_2) \odot n_2)$. We get $h \overline{\circ} h_1 \leq h \overline{\circ} h_2$, by the monotonicity of $\overline{\circ}$. For each $j$, if $(h\overline{\circ} h_1)*j = (h\overline{\circ} h_2)*j$, then
$h*(h_1*j) = h *(h_2*j)$, so 
$h^{-1} *(h*(h_1*j))= h^{-1} *(h *(h_2*j))$. % by the monotonicity of $*$ on the left coordinate. 
Since $\m H$ is a group we get $h^{-1} *(h*(i))=(h^{-1}\overline{\circ} h)*i=1_H*i=i$, therefore $h_1*j = h_2*j$, hence $n_1(j) \leq n_2(j)$, by hypothesis; also $n(h_1*j) = n(h_2*j)$. By the monotonicity of $\ocirc$, we get $n(h_2*j) \ocirc n_1(j) \leq n(h_2*j) \ocirc n_2(j)$. Since $n(h_i*j)= (n \circ \lambda_{h_i})(j)=(n \otimes h_i)(j)$, we get $(n \otimes h_1)(j) \ocirc n_1(j) \leq (n \otimes h_2)(j) \ocirc n_2(j)$, hence $((n \otimes h_1) \odot n_1)(j) \leq ((n \otimes h_2) \odot n_2)(j)$.

Now, show that for all $x \in H \times F^J$, we have $x^\ell x \leq 1$ and $1 \leq x^r x$. 
Using the fact that $f^\ell f \leq 1_F$ holds for all $f \in F$, and thus also $n^\ell \odot n \leq 1_N$ for $n \in F^J$, we get 
$(h^{-1}, n^\ell \otimes h^{-1})(h,n)
=(h^{-1} \overline{\circ} h, ((n^\ell \otimes h^{-1}) \otimes h) \odot n )
=(1_H, (n^\ell \otimes (h^{-1} \overline{\circ} h)) \odot n)
=(1_H, (n^\ell \otimes 1_H) \odot n)
=(1_H, n^\ell \odot n)
\leq (1_H, 1_N)$. Also, $(h^{-1}, n^r \otimes h^{-1})(h,n)
=(h^{-1} \overline{\circ} h, ((n^r \otimes h^{-1}) \otimes h) \odot n )
=(1_H, (n^r \otimes (h^{-1} \overline{\circ} h)) \odot n)
=(1_H, (n^r \otimes 1_H) \odot n)
=(1_H, n^r \odot n)
\geq (1_H, 1_N)$. Note that we used the fact that $\m H$ acts via $\otimes$ on the set $F^J$.

For showing $x x^r \leq 1 \leq x x^\ell$ we will use that $\m H$ acts also on the monoid $\m F^J$ and the order-preservation of $\otimes$ on the first coordinate, which we establish first. If $n_1\leq n_2$, i.e., $n_1(j)\leq n_2(j)$ for all $j$, then for all $h$ we get $n_1 \otimes h \leq n_2 \otimes h$. Indeed, for all $j$, we have $n_1(h*j)\leq n_1(h*j)$, so $(n_1 \otimes h)(j) \leq (n_2 \otimes h)(j)$. Thus, $\otimes: \m H \curvearrowright \m F^J$  is order-preserving in the left coordinate.
Now, for $x x^r \leq 1 \leq x x^\ell$, we have $(h,n)(h^{-1}, n^r \otimes h^{-1})
=(h \overline{\circ} h^{-1}, 
(n \otimes h^{-1}) \odot (n^r  \otimes h^{-1}) )
=(1_H, (n  \odot n^r) \otimes h^{-1})
\leq (1_H, 1_N \otimes h^{-1})
=(1_H, 1_N)$ and
$(h,n)(h^{-1}, n^\ell \otimes h^{-1})
=(h \overline{\circ} h^{-1}, 
(n \otimes h^{-1}) \odot (n^\ell  \otimes h^{-1}) )
=(1_H, (n  \odot n^\ell) \otimes h^{-1})
\geq (1_H, 1_N \otimes h^{-1})
=(1_H, 1_N)$

%[[They do not quite distribute]]Using the fact that  $^{\ell}$ and $^{r}$ distribute over $\otimes$, we have $(h,n)^{\ell r}=(h^{-1}, n^\ell \otimes h^{-1})^r= ((h^{-1})^ {-1}, (n^{\ell} \otimes h^{-1})^r \otimes h^{-1})=(h,(n^{\ell r} \otimes(h^{-1})^ {-1}) \otimes h^{-1})= (h,n \otimes (h \overline{\circ} h^{-1}))= (h,n \otimes 1_H)=(h,n)$. Likewise we get $(h,n)^{\ell r}=(h,n)$.

(1) We now prove that $(n\otimes h)^\ell=n^\ell \otimes h$, i.e., $(n\otimes h)^\ell(j)=(n^\ell \otimes h)(j)$, for all $j$. By the coordinate-wise definition of $^{\ell}$ on elements of $F^J$, this is equivalent to $((n\otimes h)(j))^\ell=(n^\ell \otimes h)(j)$, to
$(n(h *j))^\ell=n^\ell(h *j)$ and to
$(n(h *j))^\ell=(n(h *j))^\ell$, which holds. Likewise, we have $(n\otimes h)^r=n^r \otimes h$. We use these properties to establish the characterization of the periodic case.

Note that $(h,n)^{\ell \ell}=(h^{-1}, n^\ell \otimes h^{-1})^\ell= ((h^{-1})^ {-1}, (n^{\ell} \otimes h^{-1})^\ell \otimes (h^{-1})^{-1}))=(h,(n^{\ell \ell} \otimes h^ {-1}) \otimes h)=
(h,n^{\ell \ell} \otimes (h^{-1} \overline{\circ} h))= (h,n^{\ell \ell} \otimes 1_H)=(h,n^{\ell \ell})$. More generally, we get  $(h,n)^{(2k)}=(h,n^{(2k)})$, for every $k$. Therefore, for each positive integer $k$, we have that  $(h,n)^{(2k)}=(h,n)$ for all $(h,n)\in\m H \wr_{\m J, *} \m F$, iff  $(h,n^{(2k)})=(h,n)$ for all $(h,n)\in\m H \wr_{\m J, *} \m F$, iff  $n^{(2k)}=n$ for all $n \in F^J$.

(2) We first show that if $\m F$ is distributive, then $\m H \wr_{\m J, *} \m F$ is distributive, as well. We will show that for all $(h_1,n_1),(h_2,n_2),(h_3,n_3)\in\m H \wr_{\m J, *} \m F$ we have:
\[(h_1,n_1)\wedge((h_2,n_2)\vee(h_3,n_3))=((h_1,n_1)\wedge(h_2,n_2))\vee((h_1,n_1)\wedge(h_3,n_3))\]

For simplicity, we  define $n_{i\wedge j}$ by $(h_i,n_i)\wedge(h_l,n_l)=(h_i\wedge h_l,n_{i\wedge j})$, instead of the generic term $n_{\mt}$, and similarly for the join.
%, where the first coordinate indicates the operation in $H$ and the second coordinate must be calculated according tho the definitions of $n_{\wedge}$ and $n_{\vee}$ respectively. 
By the distributivity of $\m J$ we get  $h_1*j\wedge(h_2*j\vee h_3*j)=(h_1*j\wedge h_2*j)\vee(h_1*j\wedge h_3*j)$ and in particular, there exists $i\in\{1,2,3\}$ such that $h_1*j\wedge(h_2*j\vee h_3*j)=h_i*j$, since $\m J$ is a chain.
If $h_1*j,h_2*j,h_3*j$ are all distinct, then 
$ n_{1\wedge(2\vee3)}(j)=n_i(j)=n_{(1\wedge 2)\vee(1\wedge 3)}(j)$
%\begin{align*}
%   n_{1\wedge(2\vee3)}(j)&=n_i(j)\\
%                           &=n_{(1\wedge 2)\vee(1\wedge 3)}(j)
%\end{align*}

On the other hand, if $h_1*j,h_2*j,h_3*j$ are not all distinct, for $\{i,k\}=\{2,3\}$, we have
%\begin{align*}
%   n_{2 \vee 3}(j)&=\begin{cases}
%  n_k(j) & \text{ if } h_i*j<h_k*j=h_1*j \\
%  n_i(j) & \text{ if } h_1*j=h_k*j<h_i*j \\
%  n_2(j) & \text{ if } h_1*j<h_k*j=h_i*j \\
%  n_2(j)\vee n_3(j) & \text{ if } h_1*j>h_k*j=h_i*j \\
%  n_2(j)\vee n_3(j) & \text{ if } h_1*j=h_k*j=h_i*j 
%\end{cases}
%\end{align*}
%and then
\begin{align*}
   n_{1\wedge(2\vee 3)}(j)&=\begin{cases}
  n_1(j)\wedge n_k(j) & \text{ if } h_i*j<h_k*j=h_1*j \\
  n_1(j) & \text{ if } h_1*j=h_k*j<h_i*j \\
  n_1(j) & \text{ if } h_1*j<h_k*j=h_i*j \\
  n_2(j)\vee n_3(j) & \text{ if } h_k*j=h_i*j< h_1*j \\
  n_1(j)\wedge(n_2(j)\vee n_3(j)) & \text{ if } h_1*j=h_k*j=h_i*j 
\end{cases}
\end{align*}
%Similarly, we can calculate $\alpha_1\wedge\alpha_i$ and $\alpha_1\wedge\alpha_i$ at $j$.
%\begin{align*}
%   [(h_1,n_1)\wedge(h_i,n_i)](j)&=\begin{cases}
% (h_i*j, n_i(j)) & \text{ if } h_i*j<h_k*j=h_1*j \\
% (h_1*j, n_1(j)) & \text{ if } h_1*j=h_k*j<h_i*j \\
% (h_1*j, n_1(j)) & \text{ if } h_1*j<h_k*j=h_i*j \\
% (h_i*j, n_i(j)) & \text{ if } h_1*j>h_k*j=h_i*j \\
% (h_1*j, n_1(j)\vee n_i(j)) & \text{ if } h_1*j=h_k*j=h_i*j 
%\end{cases}
%\end{align*}
%and
%\begin{align*}
%   [(h_1,n_1)\wedge(h_k,n_k)](j)&=\begin{cases}
% (h_1*j, n_1(j)\wedge n_k(j)) & \text{ if } h_i*j<h_k*j=h_1*j \\
% (h_1*j, n_1(j)\wedge n_k(j)) & \text{ if } h_1*j=h_k*j<h_i*j \\
% (h_1*j, n_1(j)) & \text{ if } h_1*j<h_k*j=h_i*j \\
% (h_k*j, n_k(j)) & \text{ if } h_1*j>h_k*j=h_i*j \\
% (h_1*j, n_1(j)\vee n_k(j)) & \text{ if } h_1*j=h_k*j=h_i*j 
%\end{cases}
%\end{align*}
%$$
%n_{1\wedge i}=\begin{cases}
% n_i(j) \\
% n_1(j) \\
% n_1(j) \\
% n_i(j) \\
% n_1(j)\vee n_i(j))  
%\end{cases}
%n_{1\wedge k}=\begin{cases}
% n_1(j)\wedge n_k(j) & \text{ if } %h_i*j<h_k*j=h_1*j \\
% n_1(j)\wedge n_k(j) & \text{ if } h_1*j=h_k*j<h_i*j \\
% n_1(j) & \text{ if } h_1*j<h_k*j=h_i*j \\
% n_k(j) & \text{ if } h_1*j>h_k*j=h_i*j \\
% n_1(j)\vee n_k(j) & \text{ if } h_1*j=h_k*j=h_i*j 
%\end{cases}
%$$
%Therefore, i
Also, if $h_1*j,h_2*j,h_3*j$ are not all distinct, 
$ n_{(1\wedge 2)\vee(1\wedge 3)}(j)=$
\begin{align*}
      \begin{cases}
 n_1(j)\wedge n_k(j) & \text{ if } h_i*j<h_k*j=h_1*j \\
 n_1(j) & \text{ if } h_1*j=h_k*j<h_i*j \\
 n_1(j) & \text{ if } h_1*j<h_k*j=h_i*j \\
 n_2(j)\vee n_3(j) & \text{ if } h_k*j=h_i*j< h_1*j \\
 (n_1(j)\vee n_2(j))\wedge(n_1(j)\vee n_3(j)) & \text{ if } h_1*j=h_k*j=h_i*j 
       \end{cases}
\end{align*}
By the distributivity of $\m F$, we get $n_{1\wedge(2\vee 3)}(j)=n_{(1\wedge 2)\vee(1\wedge3)}(j)$ for all $j\in J$, and by the distributivity of $\m H$, we have $(h_1,n_1)\wedge((h_2,n_2)\vee(h_3,n_3))=((h_1,n_1)\wedge(h_2,n_2))\vee((h_1,n_1)\wedge(h_3,n_3))$.

On the other hand, if $\m H \wr_{\m J, *} \m F$ is distributive, for all 
%$h_1,h_2,h_3\in H$ and $n\in N$, we have $(h_1,n)\wedge((h_2,n)\vee(h_3,n))=((h_1,n)\wedge(h_2,n))\vee((h_1,n)\wedge(h_3,n))$, hence the first coordinate yields $h_1\wedge(h_2\vee h_3)=(h_1\wedge h_2)\vee(h_1 \wedge h_3)$. So $\m H$ is distributive. Now, if
$n_1,n_2,n_3\in F$ and $h\in H$,  $(h,n_1)\wedge((h,n_2)\vee(h,n_3))=((h,n_1)\wedge(h,n_2))\vee((h,n_1)\wedge(h,n_3))$, hence $(h,n_1\vee(n_2\wedge n_3))=(h,(n_1\vee n_2)\wedge(n_1\vee n_3))$, so $n_1\vee(n_2\wedge n_3)=(n_1\vee n_2)\wedge(n_1\vee n_3)$. Thus, $\m F$ is distributive.  
%An element $(h,n)$ is invertible iff $(h,n)^\ell=(h,n)^r$ iff $(h^{-1}, n^\ell \otimes h^{-1})=(h^{-1}, n^r \otimes h^{-1})$ iff $n^\ell=n^r$ iff $(n(j))^\ell=(n(j))^r$, for all $j$ iff $n(j)$ is invetible , for all $j$. Therefore,  $\m H \wr_{\m J, *} \m F$ is abelian iff ....
\end{proof}

We want to stress that the wreath product construction applies to arbitrary $\ell$-pregroups, not only to distributive ones. 

\begin{remark}
    We note that the more general definition where $\m H$ is allowed to be an $\ell$-pregroup, 
$(h,n)^\ell:=(h^\ell, n^\ell \otimes h^\ell)$ and $(h,n)^r:=(h^r, n^\ell \otimes h^r)$ fails: 
$(h,n)^{\ell r}=(h^\ell, n^\ell \otimes h^\ell)^r= (h^{\ell r}, (n^{\ell} \otimes h^\ell)^r \otimes h^{\ell r})=(h,(n^{\ell r} \otimes h^{\ell}) \otimes h)=
(h,n \otimes (h^\ell \overline{\circ} h))\leq (h,n \otimes 1_H)=(h,n)$, but unless  $h \overline{\circ} h^r=1_H$ (i.e., $h$ is invertible, for arbitrary $h$), we do not get equality.
\end{remark}

It is well known that if the monoid $\m F$ acts on the left on a set $\Omega$, then the monoid wreath product $\m H \wr_{J,*} \m F$ acts on $J\times \Omega$. It is not hard to show that if an $\ell$-pregroup $\m F$ acts on a chain $\m \Omega$ then the $\ell$-group wreath product $\m H \wr_{J, *} \m F$ acts on the lexicographic product $\m{J} \overrightarrow{\times}\m \Omega$, by $(h,g)(j,a):=(h(j),g(a))$, for $(h,g)\in\m H \wr_{J, *} \m F$ and  $(j,a)\in {J} \overrightarrow{\times}\Omega$. 
%We show that this facts extends to the $\ell$-pregroup wreath product. 
%\CHAT{I}{New}

%\begin{lemma}
%If the $\ell$-pregroup $\m F$ acts on the chain $\m \Omega$ then the $\ell$-group wreath product $\m H \wr_{J, *} \m F$ acts on the lexicographic product $\m{J} \overrightarrow{\times}\m \Omega$.
%\end{lemma}
%\begin{proof}
%For all $n\in F$ and $a\in\Omega$ we denote by $n(a)$ \CHAT{N}{We need  $n\in F^J$}  the action of $n$ on $a$. For $(h,g)\in\m H \wr_{J, *} \m F$ and  $(j,a)\in {J} \overrightarrow{\times}\Omega$, we define $(h,g)(j,a):=(h(j),g(a))$. Observe that it is well-defined since $h\in Aut(\m{J})$ and $g\in F^J$. Suppose now that $(h,g),(k,f)\in\m H \wr_{J, *} \m F$, then 
%\begin{align*}
%    ((h,g)(k,f))(j,a)&=((h\circ k),(g\otimes k)\odot f)(j,a)\\
%                     &=((h\circ k)(j),((g\otimes k)\odot f)_j(a))\\
%                     &\text{\CHAT{N}{There is no $j$ in the above definition}}\\
%                     &=(h(k(j)),g_{k(j)}(f(a)))\\
%                     &=(h,g)(k(j),f(a))\\
%                     &=(h,g)((k,f)(j,a))
%\end{align*}
%Therefore, $\m H \wr_{J, *} \m F$ acts on $\m{J} \overrightarrow{\times}\m \Omega$. \CHAT{N}{Don't we need to check the other conditions of the action?}
%\end{proof}

As promised we now show that $\m F_n(\m{J} \overrightarrow{\times}\mathbb{Z})$ is indeed an $\ell$-pregroup wreath product and use this fact to get a second representation theorem for $n$-periodic $\ell$-pregroups.

\begin{theorem}\label{t: repnper}
For every chain $\m J$ and $n \in \mathbb{Z}^+$, 
$\m F_n(\m{J} \overrightarrow{\times}\mathbb{Z})\cong \m {Aut}(\m J) \wr \m F_n(\mathbb{Z})$. Therefore, every $n$-periodic $\ell$-pregroup can be embedded in the wreath product of an $\ell$-group and the simple $n$-periodic $\ell$-pregroup $\m F_n(\mathbb{Z})$.
\end{theorem}

\begin{proof}
By Theorem~\ref{t:F(JxZ)},  every $f$ in $\m F_n(\m{J} \overrightarrow{\times}\mathbb{Z})$ can be identified with the pair $(\widetilde{f}, \overline{f})$ in $\m {Aut}(\m J) \times  (\m F_n(\mathbb{Z}))^J$, where $\overline{f}(j)=\overline{f}_j$, i.e., $\overline{f}=(\overline{f}_j)_{j \in J}$. We will verify that composition of $f$'s corresponds to multiplication of pairs  $(\widetilde{f}, \overline{f})$ in the wreath product $\m {Aut}(\m J) \wr \m F_n(\mathbb{Z})=\m {Aut}(\m J) \ltimes  (\m F_n(\mathbb{Z}))^J$ under the identification $f\equiv (\widetilde{f}, \overline{f})$. Indeed, for every $(j,m)\in J \times \mathbb{Z}$ and $f,g$ in $\m F_n(\m{J} \overrightarrow{\times}\mathbb{Z})$, we have %\CHAT{I}{I have some questions about this calculations}
%
%$$\begin{array}{rl}
% (f \circ g)(j,m)  
% & =f(g(j,m))\\
% & =f(\widetilde{g}(j), (\overline{g}_j(m))_{j \in J})\\
%& =(\widetilde{f}(\widetilde{g}(j)),(\overline{f}_{\widetilde{g}(j)}%(\overline{g}_j(m)))_{j \in J})\\
     %& =((\widetilde{f} \circ \widetilde{g})(j), ((\overline{f}_{\widetilde{g}(j)} %\circ \overline{g}_j)(m))_{j \in J}
%\\ &
%=((\widetilde{f} \circ \widetilde{g})(j), 
%(((\overline{f}\otimes \widetilde{g}) \odot \overline{g})_{j})(m))_{j \in J} )
% \\ &
%=((\widetilde{f} \circ \widetilde{g}), 
%((\overline{f}\otimes \widetilde{g}) \odot \overline{g}) )(j,m)\\ &
%=((\widetilde{f}, \overline{f})(\widetilde{g}, \overline{g}))(j,m)\end{array}$$
$$\begin{array}{rl}
 (f \circ g)(j,m)  
 & =f(g(j,m))
 %\\ & 
 =f(\widetilde{g}(j), \overline{g}_j(m))
 %\\& 
 =(\widetilde{f}(\widetilde{g}(j)),\overline{f}_{\widetilde{g}(j)}(\overline{g}_j(m)))\\
     & =((\widetilde{f} \circ \widetilde{g})(j), (\overline{f}_{\widetilde{g}(j)} \circ \overline{g}_j)(m))
%\\ &
=((\widetilde{f} \circ \widetilde{g})(j), 
(\overline{f}\otimes \widetilde{g}) \odot \overline{g})_{j}(m) )
 \\ &
=(\widetilde{f} \circ \widetilde{g}, 
(\overline{f}\otimes \widetilde{g}) \odot \overline{g} )(j,m)%\\ &
=((\widetilde{f}, \overline{f})(\widetilde{g}, \overline{g}))(j,m)\end{array}$$
where we used that 
$
\overline{f}_{\widetilde{g}(j)} \circ \overline{g}_j  
= \overline{f}(\widetilde{g}(j)) \circ \overline{g}(j)
= (\overline{f}\otimes {\widetilde{g})(j)} \circ \overline{g}(j) 
= 
(\overline{f}\otimes \widetilde{g}) \odot \overline{g})(j)
= 
(\overline{f}\otimes \widetilde{g}) \odot \overline{g})_{j}
$
%$$\begin{array}{rl}
%\overline{f}_{\widetilde{g}(j)} \circ \overline{g}_j  
%&= \overline{f}(\widetilde{g}(j)) \circ \overline{g}(j)\\ &
%= (\overline{f}\otimes {\widetilde{g})(j)} \circ \overline{g}(j) \\ &
%= (\overline{f}\otimes \widetilde{g}) \odot \overline{g})(j)\\ &
%= (\overline{f}\otimes \widetilde{g}) \odot \overline{g})_{j}
%\end{array}$$
which is, in turn, based on 
$
\overline{f}(\widetilde{g}(j))=
(\overline{f}\otimes {\widetilde{g})(j)}$, a general fact we already established above as $n(h(j))=(n\otimes h)(j)$.

%$$\begin{array}{rl}
% (f \circ g)(j,m)  & =f(g(j,m))\\
% & =f(\widetilde{g}(j), (\overline{g}_j(m))_{j \in J})\\
%    & =(\widetilde{f}(\widetilde{g}(j)),(\overline{f}_{\widetilde{g}(j)}(\overline{g}_j(m)))_{j \in J})\\
%& =((\widetilde{f} \circ \widetilde{g})(j), ((\overline{f}_{\widetilde{g}(j)} \circ \overline{g}_j)(m))_{j \in J} )\\ &
%=((\widetilde{f} \circ \widetilde{g})(j), ((\overline{f}(\widetilde{g}(j)) \circ \overline{g}(j))(m))_{j \in J} )\\ &
%=((\widetilde{f} \circ \widetilde{g})(j), ((\overline{f}\otimes {\widetilde{g})(j)} \circ \overline{g}(j))(m))_{j \in J} )\\ &
%=((\widetilde{f} \circ \widetilde{g})(j), 
%(((\overline{f}\otimes \widetilde{g}) \odot \overline{g})(j))(m))_{j \in J} )\\ &
%=((\widetilde{f} \circ \widetilde{g})(j), 
%(((\overline{f}\otimes \widetilde{g}) \odot \overline{g})_{j})(m))_{j \in J} )\\ &
%=((\widetilde{f} \circ \widetilde{g}), 
%((\overline{f}\otimes \widetilde{g}) \odot \overline{g}) )(j,m)\\ &
%=((\widetilde{f}, \overline{f})(\widetilde{g}, \overline{g}))(j,m)\end{array}$$
%where we used that 
%$\overline{f}_{\widetilde{g}(j)}=
%\overline{f}(\widetilde{g}(j))=
%\overline{f}(\widetilde{g}*j)=
%\overline{f}(\lambda_{\widetilde{g}}(j))=
%(\overline{f} \circ \lambda_{\widetilde{g}})(j)=
%(\overline{f}\otimes {\widetilde{g})(j)}$.

For the order, note that $f \leq g$ iff $(\widetilde{f}(j), (\overline{f}_j(m))_{j \in J}) \leq (\widetilde{g}(j), (\overline{g}_j(m))_{j \in J})$, for all $(j,m)$, iff 
$\widetilde{f}\leq \widetilde{g}$ and 
$\widetilde{f}(j) = \widetilde{g}(j) \Rightarrow \overline{f}_j(m) \leq \overline{g}_j(m)$ for all $(j,m)$, iff $\widetilde{f}\leq \widetilde{g}$ and 
$\widetilde{f}(j) = \widetilde{g}(j) \Rightarrow \overline{f}_j \leq \overline{g}_j$. for all $j$.
%\CHAT{I}{New}

Also, we have that if $f$ corresponds to $(\widetilde{f}, \overline{f})$, then 
$f^\ell\equiv (\widetilde{f}^{-1}, \overline{f}^\ell \otimes \widetilde{f}^{-1})$, where $\overline{f}^\ell(j):=\overline{f}(j)^\ell$ is  given by the definition of $^\ell$ to the direct product $\m F_n(\mathbb{Z})^J$. 

By Lemma~\ref{l: bounded preimage} for all $(j,n)\in {J} \overrightarrow{\times}\mathbb{Z}$ there exist $(i,m)\in{J} \overrightarrow{\times}\mathbb{Z}$ such that $f(i,m-1)<(j,n)\leq f(i,m)$ and $f^{\ell}(j,n)=(i,m)$, so $(\tilde{f}(i),\overline{f}_{i}(m-1))<(j,n)\leq (\tilde{f}(i),\overline{f}_{i}(m))$. Then $j=\tilde{f}(i)$ and $\overline{f}_{i}(m-1)<n\leq \overline{f}_{i}(m)$, therefore $\overline{f}_{i}^{\ell}(n)=m$. So $f^{\ell}(j,n)=(i,m)=(\tilde{f}^{-1}(j),\overline{f}_{\tilde{f}^{-1}(j)}^{\ell}(n))=(\widetilde{f}^{-1}, \overline{f}^\ell \otimes \widetilde{f}^{-1})(j,n)$.   

That $\m F_n(\mathbb{Z})$ is simple follows from Lemma~\ref{l: simple}.
\end{proof}
  
\begin{remark}
 We mention that in the definition of the monoid $\m H$ we denoted the operation by $\overline{\circ}$ because in our application $\m H=\m{Aut}(\m J)$ the operation is the usual composition $\circ$ of functions on $J$. Furthermore, we denoted by $\ocirc$ the operation of $\m F$, as in our case of $\m F=\m F_n(\mathbb{Z})$  it is composition $\circ$ of functions on $\mathbb{Z}$. Finally, we denoted by $\odot$ the operation of $\m N=\m F^J$, as it is the extension of the composition operation of $\m F$ to the direct power $\m F^J$. Typically, the same symbol is used for the operation in the direct product, but when this operation is functional composition  then $(n_1 \odot n_2)(j)=n_1(j) \circ n_2(j)$ looks better than $(n_1 \circ n_2)(j)$ which could be easily misinterpreted as $n_1(n_2(j))$.   
\end{remark}

  \section{The join of the periodic varieties}\label{s: joinperiodic}

In \cite{GJ} it is shown that each each periodic  $\ell$-pregroup is distributive, 
%$n$-periodic variety of $\ell$-pregroups is contained in the variety of all distributive $\ell$-pregroups; i.e.,
so $\mathsf{LP_n}\subseteq \mathsf{DLP}$, for all $n$. In this section we show that the join of all of the varieties $\mathsf{LP_n}$, for $n \in \mathbb{Z}^+$, is $\mathsf{DLP}$.
%the whole variety of distributive $\ell$-pregroups.

As mentioned earlier, $\ell$-pregroups are exactly the involutive residuated lattices that satisfy  $(xy)^\ell =y^\ell x^\ell$ and $x^{r\ell}=x=x^{\ell r}$; in the language of residuated lattices this implies $x+y=xy$, $0=1$, $\ln x=x^r$ and $\rn x=x^\ell$. Given $n \in \mathbb{Z}^+$, an involutive residuated lattice is called \emph{$n$-periodic} if it satisfies $\ln^n x=\rn^nx$; therefore, the notion of $n$-periodicity for $\ell$-pregroups is a specialization of the one for involutive residuated lattices. Since $n$-periodicity is defined equationally, the class of all $n$-periodic involutive residuated lattices is a variety, for each $n$. In \cite{GJframes} it is proved that the join of all of these $n$-periodic involutive residuated lattice varieties is equal to the whole variety of involutive residuated lattices. This proof relies, among other things, on proof-theoretic arguments on an analytic sequent calculus for the class of involutive residuated lattices; since an analytic calculus is not known for the class of distributive $\ell$-pregroups, we cannot use the same methods here, but we rely on the method of diagrams.

% As shown in \cite{GJ}, $\mathsf{LP_n}$ is a subvariety of $\mathsf{DLPG}$, for every $n \in \mathbb{Z}^+$.  We will prove that $\bigvee \mathsf{LP_n}=\mathsf{DLPG}$.  Actually, we will prove that \begin{center}   $\bigvee \mathsf{V}(\m F_n(\mathbb{Z})) \subseteq \bigvee \mathsf{LP_n} \subseteq \mathsf{DLPG} \subseteq \mathsf{V}(\{\m F_n(\mathbb{Z}): n\in \mathbb{N}\})\subseteq \bigvee \mathsf{V}(\m F_N(\mathbb{Z}))$,    \end{center}  thus rendering all of these inclusions as equalities.
 % $\mathsf{DLPG}=\mathsf{V}(\{\m F_n(\mathbb{Z}): n\in \mathbb{N}\})=\bigvee \mathsf{V}(\m F_N(\mathbb{Z}))$. Therefore, 

  \subsection{Diagrams and $n$-periodicity}

  We recall some definitions from \cite{GG}. The main idea behind these notions is to try to capture the failure of an equation in an algebra $\m F(\m \Omega)$ in a finitistic way:  retain only finitely many points of $\m \Omega$ and record the behavior of functions of $\m F(\m \Omega)$ on these points to still witness the failure. 
   A \emph{c-chain}, or \emph{chain with covers}, is a triple $(\Delta,\leq,\diagcov)$, consisting of a finite chain $(\Delta,\leq)$ and a subset of the covering relation, ${\diagcov}\subseteq{\prec}$, i.e.,  if $a\diagcov b$, then $a$ is covered by $b$.
   Given a chain $(\Delta, \leq)$, a partial function $g$ over $\Delta$ is called \emph{order-preserving} if, for all $a,b\in Dom(g)$, $a\leq b$ implies $g(a)\leq g(b)$; such a $g$ is intended as the restriction of an $f \in \m F(\m \Omega)$ to $\m \Delta$.
    A \emph{diagram} $(\m \Delta, g_{1},\ldots, g_{l})$ consists of a c-chain $\m \Delta$ and  order-preserving partial functions $g_{1},\ldots,g_{l}$ on $(\Delta, \leq)$, where $l\in\mathbb{N}$.

 Given  a c-chain $(\Delta,\leq,\diagcov)$ and
an order-preserving partial function $g$ on $\Delta$, we define the relation $g^{[\ell]}$ by: for all $x, b\in\Delta$, $(x,b) \in g^{[\ell]}$ iff $b\in Dom(g)$ and there exists $a\in Dom(g)$ such that $a\diagcov b$ and $g(a)<x\leq g(b)$. 
We also define the relation $g^{[r]}$ by: for all $x, a\in\Delta$, $(x,a) \in g^{[r]}$ iff $a\in Dom(g)$ and there exists $b\in Dom(g)$, such that $a\diagcov b$ and $g(a)\leq x< g(b)$. %If  $a\diagcov b$, then we write $a=b-1$ and $a+1=b$; \textcolor{blue}{we will not use this notation more generally for $a \prec b$.} Note that in this case, $\pi_i(a+1)=\pi_1(a)$ and $\pi_1(b-1)=\pi_1(b)$.
 The intention is that $g^{[\ell]}$ will correspond to $f^\ell$, for $f \in \m F(\m \Omega)$, and the covering relation $\diagcov$ is crucial for this correspondence and for the correct calculation of the two inverses.

\begin{lemma}\textnormal{\cite{GG}}\label{l: g^[ell] order preserving}
If $\m \Delta$ is a c-chain and $g$ is an order-preserving partial function on $\m \Delta$, then $g^{[\ell]}$ and $g^{[r]}$ are order-preserving partial functions. 
Therefore, for all $x, b\in\Delta$, 
    $g^{[\ell]}(x)=b$ iff: $b, b-1\in Dom(g)$, $b-1\diagcov b$ and $g(b-1)<x\leq g(b)$.
\end{lemma}

Given a c-chain $\m \Delta$ and an order-preserving partial function $g$ on $\m \Delta$, we define $g^{[n]}$, for all $n \in \mathbb{N}$, recursively by:  $g^{[0]}=g$ and $g^{[k+1]}:=(g^{[k]})^{[\ell]}$. Also,  we define $g^{[-n]}$, for all $n \in \mathbb{N}$, recursively by $g^{[-(k+1)]}:=(g^{[k]})^{[r]}$. Lemma~\ref{l: g^[ell] order preserving} shows that if $g$ is a partial function on $\m \Delta$, then $g^{[n]}$ is a partial function for all $n\in \mathbb{Z}$.

\medskip

We will now define the notion of $n$-periodicity for a partial function $g$ in a diagram $\m \Delta$, which is intended to  capture the fact that, when $\Delta$ is identified with a subset of $\mathbb{Z}$, $g$ is the restriction to $\Delta$ of an $n$-periodic function of $\mathbb{Z}$, i.e., a function in $\m F_n(\mathbb{Z})$. As $\Delta$ might not be convex in $\mathbb{Z}$, the information of the spacing between elements of $\Delta$ is lost and only their relative ordering is retained (together with some coverings that are needed for the correct calculation of iterated inverses $g^{[m]}$). Given $\m \Delta$ and $g$, whether $g$ deserves to be called $n$-periodic depends on the existence of such a spacing (which is provided by a way of viewing $\m \Delta$ inside $\mathbb{Z}$). This leads to the following definitions. 

A \emph{spacing embedding} of c-chains is an injection that preserves the order and the covering relations. Note that such an embedding also reflects the order (due to the fact that we are working with chains) but may not reflect the covering relation. Given a spacing embedding $e: \m \Delta_1\rightarrow \m \Delta_2$ and a partial function $g$ on a c-chain $\m \Delta_1$, the \emph{counterpart} of $g$ in $\m \Delta_2$ is defined to be the partial map $g^e=:e\circ g \circ e^{-1}$. When $\Delta_2=\mathbb{Z}$, spacing embeddings that are shifts of each other via automorphisms of $\mathbb{Z}$ are essentially equivalent for our purposes. So we will consider only spacing embeddings where the least element of their image is $0$. The  \emph{height} of a spacing embedding $e: \m \Delta \rightarrow (\mathbb{Z}, \leq_{\mathbb{Z}},\diagcov_{\mathbb{Z}})$ is then defined to be  $\max(e[\Delta])$, i.e., the size of the convexification of the image $e[\Delta]$ in $\mathbb{Z}$  (minus 1).
The following definition of $n$-periodicity is inspired by the equivalences in Lemma~\ref{l:N-periodic}.

A  %order-preserving \CHAT{I}{Order preservation follows from n-periodicity}
partial function $g$ on a c-chain $\m \Delta$ is called \emph{$n$-periodic} with respect to  a spacing embedding $e: \m \Delta \rightarrow (\mathbb{Z}, \leq_{\mathbb{Z}},\diagcov_{\mathbb{Z}})$, where $n\in \mathbb{Z}^+$, if ($\min (e[\Delta])=0$ and)
for all $x,y\in Dom(g^e)$ and $k\in\mathbb{Z}$, 
$$x\leq_{\mathbb{Z}} y+kn \Rightarrow  g^e(x)\leq_{\mathbb{Z}} g^e(y)+ kn.$$
or equivalently, because $\mathbb{Z}$ is a chain, and by setting $k:=-k$, 
$$g^e(y)<_{\mathbb{Z}} g^e(x)+ kn \Rightarrow  y<_{\mathbb{Z}} x+kn.$$
Therefore, $g$ is $n$-periodic with respect to $e$ iff $g^e$ is $n$-periodic with respect to the identity map on $\mathbb{Z}$. Observe that given that $e$ preserves and reflects the order, $n$-periodic partial functions are also order preserving.

%An  order-preserving partial function $g$ on a c-chain $(\mathbb{N}_d, \leq,\diagcov)$ is called \emph{$n$-periodic}, where $n\in \mathbb{Z}^+$, if for all $x,y\in Dom(g)$ and $k\in\mathbb{Z}$, $$x\leq_{\mathbb{Z}} y+kn \Rightarrow g(x)\leq_{\mathbb{Z}} g(y)+ kn.$$

 A main issue with this definition is that it provides no insight on whether it is possible to consider only finitely-many spacing embeddings or, given such an embedding,  whether it suffices to consider only finitely-many $k$'s in the $n$-periodicity condition.   Actually, Lemma~\ref{l:extention of counterparts} below shows that this definition merely translates, into a more palatable form, the demand that there exists a counterpart of $g$ that extends to an $n$-periodic function of $\mathbb{Z}$. %; nevertheless, this crude definition will suffice for the results in this section. 
In Lemma~\ref{l:check periodicity}  we  show that checking finitely-many $k$'s suffices. In order to obtain decidability results, we  later also prove that finitely-many  spacing embeddings suffice.

\begin{lemma}\label{l:check periodicity} 
If $\m \Delta$ is a c-chain,  $g$ is a partial function on $\m \Delta$ and $e:\m \Delta\rightarrow\mathbb{Z}$ is a spacing embedding of height $d$, then  $g$ is $n$-periodic with respect to $e$ iff for all %$k\in\mathbb{Z}$ with 
$x,y\in Dom(g^e)$
and $|k|\leq \lceil d/n \rceil$ 
we have $x \leq y+kn \Rightarrow  g^e(x)\leq g^e(y)+ kn$.
% \CHAT{N}{State and prove a lemma that given a spacing embedding we may restrict $k$ to $\lceil q/n \rceil$.}  
\end{lemma}
\begin{proof} 
For the non-trivial (backward) direction, 
%Suppose now that for all $k\in\mathbb{Z}$ such that $|k|\leq \lceil d/n \rceil$ where $d$ is the height of $e$, for all $x,y\in Dom(g^e)$: $$x\leq y+kn \Rightarrow  g^e(x)\leq g^e(y)+ kn$$
we assume that $x\leq y+kn$, $x,y\in Dom(g^e)$, and $k\in\mathbb{Z}$ with $k> \lceil d/n \rceil$. %. If $|k|\leq \lceil d/n \rceil$, we know that $ g^e(x)\leq g^e(y)+ kn$. Assume now that 
Since, $x\leq y+d\leq y+\lceil d/n \rceil n\leq y+kn$,
we get $g^e(x)\leq g^e(y)+\lceil d/n \rceil n\leq g^e(y)+kn$, by the $n$-periodicity of $g$ for the value $\lceil d/n \rceil$. If $k< -\lceil d/n \rceil$ and $x,y\in Dom(g^e)$, then  $-d\leq x-y$, so $y+kn<y-\lceil d/n \rceil n\leq  y-d \leq x$; hence the condition is holds vacuously.
%Therefore, $g^e$ is $n$-periodic.
%\CHAT{I}{New}
\end{proof}

For convenience, for a fixed $n\in\mathbb{Z}^+$, for every $x\in\mathbb{Z}$ we denote by $Rx$ the remainder and by  $Qx$ the quotient of dividing $x$ by $n$. So, $x-Rx=Qx\cdot n$ and we denote this value by $Sx$.
%is equal to $n$ times the quotient of dividing $x$ by $n$ (and not the quotient itself). 

\begin{lemma}\label{l:extention of counterparts}
If  $\m{\Delta}$ is a c-chain, $g$ is
a  partial function on $\Delta$ and $e$ is a spacing embedding $e$ on $\m \Delta$, then $g$ is $n$-periodic with respect to $e$ iff its counterpart $g^e$  can be extended to a function in $F_n(\mathbb{Z})$. 
\end{lemma}
%\CHAT{I}{New}

\begin{proof}
First note that the backward direction holds by Lemma~\ref{l:N-periodic}. Note that the empty partial function is vacuously $n$-periodic, so we will assume that the partial functions are non-empty.
%Suppose $g:\Delta\rightarrow\Delta$ is a partial function, $n$-periodic by the spacing embedding, $e:\Delta\rightarrow\mathbb{Z}$. 
For the forward direction, we note that $h:=g^e$ is an $n$-periodic partial function  on $\mathbb{Z}$ with respect to the identity spacing embedding. We define the partial function  $\hat{h}$ on $\mathbb{Z}$ by $\hat{h}(Rx)=h(x)-Sx$ for all $x\in Dom(h)$, hence that $Dom(\hat{h}) \subseteq [0,n)_\mathbb{Z}$; we will show that $\hat{h}$ is  well-defined (single-valued) and $n$-periodic at the same time. If $x,y\in Dom(h)$, $k\in\mathbb{Z}$ and $Rx\leq Ry+kn$, then $Rx+Sx\leq Ry+Sy+(Sx-Sy)+kn$, so $x\leq y+kn+(Sx-Sy)$;
%\begin{align*}
%    Rx+Sx\leq& Ry+Sy+(Sx-Sy)+kn\\
%    x\leq& y+kn+(Sx-Sy)
%\end{align*}
since $Sx$ and $Sy$ are multiples of $n$,  the periodicity of $h$ yields $h(x)\leq h(y)+kn+(Sx-Sy)$, so $(h(x)-Sx)\leq (h(y)-Sy)+kn$. By applying this for $k=0$ (and for two inequalities simultaneously) we get $Rx=Ry \Rightarrow (h(x)-Sx)= (h(y)-Sy)$. By applying this for arbitrary $k$ we get   $Rx\leq Ry+kn \Rightarrow \hat{h}(Rx)\leq \hat{h}(Ry)+kn$.

We extend $\hat{h}$ to a function $f:[0,n)_\mathbb{Z}\rightarrow\mathbb{Z}$ on the whole period, % $[0,n)= \{0,\ldots,n-1\}$, 
given by $f(x)=\hat{h}(a_x)$, %for all $x\in\mathbb{Z}$, 
where $a_x$ is the smallest $a\in Dom(\hat{h})$ with $x\leq a$, if such an $a$ exists, and $a_x$ is the biggest  $a\in Dom(\hat{h})$, otherwise.
%$J_x=\{b:b\in Dom(\hat{h})\text{ and }x\leq b\}$ and
% \begin{align*}
%     a_x=\begin{cases}
%         \bigwedge J_x &\text{if } J_x\neq\emptyset\\
%         \bigvee J_x^c &\text{if } J_x=\emptyset
%     \end{cases}
% \end{align*}
% Where $J_x^c$ is the complement of $J_x$ respect to $Dom(\hat{h})$. 
To verify that $f$ is $n$-periodic with respect to the identity, by Lemma~\ref{l:check periodicity} it is enough to check the $n$-periodicity condition for $k\in \{-1,0,1\}$; since $x\nleq y-n$ for all $x,y\in[0,n)_\mathbb{Z}$, we consider only $k\in \{0,1\}$. If $x,y\in[0,n)_\mathbb{Z}$ and  $x\leq y+n$, since $a_x,a_y\in[0,n)$ we have $a_x\leq a_y+n$, so $\hat{h}(a_x)\leq\hat{h}(a_y)+n$ by the $n$-periodicity of $\hat{h}$, hence $f(x)\leq f(y)+n$. Also, for all $x,y\in[0,n)_\mathbb{Z}$, if $x\leq y+0$, then $a_x\leq a_y$, so $\hat{h}(a_x)\leq\hat{h}(a_y)$, hence $f(x)\leq f(y)$. Therefore $f$ is a partial function on $\mathbb{Z}$ that is $n$-periodic with respect to the identity.
 
 Finally, we extend $f$ to a total function $\hat{f}$ on  $\mathbb{Z}$, by %"coping" the partial function $f$, that is $\hat{f}:\mathbb{Z}\rightarrow\mathbb{Z}$ given by
 $\hat{f}(x)=f(Rx)+Sx$  for all $x\in\mathbb{Z}$; note that $\hat{f}$ extends $f$, since $Rx=x$ and $Sx=0$, for all $x \in [0,n)_\mathbb{Z}$. To show that $\hat{f}\in F_n(\mathbb{Z})$, we will verify the conditions of Lemma~\ref{l: F_N(Z)}. For condition (2), if $x,y,k\in\mathbb{Z}$ and $x\leq y+kn$, then $Rx\leq Ry+Sy-Sx+kn$, so by the $n$-periodicity of $f$ we have
 $  f(Rx)\leq f(Ry)+Sy-Sx+kn$, so $f(Rx)+Sx\leq f(Ry)+Sy+kn$, i.e.,
$\hat{f}(x)\leq\hat{f}(y)+kn$. In particular, $\hat{f}$ is order preserving.
For condition (1), if $b\in\mathbb{Z}$ and $x\in\hat{f}^{-1}[\hat{f}(b)]$, then $\hat{f}(x)=\hat{f}(b)$,  so $\hat{f}(b)-n<\hat{f}(x)<\hat{f}(b)+n$. Also, $\hat{f}(b\pm n)=f(R(b\pm n))+S(b\pm n)=f(Rb)+Sb\pm n=\hat{f}(b)\pm n$,  so $\hat{f}(b-n)<\hat{f}(x)<\hat{f}(b+n)$, and by order preservation $b-n<x<b+n$. Therefore, 
$\hat{f}^{-1}[\hat{f}(b)]\subseteq [b-n,b+n]$. Consequently, $\hat{f}\in F_n(\mathbb{Z})$.

%\CHAT{N}{Maybe we should explain how the domain and ranges of the previous extensions work to get to condition (1)} 
% \begin{align*}
%     f(Rx)&\leq f(Ry)+Sy-Sx+kn\\
%     f(Rx)+Sx&\leq f(Ry)+Sy+kn\\
%     \hat{f}&\leq\hat{f}+kn
% \end{align*}
To see that $\hat{f}$ extends $h$, note that if $x\in Dom(h)$ then $Rx\in Dom(\hat{h})$ and $\hat{h}(Rx)=h(x)-Sx$. By definition of $\hat{f}$ and given that $f$ extends $\hat{h}$, we have $\hat{f}(x)=f(Rx)+S(x)=\hat{h}(Rx)+Sx=h(x)-Sx+Sx=h(x)$.
% and clearly $\hat{f}$ extends $h$.
\end{proof}

The next lemma shows that $^{[\ell]}$ and $^{[r]}$ keep us within the setting of  $n$-periodicity.

\begin{lemma}\label{l: ellnperiodic}
If $\m \Delta$ is a c-chain and $g$ is partial function on $\m \Delta$ that is $n$-periodic with respect to a spacing embedding $e$, then the partial functions $g^{[\ell]}$ and $g^{[r]}$ are $n$-periodic  on $\m \Delta$ with respect to $e$. Moreover, $(g^{[\ell]})^e\subseteq (g^e)^{[\ell]}$ and  $(g^{[r]})^e\subseteq (g^e)^{[r]}$. 
\end{lemma}

\begin{proof}
%By Lemma 3.4 of \cite{GG}, 
By Lemma~\ref{l: g^[ell] order preserving} we know that $g^{[\ell]}$ and $g^{[r]}$ are partial functions over $\m \Delta$. Let $h:=g^e$ be the counterpart of $g$ with respect to a spacing embedding $e$ witnessing the  $n$-periodicity of $g$; then $h$ is $n$-periodic with respect to $id_\mathbb{Z}$.

Let $x,y\in Dom(h^{[\ell]})$, and  set $a:=h^{[\ell]}(x)$ and $b:=h^{[\ell]}(y)$. Then, by definition, $a,a-1,b, b-1\in Dom(h)$, $h(a-1)<x\leq h(a)$ and $h(b-1)<y\leq h(b)$.
If $x\leq y+kn$ for some $k\in\mathbb Z$, then  $h(a-1)<x\leq  y+kn\leq h(b)+kn$, so by the $n$-periodicity of $h$, we get $a-1< b+kn$.  Consequently, $h^{[\ell]}(x)=a\leq b+kn= h^{[\ell]}(y)+kn$. 
So $h^{[\ell]}$ is an $n$-periodic partial function  with respect to $id_\mathbb{Z}$; in particular, any restriction of it is also $n$-periodic.

To show that $(g^{[\ell]})^e=h^{[\ell]}|_{Dom ((g^{[\ell]})^e)}$, note that if $x\in Dom((g^{[\ell]})^e)$, then $e^{-1}(x)\in Dom(g^{[\ell]})$, so there exist $a,a-1\in Dom(g)$, such that $g(a-1)<e^{-1}(x)\leq g(a)$ and $g^{[\ell]}(e^{-1}(x))=a$. Hence, $ege^{-1}(e(a)-1)=ege^{-1}(e(a-1))<x\leq ege^{-1}(e(a))$, i.e.,  $h(e(a)-1)<x\leq h(e(a))$, thus $h^{[\ell]}(x)=e(a)=eg^{[\ell]}(e^{-1}(x))=(g^{[\ell]})^{e}(x)$. 
%And since $h^{[\ell]}$ is $n$-periodic  with respect to $id_\mathbb{Z}$, hence $g^{[\ell]}$ is $n$-periodic  with respect to $e$.
Similarly, we can show that $(g^{[r]})^e\subseteq (g^e)^{[r]}$. % and therefore $g^{[r]}$ is also $n$-periodic.
\end{proof}

\subsection{The join}

 By \cite{GJ}, for every $n$, the variety $\mathsf{LP_n}$ of $n$-periodic $\ell$-pregroups is a subvariety of  $\mathsf{DLP}$. Here we prove that the join of all of the $\mathsf{LP_n}$'s is precisely the whole variety $\mathsf{DLP}$; in other words the class of periodic $\ell$-pregroups generates the variety of all distributive $\ell$-pregroups.

It is shown in \cite{GG} that  every equation in the language of $\ell$-pregroups is equivalent to one of the form $1\leq w_{1}\vee\ldots\vee w_{k}$
 where $m \in \mathbb{Z}^+$ and the $w_i$'s are \emph{intensional terms}, i.e.,  terms of the form $$x_{1}^{(m_1)}x_{2}^{(m_2)}\ldots x_{l}^{(m_l)},$$
 where $x_1,\ldots, x_l$ are not necessarily distinct variables, $l\in \mathbb{N}$, and $m_1, \ldots, m_l \in \mathbb{Z}$. Therefore, we may  consider equations in such \emph{intensional} form.

Recall that a failure of an equation $1\leq w_{1}\vee\ldots\vee w_{k}$ in $\m F(\m \Omega)$ consists of a homomorphism $\varphi: \m {Tm} \rightarrow \m F(\m \Omega)$, from the algebra $\m {Tm}$ of all terms in the language of $\ell$-pregroups, such that  $\varphi(1)(p) > \varphi( w_{1})(p), \ldots  , \varphi(w_{k})(p)$, for some $p \in \Omega$.

 The failure of an equation in a diagram is formulated relative to an \emph{intensional algebra}---an algebra over the language $(\cdot, 1, ^\ell, ^r)$---playing the role of $\m F(\m \Omega)$. 
 Given a c-chain $\m \Delta$, we define the algebra $\m {Pf}(\m \Delta)=(Pf(\m \Delta), {\circ}, ^{[\ell]}, ^{[r]}, i_\Delta)$, where $Pf(\m \Delta)$ is the set of all the order-preserving partial functions over $\m \Delta$, $\circ$ is  the composition of partial functions, $i_\Delta$ is the identity function on $\Delta$, and  $g \mapsto g^{[\ell]}$ and $g \mapsto g^{[r]}$ are the two inversion operations as defined on  $\m \Delta$.

The following lemma follows from the definition of counterparts and by iterations of Lemma~\ref{l: ellnperiodic}.
\begin{lemma}\label{l:e, composition and residuals}
If $\m \Delta$ is an c-chain and $g,f$ are partial functions on $\m \Delta$ that are $n$-periodic with respect to a spacing embedding $e$, then
\begin{enumerate}
    \item $(f\circ g)^e=f^{e}\circ g^{e}$ and 
    \item for all $m\in\mathbb{Z}$, $(g^{[m]})^e\subseteq (g^e)^{[m]}$.
\end{enumerate}
Moreover, for every $u\in \m{Ti}$ with $l$-many variables  and $g_1,\ldots,g_l$ partial functions on $\m \Delta$, $n$-periodic with respect to $e$, then $(u^{\m{Pf}(\m{\Delta})}(g_1,\ldots g_l))^e\subseteq u^{\m{Pf}(\mathbb{Z})}(g_1^e,,g_l^e)$.
\end{lemma}
\begin{proof}
    (1) is easy to see and (2) follows from Lemma~\ref{l: ellnperiodic}. We will prove the remaining inclusion by induction on the structure of terms. For the induction step, let  $u$ be a term of the form $u=x^{(m)}v$ where $m\in\mathbb{Z}$, 
    $v$ is a term in $\m{Ti}$ over variables $x_1,\ldots , x_l$,  and $x\in\{x_1,\ldots,x_l\}$. Also, let
    $g_1,\ldots,g_l$ be partial functions on $\m \Delta$ that are $n$-periodic with respect to the same spacing embedding $e$ such that $(v^{\m{Pf}(\m{\Delta})}(g_1,\ldots g_l))^e\subseteq v^{\m{Pf}(\mathbb{Z})}(g_1^e,\ldots ,g_l^e)$; we will write $g$ for the partial function corresponding to $x$. If $a\in Dom((u^{\m{Pf}(\m{\Delta})}(g_1,\ldots g_l))^e)$, then $a\in Dom((g^{(m)})^{e}(v^{\m{Pf}(\m{\Delta})}(g_1,\ldots g_l))^e)$ by (1), so $a\in Dom((v^{\m{Pf}(\m{\Delta})}(g_1,\ldots g_l))^e)$, in particular. So,   
    $e^{-1}a\in Dom(v^{\m{Pf}(\m{\Delta})}(g_1,\ldots g_l))$ and $v^{\m{Pf}(\m{\Delta})}(g_1,\ldots g_l)(e^{-1}a)\in Dom(g^{[m]})$. So, by (1), $(u^{\m{Pf}(\m{\Delta})}(g_1,\ldots g_l))^e(a)=(g^{[m]})^e(v^{\m{Pf}(\m{\Delta})}(g_1,\ldots g_l))^e(a)$ and by hypothesis $(g^{[m]})^e(v^{\m{Pf}(\m{\Delta})}(g_1,\ldots g_l))^e(a)=(g^{[m]})^e(v^{\m{Pf}(\mathbb{Z})}(g_1^e,\ldots g_l^e))(a)$.  By (2), $(g^{[m]})^e(v^{\m{Pf}(\mathbb{Z}}(g_1^e,\ldots g_l^e))(a)=(g^e)^{[m]}(v^{\m{Pf}(\mathbb{Z})}(g_1^e,\ldots g_l^e))(a)$, therefore, $(u^{\m{Pf}(\m{\Delta})}(g_1,\ldots g_l))^e(a)
    =u^{\m{Pf}(\mathbb{Z})}(g_1^e,\ldots,g_l^e)(a)$.
\end{proof}
For $n\in \mathbb{Z}^+$, a diagram is called \emph{$n$-periodic} with respect to some spacing embedding if all of its partial functions are $n$-periodic with with respect to that embedding. A diagram is called \emph{$n$-periodic} if it is  $n$-periodic with respect to some spacing embedding.

We say that the equation $1\leq w_{1}\vee\ldots\vee w_{k}$ in intensional form over variables $x_1, \ldots, x_l$  \emph{fails} in a diagram  $(\m \Delta, f_{1},\ldots,f_{l})$ if there is an intensional homomorphism $\varphi: \m {Ti} \rightarrow \m {Pf}(\m \Delta)$, from the algebra $\m {Ti}$ of all intensional terms, and a point $p \in \Delta$ such that   $\varphi(1)(p) > \varphi( w_{1})(p), \ldots  , \varphi(w_{k})(p)$ and $\varphi(x_i)=f_i$ for all $1 \leq i\leq l$.

In the following, we denote the lenght of an equation $\varepsilon$ by $|\varepsilon|$.

\begin{lemma} \label{l: diag2nperdiag}
If an equation $\varepsilon$ fails in a diagram, then it fails in an $n$-periodic diagram, where $n=2^{|\varepsilon|}|\varepsilon|^4$. % and $|\varepsilon|$ is the length of $\varepsilon$.
\end{lemma}

\begin{proof}
In the proof of Corollary~4.4 of \cite{GG} it is shown that if an equation $\varepsilon$ fails in a diagram, then it also fails in a diagram where the c-chain has size $n=2^{|\varepsilon|}|\varepsilon|^4$.  

Note that every diagram based on a c-chain $\m \Delta$ is $n$-periodic for $n\geq |\Delta|$, as follows.  For every partial function $g$ of the diagram, we consider the (unique) spacing embedding $e: \Delta \rightarrow \mathbb{Z}$ where $e[\Delta]=\mathbb{N}_d$ and $d=|\Delta|$. Then the partial function $g^e$ is defined on the convex subset $\mathbb{N}_d$ of $\mathbb{Z}$, so it can be extended to a partial function with domain $\mathbb{N}_n$ (in any order-preserving way) since $d \leq n$. This order-preserving partial function is vacuously $n$-periodic, so it can be further extended periodically to an $n$-periodic function on  $\mathbb{Z}$ by Lemma~\ref{l:extention of counterparts}. Thus, every partial function of the diagram is $n$-periodic, hence the equation fails in an $n$-periodic diagram.
\end{proof}

The following lemma will help us show that if an equation $\varepsilon$ fails in a distributive $\ell$-pregroup, then it fails in a periodic one.

\begin{lemma} \label{l: nperdiagram2FnZ}
If an equation fails in an $n$-periodic diagram, then it fails in $\m F_n(\mathbb{Z})$.
\end{lemma}

\begin{proof}
 Suppose that an equation $1\leq w_{1}\vee\ldots\vee w_{t}$, in intensional form, fails in a diagram
 $(\m \Delta, g_{1},\ldots,g_{l})$ that is $n$-periodic with respect to a spacing embedding $e: \m \Delta \rightarrow (\mathbb{Z}, \leq_{\mathbb{Z}},\diagcov_{\mathbb{Z}})$, where $\min (e[\Delta])=0$. This means that there exist $p\in \Delta$ and an intentional homomorphism $\varphi: \m {Ti} \rightarrow \m {Pf}(\m \Delta)$  where $\varphi(x_i)=g_i$ for all $1\leq i\leq l$, satisfying $\varphi(1)(p) > \varphi( w_{1})(p), \ldots  , \varphi(w_{t})(p)$. For each $i\in \{1,\ldots,l\}$, by Lemma~\ref{l:extention of counterparts} the counterpart of each $g_i$ by $e$ can be extended to a function $f_i\in F_n(\mathbb{Z})$.
% For each $i\in \{1,\ldots,l\}$, we consider the $n$-periodic partial function $g_i:=h_i^e$ on $\mathbb{Z}$ and define the partial function  $\hat{g}_i$ on $\mathbb{Z}$ by $\hat{g}_i(x)=g_i(x)-S(x)n$ for all $x\in Dom(g_i)$. Observe that by construction $\hat{g}$ is also $n$-periodic. Now, let us extend periodically $\hat{g_i}$. Now consider $f_i$ the extension of $g_i$ given by $f(x)=\bigwedge J_x$ for all $x\in\mathbb{Z}$, where $J_x=\{b:b\in Dom(\hat{g}_i)\text{ and }x\leq b\}$.

 %Given that $\varphi: \m {Ti} \rightarrow \m {Pf}(\m \Delta)$ is an
 
 We extend the assignment  $\hat{\varphi}(x_i)=g^e_i$ to an 
 intentional homomorphism $\hat{\varphi}: \m {Ti} \rightarrow \m {Pf}(e[\m \Delta])$; by Lemma~\ref{l:e, composition and residuals} have $\varphi(u)^e e(p)=\hat{\varphi}(u)e(p)$, for all $u\in FS$. It follows from the last paragraphs of Theorem~4.1 in \cite{GG} %, it was shown that if $f_i$ is an order-preserving extension of a partial function $g_i^e$ to all of $\mathbb{Z}$ and $\tilde{\varphi}:\m {Ti} \rightarrow \m {Pf}(e[\m \Delta])$ is an intentional homomorphism extending the assignment $\tilde{\varphi}(x_i)=g_i$, then 
 that there is a homomorphism ${\psi}:\m {Ti} \rightarrow \m F(\mathbb{Z})$ extending the assignment $\psi(x_i)=f_i$ and satisfying $\psi (u)(e(p))=\hat{\varphi}(u)(e(p))$ for all $u\in FS$. Since the extensions $f_1,\ldots,f_l$ are in $F_n(\mathbb{Z})$, the range of $\psi$ is in $\m {F}_n(\mathbb{Z})$. Also $\varepsilon$ fails in $\m F_n(\mathbb{Z})$ since for all $i\in \{1,\ldots, k\}$, we have $\psi(w_i)(e(p))=e\varphi(w_i)(p)<\psi(1)(e(p))=e\varphi(1)(p)$. 
\end{proof}

%Also, we can get the following result by specializing to a particular equation, by combining Lemma~\ref{l: diag2nperdiag} and Lemma~\ref{l: nperdiagram2FnZ}.

\begin{theorem} \label{t: DLP FnZ}
An equation $\varepsilon$ fails in $\mathsf{DLP}$ iff it fails in $\m F_n(\mathbb{Z})$, where $n=2^{|\varepsilon|}|\varepsilon|^4$.
\end{theorem}

\begin{proof}
If an equation fails in $\mathsf{DLP}$, then by \cite{GG} it fails in a diagram. By Lemma~\ref{l: diag2nperdiag}, the equation fails in an $n$-periodic diagram, where $n=2^{|\varepsilon|}|\varepsilon|^4$. % and $|\varepsilon|$ is the length of $\varepsilon$.
 By Lemma~\ref{l: nperdiagram2FnZ}, the equation fails in $\m F_n(\mathbb{Z})$. On the other hand, if $\varepsilon$ fails in $\m F_n(\mathbb{Z})$, then $\varepsilon$ fails in $\mathsf{DLP}$, since $\m F_n(\mathbb{Z}) \in \mathsf{DLP}$.
\end{proof}

This reduces the decidability of the equational theory of $\mathsf{DLP}$ to the decidability of all of the $\m F_n(\mathbb{Z})$'s: given an equation to check, we know which $\m F_n(\mathbb{Z})$ to test it in. In Section~\ref{s: F_nZ} we will prove that $\m F_n(\mathbb{Z})$ is decidable, for each $n$.

 The following theorem shows that $\mathsf{DLP}$ can be expressed as a join in two different ways. 
% In addition to showing that $\mathsf{DLP}$ is the join of all of the $\mathsf{LP_n}$'s,  of all of the varieties  $\mathsf{V}(\{\m F_n(\mathbb{Z})\})$, as well. 

\begin{theorem} \label{t: joinnper}
$\mathsf{DLP}=\bigvee \mathsf{LP_n}=\bigvee \mathsf{V}(\{\m F_n(\mathbb{Z})\})=\mathsf{V}(\{\m F_n(\mathbb{Z}):n\in \mathbb{Z}^+\})$.
\end{theorem}

\begin{proof}
%If an equation fails in $\mathsf{DLP}$, then by \cite{GG} it fails in a diagram. By Lemma~\ref{l: diag2nperdiag}, the equation fails in an $n$-periodic diagram. %, where $n=2^{|\varepsilon|}|\varepsilon|^4$ and $|\varepsilon|$ is the length of $\varepsilon$.  By Lemma~\ref{l: nperdiagram2FnZ}, the equation fails in $\m F_n(\mathbb{Z})$.
By Lemma~\ref{t: DLP FnZ}, we have $\mathsf{DLP} \subseteq \mathsf{V}(\{\m F_n(\mathbb{Z}):n\in \mathbb{Z}^+\})$. Furthermore, by general facts we have $ \mathsf{V}(\{\m F_n(\mathbb{Z}):n\in \mathbb{Z}^+\})=\bigvee \mathsf{V}(\{\m F_n(\mathbb{Z})\})$. Since each algebra $\m F_n(\mathbb{Z})$ is $n$-periodic, we get $\mathsf{V}(\{\m F_n(\mathbb{Z})\}) \subseteq \mathsf{LP_n}$, so $\bigvee \mathsf{V}(\{\m F_n(\mathbb{Z})\}) \subseteq \bigvee \mathsf{LP_n}$. Finally, by \cite{GJ}, for every $n$,  $\mathsf{LP_n} \subseteq \mathsf{DLP}$, so $\bigvee \mathsf{LP_n} \subseteq \mathsf{DLP}$.
\end{proof}

%Theorem~\ref{t: joinnper}
This complements nicely the result of \cite{GG} that $\mathsf{DLP}$ is generated by $\m F(\mathbb{Z})$. % and also by its subalgebra $\m F_{fs}(\mathbb{Z})$ of functions with finite support. 

\subsection{From the join down to the varieties?}

 It follows from Theorem~\ref{t: joinnper} that $\bigvee \mathsf{LP_n}=\bigvee \mathsf{V}(\m F_n(\mathbb{Z}))$. For each $n$, we want to find a generating algebra for $\mathsf{LP_n}$, so it is tempting to conjecture that $\mathsf{LP_n}=\mathsf{V}(\m F_n(\mathbb{Z}))$, for all (or at least some) $n$. 
This fails for  $n=1$, since $\mathsf{LP}_1$ is the variety of $\ell$-groups while $\m F_1(\mathbb{Z})$ is isomorphic the $\ell$-pregroup on the integers (it consists of only the translations on the integers) and 
 $\mathsf{V}(\m F_1(\mathbb{Z}))$ is the variety of abelian $\ell$-groups. Corollary~\ref{c: notFnZ} below shows that this actually fails for every single $n$, i.e.,  for all $n$, the algebra $\m F_n(\mathbb{Z})$ generates a proper subvariety of $\mathsf{LP_n}$. 

 In other words, it will follow that if an equation fails in a distributive $\ell$-pregroup, then  it fails in $\mathsf{LP_n}$ and it fails in $\mathsf{V}(\m F_m(\mathbb{Z}))$ for some  $m,n$ (which can be then taken to be minimal such), and also that $n \leq m$ (since $\mathsf{V}(\m F_k(\mathbb{Z}))\subseteq \mathsf{LP}_k$, for all $k$), but we necessarily have  $n<m$.

 \medskip
 
 The combination of $\mathsf{V}(\m F_n(\mathbb{Z})) \not = \mathsf{LP_n}$, for all $n$, and $\bigvee \mathsf{V}(\m F_n(\mathbb{Z}))= \bigvee \mathsf{LP_n}$ is quite interesting and surprising. This means that for each $n$, the algebra $\m F_n(\mathbb{Z})$ satisfies a special equation, but there is no special equation that is satisfied by all of the $\m F_n(\mathbb{Z})$'s. We prove that $\mathsf{V}(\m F_n(\mathbb{Z})) \not = \mathsf{LP_n}$, for each $n$, by providing such an equation tailored to each $n$. 

 The equation states that all invertible elements commute. Clearly, the invertible elements in an $\ell$-pregroup are the $1$-periodic elements and, by Lemma~\ref{l: F_N(Omega)}, they form a subalgebra, which is therefore an $\ell$-group; this is the maximal subalgebra that is an $\ell$-group and we call it the the \emph{maximal $\ell$-subgroup} of the $\ell$-pregroup. Therefore, in an $\ell$-pregroup the inverible elements commute  iff its maximal $\ell$-subgroup is abelian. It is easy to see that an element $g$ is invertible iff $g^\ell = g^r$. Therefore,  the the property that all invertibles commute can be captured by the quasiequation 
 $$(x^\ell= x^r\; \& \;y^\ell= y^r) \Rightarrow xy=yx.$$
 The lemma below entails that this quasiequation is not equivalent to any equation. However, when restricted to $n$-periodic $\ell$-pregroups, for a fixed $n$, this can be captured by the equation 
$$(x \jn x^{\ell \ell} \jn \cdots \jn  x^{\ell^{2n-2}})(y \jn y^{\ell \ell} \jn \cdots \jn  y^{\ell^{2n-2}})=(y  \jn \cdots \jn  y^{\ell^{2n-2}})(x  \jn \cdots \jn  x^{\ell^{2n-2}})$$
For convenience, for all $n$, we define the term $i_n(x):=x \jn x^{\ell \ell} \jn \cdots \jn  x^{\ell^{2n-2}}$.

\begin{lemma}\label{l: invnper}
For every $n$, and $n$-periodic $\ell$-pregroup $\m A$, the following hold.
\begin{enumerate}
    \item The invertible elements  of the algebra $\m A$ are exactly the ones of the form  $i_n(a)=a \jn a^{\ell \ell} \jn \cdots \jn  a^{\ell^{2n-2}}$, where $a \in A$.
    \item The invertible elements of the algebra $\m A$ commute iff $\m A$ satisfies the equation $i_n(x)i_n(y)=i_n(y)i_n(x)$.
    \item The equation $i_n(x)i_n(y)=i_n(y)i_n(x)$ holds in $\m F_n(\mathbb{Z})$ but fails in $\mathsf{LP_n}$. 
    %\item $\m F_n(\mathbb{Z})$ is abelian, but $\mathbf{F}_n(\mathbb{S} \overrightarrow{\times}\mathbb{Z})$ is not.
\end{enumerate}
\end{lemma}

\begin{proof}
(1) Since $\m A$ is $n$-periodic, it satisfies $x^{\ell^{2n}}=x$, so for all $a \in A$, we have
$(a \jn a^{\ell \ell} \jn \cdots \jn  a^{\ell^{2n-4}} \jn  a^{\ell^{2n-2}})^{\ell\ell}=a^{\ell\ell} \jn a^{\ell^4} \jn \cdots \jn  a^{\ell^{2n-2}} \jn  a^{\ell^{2n}}
=a^{\ell\ell} \jn a^{\ell^4} \jn \cdots \jn  a^{\ell^{2n-2}} \jn  a$. Therefore, $a \jn a^{\ell \ell} \jn \cdots \jn  a^{\ell^{2n}}$ is invertible. Conversely, if $a$ is invertible, then $a \jn a^{\ell \ell} \jn \cdots \jn  a^{\ell^{2n}}=a$, so it is of this form. (2) follows from (1).

(3) The only invertible elements of $\m F_n(\mathbb{Z})$ are the translations, so they commute with each other; thus $\m F_n(\mathbb{Z})$ satisfies the equation. However, there exist non-commutative $\ell$-groups, such as the group $\m F_1(\mathbb{Q})$ of all order-preserving permutations on $\mathbb{Q}$. All $\ell$-pregroups are $1$-periodic, hence also $n$-periodic. Thus, $\m F_1(\mathbb{Q})$ is in $\mathsf{LP_n}$, but not all of its invertible elements commute, since all of its elements are invertible.
%We define the functions $b$ and $d$ in $\mathbf{F}_n(\mathbb{Q} \overrightarrow{\times}\mathbb{Z})$ by \begin{center}     $b(0,m)=(0,m+1)$ and $b(j,m)=(j,m)$ for $j\not=1$. \end{center}and $d(j,m)=(j+1, m)$, both of which are invertible. However, $bd \not = db$, since $bd(0,0)=b(1,0)=(1,0)$ and $db(1,0)=d(1,0)=(2,0)$.
\end{proof}

%We will prove that for each $n$,  $\m F_n(\mathbb{Z})$ has a decidable equational theory.

Actually, even though by the following corollary $\m F_n(\mathbb{Z})$ fails to generate $\mathsf{LP_n}$, our second representation theorem (Theorem~\ref{t: repnper}) that every $n$-periodic $\ell$-pregroup can be embedded into a wreath product $\m H \wr \m F_n(\mathbb{Z})$, where $\m H$ is an $\ell$-group, showcases the importance of the algebra $\m F_n(\mathbb{Z})$ in studying $\mathsf{LP_n}$. In that respect, Lemma~\ref{l: nperdiagram2FnZ} (which shows that failures in $n$-periodic diagrams can be embedded in $\m F_n(\mathbb{Z})$) will turn out to be very useful in the next sections.

\begin{corollary} \label{c: notFnZ}
$\mathsf{LP_n}\not =\mathsf{V}(\m F_n(\mathbb{Z}))$, for all $n$.
\end{corollary}

For each $n$, we want to find a chain $\m \Omega_n$ such that $\mathsf{LP_n}=\mathsf{V}(\m F_n(\m \Omega_n))$. Corollary~\ref{c: notFnZ} shows that we cannot take $\m \Omega_n=\mathbb{Z}$ for any $n$. At this moment it is not clear that such an  $\m \Omega_n$ exists for any $n$ other than $n=1$: we can take $\m \Omega_1=\mathbb{Q}$ because the variety of $\ell$-groups is generated by $\m F_1(\mathbb{Q})$. Actually, the next result shows that we cannot take $\m \Omega_n=\mathbb{Q}$ for any $n>1$.

Recall that an element of a chain is called a \emph{limit point} if it is the join of all of the elements strictly below it or the meet of all of the elements strictly above it.

\begin{lemma} If $\m \Omega$ is a chain where every point is a limit point, then we get $\m F_n(\m \Omega)=\m F_1(\m \Omega)$, for all $n$. In particilar, $\m F_n(\mathbb{Q})=\m F_1(\mathbb{Q})$, for all $n$. Therefore, $\mathsf{LP_n} \not =\mathsf{V}(\m F_n(\mathbb{Q}))$ for each $n>1$.
\end{lemma}

\begin{proof}
By Lemma~2.7 of \cite{GG}, for $f \in F_n(\m \Omega)$ if the preimage of an element under $f$ has more than one element, then the element is not a limit point, a contradiction; so every preimage must have at most one element, i.e., $f$ is one-to-one. To show that it is onto, assume an element $a$ is not in the image. By Lemma~\ref{l: bounded preimage} there exists a covering pair $b \prec c$ such that $a \in (f(b), f(c))$, contradicting the fact that there are no covering pairs in $\m \Omega$. 
\end{proof}

In Section~\ref{s: From F_N(Omega) to an N-Diagram} we prove that for each $n$ there exists indeed a $\m \Omega_n$ such that $\mathsf{LP_n}=\mathsf{V}(\m F_n(\m \Omega_n))$. Actually, we do better than that: we identify a single chain $\m \Omega$ such that $\mathsf{LP_n}=\mathsf{V}(\m F_n(\m \Omega))$ for all $n$. This single/uniform chain is $\m \Omega=\mathbb{Q} \overrightarrow{\times}\mathbb{Z}$. Due to the wreath-product decomposition  $\m H \wr \m F_n(\mathbb{Z})$ of $\m F_n(\m \Omega)$, where $\m \Omega$ is integral, the analysis in Section~\ref{s: From F_N(Omega) to an N-Diagram} will benefit from the study of the algebra $\m F_n(\mathbb{Z})$, so the next section is devoted to that.

\section{Decidability for $F_n(\mathbb{Z})$}\label{s: F_nZ}

The wreath product decomposition for $\m F_n(\m \Omega)$ of Theorem~\ref{t: repnper} brings to the forefront the role of $\m F_n(\mathbb{Z})$ in understanding $n$-preriodic $\ell$-pregroups. In particular, it turns out that a lot of the notions (such as $n$-periodicity in a diagram) that will be needed for the generation result for $\mathsf{LP_n}$ already show up when studying the structure of $\m F_n(\mathbb{Z})$. In this section we study this structure and actually prove that the equational theory of $\m F_n(\mathbb{Z})$ is decidable. The results of this section will also be  the main driving force for proving that 
the equational theory of $n$-preriodic $\ell$-pregroups is decidable in Section~\ref{s: From F_N(Omega) to an N-Diagram}.

\subsection{Failures in compatible surjections}

Given an equation $\varepsilon$ in  intensional form $1\leq w_{1}\vee\ldots\vee w_{k}$ we will define a  set $\Delta_\varepsilon$ of terms in an expansion of the intensional language. %that will serve as a template for evaluations into a sufficient collection of failing diagrams for $\varepsilon$. 
First we
define the set of \emph{final subwords} of $\varepsilon$, $
 FS_\varepsilon:=\{ u: w_1=vu\text{ or  } ... \text{ or  } w_k=vu, \text{ for some } v\} 
  $.
  
  We actually take $\m {Ti}$ to satisfy the equalities $v\cdot 1=v=1 \cdot v$, so strictly speaking we take it to be a quotient of the absolutely free algebra, so $FS_\varepsilon$ contains $1$.
  The particular syntactic expressions below, serving as names for the points in $\Delta_\varepsilon$, will not be important for our results, since the work done in \cite{GG} (which we will be citing when appropriate) ensures that $\Delta_\varepsilon$ has enough points for the iterated inverses $g^{[m]}$ to be calculated correctly; we include the definition for completeness.
  
 In the following the notation $-a$ and $+a$ will be interpreted by the lower and upper cover of $a$ and $0a=a$, when evaluating the terms. For a variable $x$ among the variables $x_1, \ldots, x_n$ of  $\varepsilon$, $m\in\mathbb{N}$ and $v\in FS_\varepsilon$ we define 
  %the following sets; they are syntactic analogues of the $\Delta$'s that were defined relative to a chain $\m \Omega$ and an $f \in F(\m \Omega)$.
  \begin{align*}
 \Delta_{x,m}^{v} & :=\{ v\} \cup \bigcup_{j=0}^{m} \{ \sigma_{j}x^{(j)}\ldots\sigma_{m}x^{(m)}v:\,\sigma_j,\ldots,\sigma_{m}\in\{-,0\},\sigma_0=0\}\\
\Delta_{x,-m}^{v} & :=\{ v\} \cup \bigcup_{j=0}^{m} \{ \sigma_{j}x^{(-j)}\ldots\sigma_{m}x^{(-m)}v:\,\sigma_{j},\ldots,\sigma_{m}\in\{+,0\},\sigma_0=0\}\\
S_\varepsilon&:=\{ (i,m,v):i\in\{1,\ldots,n\},m\in\mathbb{Z},v\in FS_\varepsilon\text{ and }x_{i}^{(m)}v\in FS_\varepsilon\} \\
\Delta_{\varepsilon}&:=\{1\} \cup \underset{(i,m,v)\in S_\varepsilon}{\bigcup}\Delta_{x_{i},m}^{v}
\end{align*}

%Observe that if $w \in FS_\varepsilon$, then $w=1$ (in which case it is in $\Delta_\varepsilon$) or it is of the form $x_i^{(m)}v$, where $v \in FS_\varepsilon$, $i \in \{1, \ldots, n\}$ and $m \in \mathbb{Z}$. So, $x_i^{(m)}v \in \Delta_{x_i,m}^{v}$ (it appears in the formula for $\Delta_{x_i,m}^{v}$ as part of the union for  $j=m$ and $\sigma_m=0$). Therefore, $FS_\varepsilon \subseteq \Delta_\varepsilon$. When the equation $\varepsilon$ is understood from the context we will omit the subscripts in $FS_\varepsilon$ and $S_\varepsilon$. Clearly $FS_\varepsilon$ is finite. Also, since there are only finitely many $m \in \mathbb{Z}$ such that $x_i^{(m)}$ appears in $\varepsilon$, the set $S_\varepsilon$ is finite, hence also $\Delta_\varepsilon$ is finite.
It is easy to see that  $\Delta_\varepsilon$ is finite and that it contains $ FS_\varepsilon$.

The notion of a compatible surjection that we now define, is a sister notion to that of a diagram that allows for tighter control in falures of equations. The underlying set of a diagram can be arbitrary and the partial functions need to be explicitly given, but for compatible surjections (which are tight to a given equation $\varepsilon$) the undrerlying set has labels inhereted from $\Delta_{\varepsilon}$ and this is enough infortation to construct the partial functions. 
Given an equation $\varepsilon(x_1,\ldots,x_l)$ in intensional form, 
a \emph{compatible surjection} for $\varepsilon$ is  an onto map $\varphi:\Delta_{\varepsilon}\rightarrow\mathbb{N}_q$, where $\mathbb{N}_q=\{1, \ldots, q\}$ has its the natural order (and $q\leq |\Delta_\varepsilon|$), such that:
\begin{itemize}
    \item[(i)] The relation $g_i:=\{(\varphi(u),\varphi(x_iu)) \mid u, x_iu\in\Delta_{\varepsilon}\}$ on $\mathbb{N}_q$ is an order-preserving partial function for all $i \in \{1,\ldots,l\}$.
    \item[(ii)]  The relation ${\diagcov} :=\{(\varphi(v), \varphi(+v)) \mid v, +v\in\Delta_{\varepsilon}\} \cup \{(\varphi(-v),\varphi(v)) \mid v, -v\in\Delta_{\varepsilon}\}$ on $\mathbb{N}_q$ is contained in the covering relation $\prec$ of $\mathbb{N}_q$.
   % \item[(ii)]  The relation ${\diagcov} =\{(\varphi(v), \varphi(+v)) \mid v, +v\in\Delta_{\varepsilon}\} \cup \{(\varphi(-v),\varphi(v)) \mid v, -v\in\Delta_{\varepsilon}\}$ on $\mathbb{N}_q$ is an order-preserving partial function.
   %    \item[(iii)]  For all $a,b \in \mathbb{N}_q$, $a\diagcov b$ implies $a\prec b$. 
    \item[(iii)] $\varphi(x_i^{(m)}u)=g_i^{[m]}(\varphi(u))$, when $i \in \{1,\ldots,l\}$, $m\in\mathbb{Z}$ and $u, x_i^{(m)}u\in\Delta_{\varepsilon}$.
\end{itemize}

Note that the first two conditions ensure that $\m D_{\varepsilon, \varphi}:= (\mathbb{N}_q, {\leq}, {\diagcov}, g_1, \ldots, g_l)$ is a diagram; in (iii),  $g_i^{[m]}$ is calculated in this diagram. Also, it follows that the relation $\diagcov$ is an order-preserving partial function. We say that $\varphi$ is \emph{$n$-periodic} with respect to a spacing embedding, if all $g_i$'s are $n$-periodic with respect to that embedding; this is equivalent to $\m D_{\varepsilon, \varphi}$ being $n$-periodic with respect to that embedding. We say that  $\varphi$ is \emph{$n$-periodic} if it is $n$-periodic with respect to some spacing embedding.

We say that the equation $1\leq w_{1}\vee\ldots\vee w_{k}$ in intensional form \emph{fails} in a compatible surjection $\varphi$, if $\varphi(w_1), \ldots, \varphi(w_k)<\varphi(1)$. 

 We denote by $\m {Pf}(\m \Delta)^{\pm}$ the expansion of the algebra $\m {Pf}(\m \Delta)$ with two distinguished elements $+$ and $-$, where $+(a)=b$ iff $a\diagcov b$ and  $-(a)=b$ iff $b\diagcov a$.  Also, every intensional homomorphism  $\varphi: \m {Ti} \rightarrow \m {Pf}(\m \Delta)$ extends to a homomorphism $\varphi^\pm: \m {Ti}^{\pm} \rightarrow \m {Pf}(\m \Delta)^{\pm}$, where $\m {Ti}^{\pm}$ is the expansion with the two constants. Likewise, when $\m \Omega$ is an integral chain, we can consider homomorphisms $\varphi^\pm: \m {Ti}^{\pm} \rightarrow \m F(\m \Omega)^{\pm}$, where $\m F(\m \Omega)^{\pm}$ denotes the expansion of $\m F(\m \Omega)$ with the functions $+$ and $-$ that give the upper cover and the lower cover of an element. 

\begin{theorem}\label{t: DLPG to comsur}
    If an equation $\varepsilon$ fails in $\m F_n(\mathbb{Z})$, then it also fails in some $n$-periodic compatible surjection for $\varepsilon$. 
\end{theorem}
\begin{proof}
If $\varepsilon=\varepsilon(x_1,\ldots,x_l)$ is an equation in intentional form $1\leq w_{1}\vee\ldots\vee w_{k}$ that fails in  $\m F_n(\mathbb{Z})$, then there exists a list $f=(f_1,\ldots,f_l)$ of elements of $\m F_n(\mathbb{Z})$ and $p\in\mathbb{Z}$  such that $w_{1}^{\m F_n(\mathbb{Z})}(f)(p), \ldots,  w_{k}^{\m F_n(\mathbb{Z})}(f)(p)>p$; here $f_i=\psi(x_i)$ where $\psi:\m {Ti} \ra \m F_n(\mathbb{Z})$ is the homomorphism witnessing the failure of $\varepsilon$. We denote by $\psi^\pm:\m {Ti}^\pm \ra \m F_n(\mathbb{Z})^\pm$ the extension of $\psi$.
In Theorem~3.6 of \cite{GG} it is shown that $\psi_p:\Delta_{\varepsilon}\rightarrow\psi_p[\Delta_{\varepsilon}]$ is a compatible surjection, where the order on $\psi_p[\Delta_{\varepsilon}]$ is inherited from $\mathbb{Z}$ and 
$$\psi_p(u):=\psi^\pm(u)(p)=u^{\m F_n(\mathbb{Z})^\pm}(f)(p)$$ 
for $u \in \Delta_\varepsilon$. (In \cite{GG} the compatible surjection is actually taken to be the composition of $\psi_p$ with the (unique) isomorphism of the chain $\psi_p[\Delta_{\varepsilon}]$ with the initial segment $\mathbb{N}_q$ of $\mathbb{Z}^+$, where $q=|\psi_p[\Delta_{\varepsilon}]|$, but we do not need to do that.) To simplify the notation, we write  $u_{fp}$ for $u^{\m F_n(\mathbb{Z})^\pm}(f)(p)$. Also, in the same theorem of \cite{GG} it is shown that $\varepsilon$ fails in $\psi_p$.

We will show that $\psi_p$ is $n$-periodic.
   Suppose $\psi_p(u)\leq\psi_p(v)+kn$ for some $u,v \in \Delta_\varepsilon$ and $k\in\mathbb{Z}$, i.e., $ u_{f,p}\leq v_{f,p}+kn$. %, and since $\psi_p$ respects and reflects coverings, $ u_{f,p}\leq v_{f,p}+kn$.
   Given that $f_i\in F_n(\mathbb{Z})$, we have $\psi_p(x_iu)=(x_iu)_{f,p}=f_i( u_{f,p})\leq f_i(v_{f,p})+kn=(x_iv)_{f,p}+kn=\psi_p(x_iv)$. 
   This shows that all the partial functions are $n$-periodic with respect to the identity spacing embedding on $\mathbb{Z}$.
   %Hence, $\psi_p(x_iu)=(x_iu)_{f,p}\leq (x_iv)_{f,p}+kn=\psi_p(x_iv)+kn$. 
   %meaning that  $\varphi((x_iu)_{f,p})\leq\varphi((x_iv)_{f,p})+kn$. Concluding that, for all $i=1,\ldots,l$, $e_i$ is an $n$-periodic partial function.
\end{proof}

\subsection{Controlling the automorphisms of $\mathbb{Z}$.}

As mentioned before, a key issue with the definition of $n$-periodicity in a diagram is that it does not provide any control of the spacing embeddings in $\mathbb{Z}$. The spacing has to do with the fact that some functions $f$ in $\m F_n(\mathbb{Z})$ have a big numerical difference $f(m)-m$ between their input and output values and thus may span multiple periods (this can be thought of as the \emph{height} %\CHAT{I}{I want to use heights for spacing embeddings, is this term necessary} 
of $f$ as measured at $m$). The lack of control of these differences translates into lack of control of the spacing embedding $e$ (when moving from a partial function on a diagram to a function of  $\m F_n(\mathbb{Z})$). Put differently, the problem is not the numerical size of the set $\Delta$ on which a diagram is based, nor of its bijective set $e[\Delta]$ in $\mathbb{Z}$, but rather the size of the convexification of $e[\Delta]$ in $\mathbb{Z}$.
We will show that every function $f$ in $\m F_n(\mathbb{Z})$  naturally decomposes of  into an automorphism $f^\circ \in F_1(\mathbb{Z})$ and a \emph{short} function $f^* \in F_n(\mathbb{Z})$, i.e., a function whose height at every point is bounded by $2n$. In that sense, the height problem is now focused only on the $f^\circ$ component of $f$; in this section, we prove results that control the height of $f^\circ$.

\medskip

% The following theorem gives us for every c-chain, a systematic way to build spacing embedding for it ensuring that for any partial function coming from an automorphism, the counterpart of it can also be extended to an automorphism on $\mathbb{Z}$. 

%Given a spacing embedding  $e: \m \Delta \rightarrow (\mathbb{Z}, \leq_{\mathbb{Z}},\diagcov_{\mathbb{Z}})$, where $e[\Delta]=$..., we define $y$

The following definition characterizes all the spacing embeddings that interact well with a c-chain. %: every  partial function on $\mathbb{Z}$ that is $n$-periodic with respect to the the identity spacing embedding has a counter part that can be extended to an $n$-periodic function. 
Given a finite  sub c-chain $\m \Delta$ of $\mathbb{Z}$, a spacing embedding  $e: \m \Delta \rightarrow (\mathbb{Z}, \leq_{\mathbb{Z}},\diagcov_{\mathbb{Z}})$ is said to \emph{transfer $n$-periodicity} if for every $f\in F_n(\mathbb{Z})$, $(f|_{\Delta\times\Delta})^e$, when  non-empty, can be extended to an element of $F_n(\mathbb{Z})$; here $f|_{\Delta\times\Delta}$ is the partial function obtained by restricting both domain and range to $\Delta$ and $(f|_{\Delta\times\Delta})^e$ is its counterpart in $e[\Delta]$. Below we show that these spacing embeddings, for $n=1$, can be found by solving a linear system.

A finite subset $\Delta=\{p_0,\ldots,p_l\}$ of $\mathbb{Z}$, where $p_0<p_1<\ldots<p_l$, is called a \emph{solution} to a system of equations, if the vector $y_\Delta:=(p_1-p_0,\ldots, p_l-p_{l-1})$ of spaces between the points of $\Delta$ is a solution. Also, we say that a  spacing embedding is a \emph{solution} to a system, if its image is a solution.

\begin{lemma}\label{l:equations in delta}
If $\Delta=\{p_0,\ldots,p_l\}$ is a finite subset  of $\mathbb{Z}$, where $p_0<p_1<\ldots<p_l$, and $y_\Delta:=(p_1-p_0,\ldots, p_l-p_{l-1})$, then  for every $j,z,j',z'\in\{0,\ldots,l\}$ with $z\leq j$ and $z'\leq j'$, we have
$ p_j-p_z =p_{j'}-p_{z'}$
iff   
$$  \sum_{k=z+1}^{j} y_k -\sum_{k=z'+1}^{j'} y_k = (j'-z')-(j-z).$$
\end{lemma}
\begin{proof}
%Since $ p_j-p_z =\sum_{k=z+1}^{j} (p_{k}-p_{k-1})= \sum_{k=z+1}^{j} (y_{k}+1) =\sum_{k=z+1}^{j} y_k +(j-z)$ for  $z\leq j$, 
Note that for %all $j,z\in\{0,\ldots,l\}$ with 
$z\leq j$, we have $$ p_j-p_z =\sum_{k=z+1}^{j} (p_{k}-p_{k-1})= \sum_{k=z+1}^{j} (y_{k}+1) =\sum_{k=z+1}^{j} y_k +(j-z).$$ %if $j,z,j',z'\in\{0,\ldots,l\}$, 
Therefore, for $z\leq j$ and $z'\leq j'$ we have:  $p_j-p_z =p_{j'}-p_{z'}$ iff 
 \begin{align*}
  \sum_{k=z+1}^{j} y_k +j-z=&\sum_{k=z'+1}^{j'} y_k +j'-z' \qquad  \text{ iff } \nonumber\\
  \sum_{k=z+1}^{j} y_k -\sum_{k=z'+1}^{j'} y_k=&(j'-z')-(j-z) \qedhere
 \end{align*}
% On the other hand, suppose that  $$  \sum_{k=z+1}^{j} y_k -\sum_{k=z'+1}^{j'} y_k=(j'-z')-(j-z)  $$ then,
%\begin{align*}
%\sum_{k=z+1}^{j} y_k -\sum_{k=z'+1}^{j'} y_k=&(j'-z')-(j-z) \\
%\sum_{k=z+1}^{j} (p_{k+1}-p_k-1) -\sum_{k=z'+1}^{j'} (p_{k+1}-p_k-1)=&(j'-z')-(j-z) \\
%[p_j-p_z-(j-z)]-[p_{j'}-p_{z'}-(j'-z')]=&(j'-z')-(j-z) \\
%p_j-p_z=&p_{j'}-p_{z'}
%\end{align*}
\end{proof}

A system of equations $AY=b$, where $A$ is an $L\times l$ matrix  and $Y$ and $b$  column vectors of size $l$, is said to be \emph{$\Delta$-bounded} if $\Delta$ is a solution to the system, 
%$L\leq |\Delta|^2 2^{|\Delta|^2}+|\Delta| $ \CHAT{I}{This is true but it is not necessary for the calculations, does it make sense to include it in the definition?}\CHAT{N}{So we can completely omit the restriction to $L?$},
$l\leq |\Delta|$, the entries in $b$ are integers in $[-2|{\Delta}|, 2|{\Delta}|]$ %with $|b_k|\leq 2|{\Delta}|$ for all $k$ 
and the entries of $A$ are integers in $\{-1,0,1\}$. 

The next lemma states that finding all the spacing embedddings that transfer $1$-periodicity amounts to solving a linear system.
We stress that sub c-chains of $\mathbb{Z}$ are not assumed to be  convex.
 
 \begin{lemma}\label{l: AY=b equation in F_1(Z)}
    If $\m \Delta$ is a finite  sub c-chain of $\mathbb{Z}$, then  there exists a $\Delta$-bounded system of equations such that a spacing embedding with domain $\m \Delta$ is a solution iff it transfers 1-periodicity. %\CHAT{N}{Do you mean $e[\Delta]$?} \CHAT{I}{Yes, it was more conveninent to use overline}
     %consistent system of equations $AY=b$, where $A$ is a matrix of dimensions $L\times l$ with $L\leq 2^{|\Delta|^2}\cdot |\Delta|^2$, $l\leq |\Delta|$, entries in $\{-1,0,1\}$, $b$ has integer entries, and $|b_k|\leq 2|{\Delta}|$ for all $k$. Satisfying that if $y_{\bar{\Delta}}$ is a solution for the system and $\overline{\cdot}:\Delta\rightarrow{\mathbb{Z}}$ is a spacing embedding with $Im(\bar{\Delta})$. Then for all $f\in F_1(\mathbb{Z})$, the counterpart in $\overline{\Delta}$ of $f\cap (\Delta\times\Delta) $ can be extended to an element of $F_1(\mathbb{Z})$.  %$\overline{\cdot}:\Delta\rightarrow\overline{\Delta}\subseteq{\mathbb{Z}}$ is an order-preserving bijection preserving $\diagcov$. Then for all $f\in Aut(\mathbb{Z})$, there exists $g\in Aut(\mathbb{Z})$ such that  $g(\overline{a})=\overline{f(a)}$ for all $a,f(a)\in\Delta$.
     %\CHAT{I}{the n-periodic spacial embeddings don't give us an equivalence relation between deltas}
\end{lemma}

\begin{proof}
First consider all the possible non-empty intersections of automorphisms of $\mathbb{Z}$ with $\Delta\times\Delta$, and note that there are finitely-many such intersections, say $s$-many; let $f_1,\dots,f_s\in F_1(\mathbb{Z})$ be functions realizing these intersections. Since the intersections are subsets of $\Delta\times \Delta$, we get $s\leq 2^{|\Delta\times\Delta|}=2^{|\Delta|^2}$.
If $\Delta=\{p_0,\ldots,p_l\}$, where $p_0<p_1<\ldots<p_l$, we set $y_j:=p_j-p_{j-1}-1$, for $j \in \{1, \ldots, ,l\}$ and $y_{\Delta}:=(y_1,\ldots,y_{l})$; in particular, if $p_{j-1}\diagcov p_j$ we have $y_j=0$. 
For %all $j,z\in\{0,\ldots,l\}$ with 
$z< j$ and for all $i$, 
%we have 
%$$ p_j-p_z =\sum_{k=z+1}^{j} (p_{k}-p_{k-1})= \sum_{k=z+1}^{j} (y_{k}+1)
%=\sum_{k=z+1}^{j} y_k +(j-(z+1)+1)
%=\sum_{k=z+1}^{j} y_k +(j-z) $$
%\begin{align*}
% p_j-p_z&=\sum_{k=z+1}^{j} (p_{k}-p_{k-1})=\sum_{k=z+1}^{j} (y_{k}+1)\\
%        &=\sum_{k=z+1}^{j} y_k +(j-(z+1)+1)=\sum_{k=z+1}^{j} y_k +(j-z)    \label{e:4} 
%\end{align*}
%Since all $f_i$'s are translations, we get $f_i(a)-f_i(b)=a-b$, when $a,b\in\mathbb{Z}$.
%$f_{i}\in F_1(\mathbb{Z})$ for each $i\in\{1,\ldots,l\}$, if $a,b\in\mathbb{Z}$, $f_i(a)-f_i(b)=a-b$. 
%For each $f_i$, if
%In particular, if %$f_i(p_z),f_i(p_j)\in\Delta$, 
if  $p_{z'}:=f_{i}(p_z)$ and $p_{j'}:=f_{i}(p_j)$ are in $\Delta$, then $p_{j}-p_{z} =p_{j'}-p_{z'}$ since $f_i$ is a translation, %given that $f_i\in F(\mathbb{Z})$, 
so by Lemma~\ref{l:equations in delta} we get that $(y_n)$ satisfies the equation:
%and $z\leq j$, by order preservation of $f_{i}$, $p_{z'}=f_{i}(p_z)\leq f_{i}(p_j)=p_{j'}$hence $z'\leq j'$, so:
% \begin{align}
%  p_{j}-p_{z}&=p_{j'}-p_{z'} \nonumber\\
%  \sum_{k=z+1}^{j} y_k +j-z&=\sum_{k=z'+1}^{j'} y_k +j'-z' \nonumber\\
%  \sum_{k=z+1}^{j} y_k -\sum_{k=z'+1}^{j'} y_k&=(j'-z')-(j-z) 
% \end{align}
%For each $p_j,p_z\in\Delta$ with $f_i(p_z),f_i(p_j)\in\Delta$ and $z<j$, we consider the equation
%  \begin{align}
%  \sum_{k=z+1}^{j} y_k -\sum_{k=z'+1}^{j'} y_k= (j'-z')-(j-z) %\tag{$i,z,j$}
% \end{align}
% consider then the equations
   \begin{align}
  \sum_{k=z+1}^{j} Y_k -\sum_{k=z'+1}^{j'} Y_k&= (j'-z')-(j-z) \tag{$i,z,j$}
 \end{align} 
 We also consider the equations $Y_k=0$ for all $k\in \{1,\ldots, l\}$ where $p_{k-1}\diagcov p_k$ (capturing the fact that $y_k=0$ for these $k$'s). 
The resulting finite system, of all $(i,z,j)$-equations and all the $Y_k=0$ equations, can be written as $AY=b$, for an $L \times l$ matrix $A$, and column vectors $Y$ and $b$ of size $l$, where the entries of $A$ are in $\{-1,0,1\}$, the entries of $b$ have absolute value at most $2|\Delta|$, since they all have the form $(j'-z')-(j-z)$, and $\Delta$ is a solution of the system. Therefore, the system is $\Delta$-bounded.
%The matrix $A$ has dimensions $L\times l$ where $L$ is the number of equations of the system, $l$ is the number of elements in ${\Delta}$, and $b$ has dimensions $L\times 1$. Observe that $l=|{\Delta}|$ and $L\leq n\cdot |{\Delta}|^2+|\Delta|\leq 2^{|\Delta|^2}\cdot |\Delta|^2+|\Delta|$. Observe, that by our construction $y_{\Delta}$ is a solution of $Ax=b$, hence the system is consistent.
Let  $\overline{\cdot}:\Delta\rightarrow{\mathbb{Z}}$ be a spacing embedding.

If $\overline{\Delta}$ is a solution of the  system $AY=b$,   for each $i\in \{1,\ldots,s\}$,
we define the function $g_i:\mathbb{Z}\rightarrow\mathbb{Z}$  by $g_i(x)=x+(\overline{f_i(c_i)}-\overline{c_i})$ for all $x\in\mathbb{Z}$, where  $c_i,f_i(c_i)\in\Delta$; thus $g_i\in F_1(\mathbb{Z})$. %\CHAT{N}{What if $f_i$ is the empty function?}\CHAT{I}{They should be non-empty} 
Note that $g_i$ does not depend on the choice of $c_i$:  If  $k,f_i(k)\in\Delta$, then there exist $j,z,j',z'\in\{1,\ldots,l\}$ such that $\overline{c_i}=\overline{p_j}$, $\overline{k}=\overline{p_z}$, $\overline{f_{i}(c_i)}=\overline{p_{j'}}$, and $\overline{f_{i}(k)}=\overline{p_{z'}}$; since $\bar{\Delta}$ is a solution, $y_\Delta$  it satisfies the equation $(i,\min\{j,z\},\max\{j,z\})$, which by Lemma~\ref{l:equations in delta}  is equivalent $\overline{f_{i}(k)}-\overline{f_{i}(c_i)}=\overline{k} -\overline{c_i}$. Hence,  $g_i(\overline{k})=\overline{k}+(\overline{f_i(c_i)}-\overline{c_i})=\overline{f_{i}(k)}$, for all $k$ such that $k,f_i(k)\in\Delta$. As a result, the restriction of $g_i$ to $\overline{\Delta}$ is equal to the counterpart of $f_i|_{ \Delta\times\Delta}$ in $\overline{\Delta}$.
%Suppose now that $k,f_i(k)\in\Delta$ we know that $\overline{f_{i}(k)}-\overline{f_{i}(c_i)}=\overline{k} -\overline{c_i}$, so $\overline{f_{i}(k)}-\overline{k}=\overline{f_{i}(c_i)} -\overline{c_i}$. Meaning that $g_i$ does not depend on the choice of $c_i$ and moreover $g_i(\overline{k})=\overline{f_i(c_i)}+(\overline{k}-\overline{c_i})=\overline{f_{i}(k)}$. Hence, $g_i$ restricted to $\overline{\Delta}$ coincides with $f_i|_{ \Delta\times\Delta}$ counterpart in $\overline{\Delta}$.\CHAT{I}{Is it obvious that gi extends the counterpart here?}

Conversely,  if $\overline{\cdot}:\Delta\rightarrow{\mathbb{Z}}$  transfers 1-periodicity, % it will clearly satisfy the system $AY=b$. Since 
for every   $i\in\{1,\dots,s\}$ the counterpart  of ${f_i}|_{\Delta\times\Delta}$ by $\overline{\cdot}$ can be extended to a function in $F_1(\mathbb{Z})$, i.e., for all $p_j,p_z\in \Delta$  such that $p_{j'}:=f_i(p_j),p_{z'}:=f_i(p_z)\in\Delta$, we have that $\overline{p_j}-\overline{p_z}=\overline{f_i(p_j)}-\overline{f_i(p_z)}$. Hence, 
$    \sum_{k=z+1}^{j} \overline{y_k} -\sum_{k=z'+1}^{j'} \overline{y_k}=(j'-z')-(j-z) $ by Lemma~\ref{l:equations in delta}, so if $z<j$, then $\overline{\Delta}$ satisfies the equation $(i,z,j)$. Since this is true for all $i\in\{1,\ldots,s\}$, $\overline{\cdot}$ preserves coverings and $j,z\in Dom({f_i}|_{\Delta\times\Delta})$, we have that $\overline{\Delta}$ satisfies $Ax=b$. 
%\CHAT{N}{Include the reason why.}
%Suppose $\overline{\cdot}:{\Delta}\rightarrow\overline{\Delta}\subseteq\mathbb{Z}$ an order-preserving bijection, preserving $\diagcov$, such that $y_{\bar{\Delta}}$ is a solution for the $AY=b$ system. Consider for each $i=1,\ldots,n$, an element $c_i,f_i(c_i)\in\Delta$, and $g_i:\mathbb{Z}\rightarrow\mathbb{Z}$ given by $g_i(k)=k+(\overline{f_i(c_i)}-\overline{c_i})$. Clearly, by definition $g_i\in Aut(\mathbb{Z})$. Suppose now that $k,f_i(k)\in\Delta$, since $y_{\bar{\Delta}}$ is a solution for the $AY=b$ system, we know that $\overline{f_{i}(k)}-\overline{f_{i}(c_i)}=\overline{k} -\overline{c_i}$, so $\overline{f_{i}(k)}-\overline{k}=\overline{f_{i}(c_i)} -\overline{c_i}$. Meaning that $g_i$ does not depend on the choice of $c_i$ and moreover $g_i(\overline{k})=\overline{f_i(c_i)}+(\overline{k}-\overline{c_i})=\overline{f_{i}(k)}$. 
\end{proof}

The following theorems will help us guarantee the existence of small solutions for $\Delta$-bounded  systems of equations. % of the type obtained in Lemma~\ref{l: AY=b equation in F_1(Z)}.

 \begin{theorem} \textnormal{\cite{BFT}}
 Let $Ay = b$ be a system of $M \times l$ linear equations with integer coefficients. Assume the rows of $A$ are linearly independent (hence $M\leq l$) and denote by $\gamma'$ (respectively $\gamma$) the maximum of the absolute values of the $M \times M$ minors of the matrix $A$ (the augmented matrix $(A|b)$). If the system has a solution in non-negative integers, then the system has a solution $(y_k)$ in the non-negative integers with $y_k \leq \gamma'$ for $l - M$ variables and $y_i \leq (l- M + 1)\gamma$ for $M$ variables.
 \end{theorem}
  Given that $l-M+1\geq 1$ and $\gamma'\leq \gamma$ we obtain the following corollary.
  \begin{corollary}\label{c:small solutions}
Let $Ay = b$ be a system of $M \times l$ linear equations with integer coefficients. Assume the rows of $A$ are linearly independent, and denote by $\gamma$ the maximum of the absolute values of the $M \times M$ minors of the augmented matrix $(A| b)$. If the system has a solution in non-negative integers, then the system has a solution $(y_k)$ in non-negative integers with $y_k \leq (l-M+1)\gamma$ for all $k$.  
\end{corollary}

Using these corollaries and Lemma~\ref{l: AY=b equation in F_1(Z)}, we will prove that for every sub c-chain of $\mathbb{Z}$, there exists a spacing embedding of bounded height that transfers 1-periodicity. %This result is important because given a finite c-chain, there is no control over the size of its convexification. This lemma allows us to build partial functions with similar behavior to those over the original c-chain in a new one with a bounded convexification.  
For convenience, we define $\rho(a):=2a^3a!+a+1$, for  $a\in\mathbb{Z}^+$.

\begin{lemma}\label{l: bounded spacing embedding in Aut(Z)}
Given a finite  sub c-chain $\m \Delta$ of $\mathbb{Z}$, there exists a spacing embedding with domain $\Delta$ that transfers 1-periodicity and that has height at most $\rho(|\Delta|)$.
%there exists $g\in F_1(\mathbb{Z})$ such that  $g(\overline{a})=\overline{f(a)}$ for all $a,f(a)\in\Delta$.
\end{lemma}

\begin{proof}
By Lemma~\ref{l: AY=b equation in F_1(Z)}, there exists a $\Delta$-bounded system of equations $AY=b$
%where $A$ is a matrix of dimensions $L\times l$ with $L\leq 2^{|\Delta|^2}\cdot |\Delta|^2+|\Delta|$, $l\leq |\Delta|$ and entries in $\{-1,0,1\}$, $b$ has integer entries, and $|b_k|\leq 2|{\Delta}|$, for all $k$,
satisfying the conditions of the lemma. 
Since the system $Ay=b$ is $\Delta$-bounded, $\Delta$ is a solution for it, hence the system is consistent. It follows, say by Theorem~2.38 of \cite{RC}, that $rank(A)=rank(A|b)$. So we can select a maximal collection of linearly independent rows of $(A|b)$ to obtain an equivalent  (sub)system $A'y=b'$ where $rank(A|b)=rank(A'|b')=rank(A')$, hence all the rows are linearly independent. Therefore, $A'$ has  dimensions $M\times l$ where $M\leq l$. Also, since $y_\Delta$ is a solution (in the non-negative integers) of $Ay=b$, it is also a solution of the equivalent system $A'y=b'$.

Observe that if $C=(c_{i,j})$ is an $M\times M$ submatrix of $(A'|b')$, then 
$$
  |det(C)|=|\sum_{\sigma\in S_M} sgn(\sigma)c_{\sigma(1),1}\cdot\ldots\cdot c_{\sigma(M),M}|
          \leq \sum_{\sigma\in S_M} |c_{\sigma(1),1}|\cdot\ldots\cdot |c_{\sigma(M),M}|$$
Since $Ax=b$ is $\Delta$-bounded, all columns of $(A',b')$, aside from the last column which equals $b'$, have entries among $-1,0,1$. Therefore,  for all $\sigma\in S_M$, we have $|c_{\sigma(i),i}|\leq 1$ for $i\neq M$ and $|c_{\sigma(M),M}|%\leq 2l
\leq 2|{\Delta}|$, since the entries of $b$ are integers in $[-2|{\Delta}|, 2|{\Delta}|]$ (again because of  $\Delta$-boundedness). So 
$$  |det(C)|\leq \sum_{\sigma\in S_M} |c_{\sigma(1),1}|\cdot\ldots\cdot |c_{\sigma(M),M}|
          \leq \sum_{\sigma\in S_M} |c_{\sigma(M),M}|\leq M!\cdot 2|\Delta|$$
Therefore, if  $\gamma$ is the maximum of the absolute values of the $M \times M$ minors of the augmented matrix $(A'| b')$, then $\gamma\leq M!\cdot 2|\Delta|\leq  l!\cdot 2|\Delta|\leq (|\Delta|)!\cdot 2|\Delta|$, where we used that $l\leq |\Delta|$ due to $\Delta$-boundedness. Also $l-M+1\leq l\leq |\Delta|$. So, by Corollary~\ref{c:small solutions}, there exists a solution $(y_k)$ in the non-negative integers with $y_k\leq  (|\Delta|)!\cdot 2|\Delta|^2$ for all $k$.

Using the solution $(y_k)$, we define the map  $\overline{\cdot}:\Delta\rightarrow\mathbb{Z}$ by setting $\overline{p_0}=0$ and  $\overline{p_i}:=i+\sum ^{i}_{j=1} y_j$, for all $i\in \{1,\ldots, l\}$; so $\overline{\Delta}=\{\overline{p_0}, \ldots, \overline{p_l}\}$. The map $\overline{\cdot}$ is order preserving, because the solution $(y_k)$ has non-negative entries, and it preserves $\diagcov$, since every zero entry in $y_{\Delta}$ implies that the same entry is zero in $y_{\bar{\Delta}}$. Also, observe that $y_{\bar{\Delta}}=(y_k)$, hence $\overline{\cdot}$ is a solution of $A'x=b'$, thus also of $Ax=b$.  %\CHAT{N}{Are we assuming that $y_{\Delta}$ is a solution of the system, or only that the system is consistent?}\CHAT{I}{Both}\CHAT{N}{I dont see where we say that $y_{\Delta}$ is a solution of the system. Also, we mention the consistency of system, but I do not see the statement of Lemma~\ref{l: AY=b equation in F_1(Z)} saying anything about consistency. It only links two equivalent conditions.}\CHAT{I}{$\Delta$ is a solution is part of the conditions for the system to be $\Delta$-bounded, and because it has a solution it is consistent}
Therefore, $\overline{\cdot}$ is a spacing embedding that is a solution of $Ax=b$. Since we took $Ax=b$ to be the system given by Lemma~\ref{l: AY=b equation in F_1(Z)}, by that lemma it follows that $\overline{\cdot}$ transfers $1$-periodicity. 

%Observe now that since $l=|{\Delta}|$ and $y_i\leq  |\Delta|\cdot \gamma '\leq (|\Delta|)!\cdot 2|\Delta|^2$, we conclude that 
Since $y_k\leq  (|\Delta|)!\cdot 2|\Delta|^2$ for all $k$, the height of $\overline{\cdot}$ is  $1+ \overline{p_i}=1+l+\sum ^{i}_{j=1} y_j\leq 1+l + l ( (|\Delta|)!\cdot 2|\Delta|^2) \leq 1+ |\Delta|+ |\Delta| ( (|\Delta|)!\cdot 2|\Delta|^2)= 1+ |\Delta|+  (|\Delta|)!\cdot 2|\Delta|^3$, hence the height is bounded by  $\rho(|\Delta|)$. 
%\CHAT{N}{I am not sure these calculations are correct. Can you check again?} 
%\CHAT{N}{Unfortunately, we cannot use Lemma~\ref{l: AY=b equation in F_1(Z)}, because that lemma talks (existentially) about its own system. Unless we inspect the proof of the lemma, we cannot know what that system is.... One solution is to pull the construction of the system outside the proof so we can refer to it.}\CHAT{I}{But the diagram came from Lemma 4.2, we create a spacing embedding from it and the lemma guarantees that it transfers 1-periodicity} \CHAT{N}{So we cannot use the statement of Lemma 4.2 and we have to use its proof, right?}\CHAT{I}{No, I think we are just using the statement}
%and there exist $g_1,\ldots,g_n\in Aut(\mathbb{Z})$ such that $g_i(\overline{k})=\overline{f_i(k)}$ for all $k\in\Delta$.
\end{proof}

\subsection{From $\m F_n(\mathbb{Z})$ to a short $n$-periodic compatible surjection.}

We will describe the decomposition of functions $f \in \m F_n(\mathbb{Z})$ into functions of $\m{F}_1(\mathbb{Z})$ and short functions of $\m{F}_n(\mathbb{Z})$, and combine it with the results of the previous section to ensure that for any $n$-periodic diagram there is a short spacing embedding witnessing the $n$-periodicity.
%with 'small' height, meaning that a 'short convexification' of it is possible.

A spacing embedding $e:\m \Delta \rightarrow \mathbb{Z}$ over a c-chain $\m\Delta$ is called \emph{$n$-short} if its height is at most $\nu(|\Delta|)$, where $\nu(a)=(\rho(3a)+1)n$ %\CHAT{N}{I think we need  $\nu(a)=(\rho(3a)+1)n$ in the theorem where we use it} \CHAT{I}{Yes} 
for all $a\in\mathbb{Z}^+$. (As always we assume that images of spacing embeddings are non-negative and include $0$). %$0=\min({e[\Delta]})\leq\max({e[\Delta]}) \leq \nu(|\Delta|)$ 
%An $n$-periodic compatible surjection  is called \emph{short} if it is $n$-periodic with respect to some $n$-short spacing embedding.
Recall that for a fixed $n\in\mathbb{Z}^+$, for every $x\in\mathbb{Z}$, $Rx$ denotes the remainder and $Qx$ the quotient of dividing $x$ by $n$, while  $Sx:=x-Rx=nQx$.
%\CHAT{N}{I think this looks like multiplication. Also, I think that $S$ is better than $q$, as we use $q$ for a numerical variable, and the same for $r$:}
%For a fixed $n\in\mathbb{Z}^+$, given $x\in\mathbb{Z}$ we denote by $Sx$ the quotient the division of $x$ by $n$ and by $Rx$ the remainder; so $Rx=x-Sx$ and $0 \leq Rx<n$.

Given $g\in F_1(\mathbb{Z})$ and $h\in F_n(\mathbb{Z})$, we have $g,h\in F_n(\mathbb{Z})$ so  $g\circ h\in F_n(\mathbb{Z})$. Conversely,  for $f\in F_n(\mathbb{Z})$, we define $f^{\circ}(x):=x+Sf(0)$ and $f^{*}(x):=f(x)-Sf(0)$, for all $x \in \mathbb{Z}$. In the following lemmas we will make use of the fact that for all $g\in F_1(\mathbb{Z})$ and $x\in\mathbb{Z}$, we have $g(x)=g(x-k)+k$ for all $k\in\mathbb{Z}$, and in particular $g(x)=x+g(0)$.   

\begin{lemma}\label{l:fof* decomp}
The maps $f \mapsto (f^\circ, f^*)$ and $(f^\circ, f^*)\mapsto f^{\circ}\circ f^{*}$ define a bijection between $F_n(\mathbb{Z})$ and the set of pairs $f^{\circ}\in F_1(\mathbb{Z})$  and  $f^{*}\in F_n(\mathbb{Z})$ with $Sf^{\circ}(0)=f^{\circ}(0)$ and $0\leq f^{*}(0)<n$. %\CHAT{N}{The second condition is not needed, but it follows if we want it.}
%    Given $f\in F_n(\mathbb{Z})$, there exists $f^{\circ}\in F_1(\mathbb{Z})$ and $f^{*}\in F_n(\mathbb{Z})$ with $0\leq f^{*}(0)<n$, such that $f=f^{\circ}\circ f^{*}$. And if $f^{\circ}\in F_1(\mathbb{Z})$ and $f^{*}\in F_n(\mathbb{Z})$ with $0\leq f^{*}(0)<N$, the map $f^{\circ}\circ f^{*}\in F_n(\mathbb{Z})$. 
\end{lemma}

\begin{proof}
   For $f\in F_n(\mathbb{Z})$, we have $f^{\circ}(x):=x+Sf(0)$ and $f^{*}(x):=f(x)-Sf(0)$  for all $x \in \mathbb{Z}$, so $f^{\circ}\in F_1(\mathbb{Z})$, $f^{*}\in F_n(\mathbb{Z})$, $Sf^{\circ}(0)=S(Sf(0))=Sf(0)=f^{\circ}(0)$ and $0\leq f^{*}(0)=f(0)-Sf(0)<n$. So the map  $f \mapsto (f^\circ, f^*)$ indeed maps to suitable pairs.

  Given $f\in F_n(\mathbb{Z})$, we have   
    $(f^{\circ}\circ f^{*})(x)=f^{\circ}(f(x)-Sf(0))
                                =f(x)-Sf(0)+Sf(0)=f(x)$.
    So, the two maps compose to the identity, one way.
    
%   For the other composition 
    %Conversely,  note that if  $f^{\circ}\in F_1(\mathbb{Z})$ and $f^{*}\in F_n(\mathbb{Z})$, then $f^{\circ}, f^{*}\in F_n(\mathbb{Z})$ so  $f^{\circ}\circ f^*\in F_n(\mathbb{Z})$.
%   we start with a pair $(f^{\circ},f^{*})$ such that $f^{\circ}\in F_1(\mathbb{Z})$, $f^{*}\in F_n(\mathbb{Z})$, $Sf^{\circ}(0)=f^{\circ}(0)$, and $0\leq f^{*}(0)<n$. Then for all $x\in\mathbb{Z}$,
%    we have $(f^{\circ}\circ f^*)^{\circ}(x)=x+S(f^{\circ}f^*(0))=x+S(f^*(0)+Sf(0))=x+S(f(0)-Sf(0)+Sf(0)) =x+S(f(0))=f^{\circ}(x)$  and $(f^{\circ}\circ f^*)^{*}(x)=f^{\circ}f^*(x)-S(f^{\circ}f^*(0))
%     =f^*(x)+Sf(0)-S(f^*(0)+Sf(0))=f^*(x)+Sf(0)-Sf(0)=f^*(x)$, where we used the fact that the quotient of $(f^*(0)+Sf(0)$ is $Sf(0)$, since $0\leq f^{*}(0)<n$.
 Conversely, if $g\in F_1(\mathbb{Z})$, $h\in F_n(\mathbb{Z})$, $Sg(0)=g(0)$, and $0\leq h(0)<n$, then $S(gh(0))=S(h(0)+g(0))=S(h(0)+Sg(0))=S(g(0))=g(0)$. Therefore, 
 $(g\circ h)^{\circ}(x)=x+S(gh(0))=x+g(0)=g(x)$ and
 $(g\circ h)^{*}(x)=gh(x)-S(gh(0))=h(x)+g(0)-g(0)=h(x)$.
 %Therefore, the correspondence is a bijection.
 %   \begin{align*}
 %       (g\circ h)^{\circ}(x)&=x+S(gh(0))\\            
 %                                      &=x+S(h(0)+g(0))\\
 %%                                      &=x+S(h(0)+Sg(0))\\
 %                                      &=x+S(g(0))\\
 %                                      &=x+g(0)=g(x)
 %   \end{align*}
%\CHAT{I}{In the last line I used Sf circ(0)=f circ(0)}
%and 
% 
%    \begin{align*}
%        (g\circ h)^{*}(x)&=gh(x)-S(gh(0))\\
%                                   &=gh(x)-S(h(0)+g(0))\\
%  %                                 &=gh(x)-Sg(0)+Sh(0)\\
%                                   &=gh(x)-Sg(0)\\
%                                   &=g(0)+h(x)-Sg(0)\\
%                                   &=g(0)+h(x)-g(0)=h(x)
%    \end{align*}
% Conversely, if $f^{\circ}\in F_1(\mathbb{Z})$ and $f^{*}\in F_n(\mathbb{Z})$ with $0\leq f^{*}(0)<n$, 
 %we will show that $f^{\circ}\circ f^{*}\in F_n(\mathbb{Z})$.  then $f^{\circ}$ and $f^{*}$ are order preserving, so their composition $f^{\circ}\circ f^{*}$ is also order preserving. %$h(kn+r)=f^{\circ}(k)n+f^{*}(r)$. Also, for all $x,k\in\mathbb{Z}$, 
% \begin{align*}
%     (f^{\circ}\circ f^*)(x+kn)&=f^{\circ}(f^*(x+kn))\\
%                               &=f^{\circ}(f^*(x)+kn)\\
%                               &=f^{\circ}(f^*(x))+kn)\\
%                               &=(f^{\circ}\circ f^*)(x)+kn
% \end{align*}
% Therefore $f^{\circ}\circ f^*\in F_n(\mathbb{Z})$.
%\CHAT{N}{We should prove the two processes are inverse of each other.}\CHAT{I}{I think this is done now}\CHAT{N}{We need to replace q by S}
\end{proof}

Lemma~\ref{l: bounded spacing embedding in Aut(Z)} provides spacing embeddings (of controlled height) that transfer $1$-periodicity. Using that result and the decomposition of Lemma~\ref{l:fof* decomp}, we  will construct spacing embeddings  (of controlled height) that transfer $n$-periodicity, for any given $n$.

  \begin{example}\label{e: 2}
  We will demonstrate the idea of how, given a sub c-chain of $\mathbb{Z}$ based on  $\Delta=\{0,4\}$, we can obtain a spacing embedding $e$ with domain $\Delta$ that transfers $n$-periodicity.  In other words, we will construct $e$ such that given any function $f$ of $F_n(\mathbb{Z})$, the partial function $(f|_{\Delta\times \Delta})^e$ is extendable to an $n$-periodic function on $\mathbb{Z}$.
  
  We first decompose $f$ as $f=f^\circ \circ f^*$, where $f^{\circ}\in F_1(\mathbb{Z})$ and $f^{*}\in F_n(\mathbb{Z})$ with  $Sf^{\circ}(0)=f^{\circ}(0)$ and $0\leq f^{*}(0)<n$, according to  Lemma~\ref{l:fof* decomp};   Figure~\ref{f: figures} shows an example of that. We will keep $f^*$ as it is, but we will obtain an improved version of $f^\circ$ and compose it back with $f^*$. 

 We will first view  $f^\circ$ as providing us information about only the multiples of $n$---i.e., only about the subset $n\mathbb{Z}$ of $\mathbb{Z}$---and we will scale $n\mathbb{Z}$ into a copy of $\mathbb{Z}$.
Since $f^{\circ}(0)=Sf^{\circ}(0)$, we have  $f^{\circ}(0)=kn$ for some $k\in\mathbb{Z}$, hence $f^\circ(x)=x+kn$, for all $x\in\mathbb{Z}$; i.e., $f^\circ$ is the translation by $kn$. 
By dividing $kn$ by $n$, we consider the function $f^{\circ}\div n$ that translates by $k$, i.e., $f^{\circ}\div n:\mathbb{Z}\rightarrow\mathbb{Z}$ is given by $(f^{\circ}\div n)(x) %=Q(f^{\circ}(xn))=Q(xn+kn)
=x+k$ for all $x\in\mathbb{Z}$; we set $g:= f^{\circ}\div n$ for brevity. The function $f^{\circ}\div n$ captures the behavior of $f^\circ$ on the first elements of the periods (the multiples of $n$), and it scales $n\mathbb{Z}$ into a copy of $\mathbb{Z}$. We will apply Lemma~\ref{l: bounded spacing embedding in Aut(Z)} to this scaled function.

We define the set $\widetilde{\Delta}=\{Qx:x\in\Delta\}=\{0,2\}$ and recall that Lemma~\ref{l: bounded spacing embedding in Aut(Z)} produces a spacing embedding $d:\tilde{\Delta}\rightarrow\mathbb{Z}$ that transfers 1-periodicity and has bounded height. So, $(g|_{\tilde{\Delta}\times\tilde{\Delta}})^d$ can be extended to 
a function $g_{\tilde{\Delta},d}:\mathbb{Z}\rightarrow\mathbb{Z}$ that is $1$-periodic; say  $g_{\tilde{\Delta},d}(x)=x+\ell$ for all $x\in\mathbb{Z}$, for some $\ell \in\mathbb{Z}$. By Lemma~\ref{l: bounded spacing embedding in Aut(Z)}, the height is controlled in terms of the convexification of $d[\tilde{\Delta}]$.

We now scale this improved function $g_{\tilde{\Delta},d}$ back: we view the domain $\mathbb{Z}$ of $g_{\tilde{\Delta},d}$ as a copy of  the multiples of $n$, and we move from this subset $n\mathbb{Z}$  to all of $\mathbb{Z}$.
We define the function $g_{\tilde{\Delta},d}\times n:\mathbb{Z}\rightarrow\mathbb{Z}$ by $(g_{\widetilde{\Delta},d}\times n)(x)=x+n \ell$ for all $x\in\mathbb{Z}$. 

 By composing  $g_{\tilde{\Delta},d}\times n$ with $f^*$, we obtain a shorter version of $f$. We claim that $e(x)=d(Qx)n+Rx$, for all $x\in \Delta$, defines the desired spacing embedding and that $f_{\Delta, e}:=(g_{\tilde{\Delta},d}\times n) \circ f^*$  %$f^e$ 
 is an $n$-periodic function extending  $(f|_{\Delta\times \Delta})^e$; see Figure~\ref{f: figures}. This is the general process we follow in  the proof of the next lemma; for technical reasons the c-chain $\widetilde{\Delta}$ will need to be suitably larger, as we will see in the argument. 
  \end{example}

\begin{figure}[ht]\def\eq{=}
 \begin{center}
{\scriptsize

\begin{tikzpicture}
[scale=0.43]
%--------
\node[fill,draw,circle,scale=0.3,left](9) at (0,9){};
\node[right](9.1,0) at (-.9,9.25){$7$};
\node[fill,draw,circle,scale=0.3,right](9r) at (2,9){};
\node[right](9.1) at (2.1,9.25){$7$};

\node[fill,draw,circle,scale=0.3,left](8) at (0,8){};
\node[right](8.1,0) at (-.9,8.25){$6$};
\node[fill,draw,circle,scale=0.3,right](8r) at (2,8){};
\node[right](8.1) at (2.1,8.25){$6$};

\node[fill,draw,circle,scale=0.3,left](7) at (0,7){};
\node[right](7.1,0) at (-.9,7.25){$5$};
\node[fill,draw,circle,scale=0.3,right](7r) at (2,7){};
\node[right](7.1) at (2.1,7.25){$5$};

\node[fill,draw,circle,red,scale=0.3,left](6) at (0,6){};
\node[right](6.1,0) at (-.9,6.25){$4$};
\node[fill,draw,circle,red,scale=0.3,right](6r) at (2,6){};
\node[right](6.1) at (2.1,6.25){$4$};

\node[fill,draw,circle,scale=0.3,left](5) at (0,5){};
\node[right](9.1,0) at (-.9,5.25){$3$};
\node[fill,draw,circle,scale=0.3,right](5r) at (2,5){};
\node[right](5.1) at (2.1,5.25){$3$};

\node[fill,draw,circle,scale=0.3,left](4) at (0,4){};
\node[right](8.1,0) at (-.9,4.25){$2$};
\node[fill,draw,circle,scale=0.3,right](4r) at (2,4){};
\node[right](4.1) at (2.1,4.25){$2$};

\node[fill,draw,circle,scale=0.3,left](3) at (0,3){};
\node[right](3.1,0) at (-.9,3.25){$1$};
\node[fill,draw,circle,scale=0.3,right](3r) at (2,3){};
\node[right](3.1) at (2.1,3.25){$1$};

\node[fill,draw,circle,red,scale=0.3,left](2) at (0,2){};
\node[right](2.1,0) at (-.9,2.25){$0$};
\node[fill,draw,circle,red,scale=0.3,right](2r) at (2,2){};
\node[right](2.1) at (2.1,2.25){$0$};
%---------
\node at (9)[above=3pt]{$\vdots$};
\node at (9r)[above=3pt]{$\vdots$};
\node at (2)[below=-1pt]{$\vdots$};
\node at (2r)[below=-1pt]{$\vdots$};
%--------
\draw[-](2)--(6r);
\draw[-](3)--(6r);
\draw[-](4)--(8r);
\draw[-](5)--(8r);

%--------
\node at (1,1.25){\normalsize $f$};
\end{tikzpicture}
\quad
\begin{tikzpicture}
[scale=0.43]
%--------
\node[fill,draw,circle,scale=0.3,left](9) at (0,9){};
\node[right](9.1,0) at (-.9,9.25){$7$};
\node[fill,draw,circle,scale=0.3,right](9r) at (2,9){};
\node[right](9.1) at (2.1,9.25){$7$};

\node[fill,draw,circle,scale=0.3,left](8) at (0,8){};
\node[right](8.1,0) at (-.9,8.25){$6$};
\node[fill,draw,circle,scale=0.3,right](8r) at (2,8){};
\node[right](8.1) at (2.1,8.25){$6$};

\node[fill,draw,circle,scale=0.3,left](7) at (0,7){};
\node[right](7.1,0) at (-.9,7.25){$5$};
\node[fill,draw,circle,scale=0.3,right](7r) at (2,7){};
\node[right](7.1) at (2.1,7.25){$5$};

\node[fill,draw,circle,scale=0.3,left](6) at (0,6){};
\node[right](6.1,0) at (-.9,6.25){$4$};
\node[fill,draw,circle,scale=0.3,right](6r) at (2,6){};
\node[right](6.1) at (2.1,6.25){$4$};

\node[fill,draw,circle,scale=0.3,left](5) at (0,5){};
\node[right](9.1,0) at (-.9,5.25){$3$};
\node[fill,draw,circle,scale=0.3,right](5r) at (2,5){};
\node[right](5.1) at (2.1,5.25){$3$};

\node[fill,draw,circle,scale=0.3,left](4) at (0,4){};
\node[right](8.1,0) at (-.9,4.25){$2$};
\node[fill,draw,circle,scale=0.3,right](4r) at (2,4){};
\node[right](4.1) at (2.1,4.25){$2$};

\node[fill,draw,circle,scale=0.3,left](3) at (0,3){};
\node[right](3.1,0) at (-.9,3.25){$1$};
\node[fill,draw,circle,scale=0.3,right](3r) at (2,3){};
\node[right](3.1) at (2.1,3.25){$1$};

\node[fill,draw,circle,scale=0.3,left](2) at (0,2){};
\node[right](2.1,0) at (-.9,2.25){$0$};
\node[fill,draw,circle,scale=0.3,right](2r) at (2,2){};
\node[right](2.1) at (2.1,2.25){$0$};
%---------
\node at (9)[above=3pt]{$\vdots$};
\node at (9r)[above=3pt]{$\vdots$};
\node at (2)[below=-1pt]{$\vdots$};
\node at (2r)[below=-1pt]{$\vdots$};
%--------
\draw[-](2)--(2r);
\draw[-](3)--(2r);
\draw[-](4)--(4r);
\draw[-](5)--(4r);
\draw[-](6)--(6r);
\draw[-](7)--(6r);
\draw[-](8)--(8r);
\draw[-](9)--(8r);

%--------
\node at (1,1.25){\normalsize$f^{*}$};
\end{tikzpicture}
\quad
\begin{tikzpicture}
[scale=0.43]
%--------
\node[fill,draw,circle,scale=0.3,left](9) at (0,9){};
\node[right](9.1,0) at (-.9,9.25){$7$};
\node[fill,draw,circle,scale=0.3,right](9r) at (2,9){};
\node[right](9.1) at (2.1,9.25){$7$};

\node[fill,draw,circle,scale=0.3,left](8) at (0,8){};
\node[right](8.1,0) at (-.9,8.25){$6$};
\node[fill,draw,circle,scale=0.3,right](8r) at (2,8){};
\node[right](8.1) at (2.1,8.25){$6$};

\node[fill,draw,circle,scale=0.3,left](7) at (0,7){};
\node[right](7.1,0) at (-.9,7.25){$5$};
\node[fill,draw,circle,scale=0.3,right](7r) at (2,7){};
\node[right](7.1) at (2.1,7.25){$5$};

\node[fill,draw,circle,red,scale=0.3,left](6) at (0,6){};
\node[right](6.1,0) at (-.9,6.25){$4$};
\node[fill,draw,circle,red,scale=0.3,right](6r) at (2,6){};
\node[right](6.1) at (2.1,6.25){$4$};

\node[fill,draw,circle,scale=0.3,left](5) at (0,5){};
\node[right](9.1,0) at (-.9,5.25){$3$};
\node[fill,draw,circle,scale=0.3,right](5r) at (2,5){};
\node[right](5.1) at (2.1,5.25){$3$};

\node[fill,draw,circle,scale=0.3,left](4) at (0,4){};
\node[right](8.1,0) at (-.9,4.25){$2$};
\node[fill,draw,circle,scale=0.3,right](4r) at (2,4){};
\node[right](4.1) at (2.1,4.25){$2$};

\node[fill,draw,circle,scale=0.3,left](3) at (0,3){};
\node[right](3.1,0) at (-.9,3.25){$1$};
\node[fill,draw,circle,scale=0.3,right](3r) at (2,3){};
\node[right](3.1) at (2.1,3.25){$1$};

\node[fill,draw,circle,red,scale=0.3,left](2) at (0,2){};
\node[right](2.1,0) at (-.9,2.25){$0$};
\node[fill,draw,circle,red,scale=0.3,right](2r) at (2,2){};
\node[right](2.1) at (2.1,2.25){$0$};
%---------
\node at (9)[above=3pt]{$\vdots$};
\node at (9r)[above=3pt]{$\vdots$};
\node at (2)[below=-1pt]{$\vdots$};
\node at (2r)[below=-1pt]{$\vdots$};
%--------
\draw[-](2)--(6r);
\draw[-](3)--(7r);
\draw[-](4)--(8r);
\draw[-](5)--(9r);

%--------
\node at (1,1.25){\normalsize$f^{\circ}$};
\end{tikzpicture}
\quad
\begin{tikzpicture}
[scale=0.43]  
\node[fill,draw,circle,scale=0.3,left](8) at (0,8){};
\node[right](8.1,0) at (-.9,8.25){$3$};
\node[fill,draw,circle,scale=0.3,right](8r) at (2,8){};
\node[right](8.1) at (2.1,8.25){$3$};

\node[fill,draw,circle,red,scale=0.3,left](6) at (0,6){};
\node[right](6.1,0) at (-.9,6.25){$2$};
\node[fill,draw,circle,red,scale=0.3,right](6r) at (2,6){};
\node[right](6.1) at (2.1,6.25){$2$};

\node[fill,draw,circle,scale=0.3,left](4) at (0,4){};
\node[right](8.1,0) at (-.9,4.25){$1$};
\node[fill,draw,circle,scale=0.3,right](4r) at (2,4){};
\node[right](4.1) at (2.1,4.25){$1$};

\node[fill,draw,circle,red,scale=0.3,left](2) at (0,2){};
\node[right](2.1,0) at (-.9,2.25){$0$};
\node[fill,draw,circle,red,scale=0.3,right](2r) at (2,2){};
\node[right](2.1) at (2.1,2.25){$0$};
%-----------
\draw[-](2)--(6r);
\draw[-](4)--(8r);
%---------
\node at (8)[above=3pt]{$\vdots$};
\node at (8r)[above=3pt]{$\vdots$};
\node at (2)[below=-1pt]{$\vdots$};
\node at (2r)[below=-1pt]{$\vdots$};
%--------
\node at (1,-0.1){\normalsize$f^{\circ}\div n$};
\end{tikzpicture}
\quad
\begin{tikzpicture}
[scale=0.43]  
\node[fill,draw,circle,scale=0.3,left](8) at (0,8){};
\node[right](8.1,0) at (-.9,8.25){$3$};
\node[fill,draw,circle,scale=0.3,right](8r) at (2,8){};
\node[right](8.1) at (2.1,8.25){$3$};

\node[fill,draw,circle,scale=0.3,left](6) at (0,6){};
\node[right](6.1,0) at (-.9,6.25){$2$};
\node[fill,draw,circle,scale=0.3,right](6r) at (2,6){};
\node[right](6.1) at (2.1,6.25){$2$};

\node[fill,draw,circle,red,scale=0.3,left](4) at (0,4){};
\node[right](8.1,0) at (-.9,4.25){$1$};
\node[fill,draw,circle,red,scale=0.3,right](4r) at (2,4){};
\node[right](4.1) at (2.1,4.25){$1$};

\node[fill,draw,circle,red,scale=0.3,left](2) at (0,2){};
\node[right](2.1,0) at (-.9,2.25){$0$};
\node[fill,draw,circle,red,scale=0.3,right](2r) at (2,2){};
\node[right](2.1) at (2.1,2.25){$0$};
%-----------
\draw[-](2)--(4r);
\draw[-](4)--(6r);
\draw[-](6)--(8r);
%---------
\node at (8)[above=3pt]{$\vdots$};
\node at (8r)[above=3pt]{$\vdots$};
\node at (2)[below=-1pt]{$\vdots$};
\node at (2r)[below=-1pt]{$\vdots$};
%--------
\node at (1,-0.1){\normalsize$g_{\tilde{\Delta},d}$};
\end{tikzpicture}
\quad
\begin{tikzpicture}
[scale=0.43]
%--------
\node[fill,draw,circle,scale=0.3,left](9) at (0,9){};
\node[right](9.1,0) at (-.9,9.25){$7$};
\node[fill,draw,circle,scale=0.3,right](9r) at (2,9){};
\node[right](9.1) at (2.1,9.25){$7$};

\node[fill,draw,circle,scale=0.3,left](8) at (0,8){};
\node[right](8.1,0) at (-.9,8.25){$6$};
\node[fill,draw,circle,scale=0.3,right](8r) at (2,8){};
\node[right](8.1) at (2.1,8.25){$6$};

\node[fill,draw,circle,scale=0.3,left](7) at (0,7){};
\node[right](7.1,0) at (-.9,7.25){$5$};
\node[fill,draw,circle,scale=0.3,right](7r) at (2,7){};
\node[right](7.1) at (2.1,7.25){$5$};

\node[fill,draw,circle,scale=0.3,left](6) at (0,6){};
\node[right](6.1,0) at (-.9,6.25){$4$};
\node[fill,draw,circle,scale=0.3,right](6r) at (2,6){};
\node[right](6.1) at (2.1,6.25){$4$};

\node[fill,draw,circle,scale=0.3,left](5) at (0,5){};
\node[right](9.1,0) at (-.9,5.25){$3$};
\node[fill,draw,circle,scale=0.3,right](5r) at (2,5){};
\node[right](5.1) at (2.1,5.25){$3$};

\node[fill,draw,circle,red,scale=0.3,left](4) at (0,4){};
\node[right](8.1,0) at (-.9,4.25){$2$};
\node[fill,draw,circle,red,scale=0.3,right](4r) at (2,4){};
\node[right](4.1) at (2.1,4.25){$2$};

\node[fill,draw,circle,scale=0.3,left](3) at (0,3){};
\node[right](3.1,0) at (-.9,3.25){$1$};
\node[fill,draw,circle,scale=0.3,right](3r) at (2,3){};
\node[right](3.1) at (2.1,3.25){$1$};

\node[fill,draw,circle,red,scale=0.3,left](2) at (0,2){};
\node[right](2.1,0) at (-.9,2.25){$0$};
\node[fill,draw,circle,red,scale=0.3,right](2r) at (2,2){};
\node[right](2.1) at (2.1,2.25){$0$};
%---------
\node at (9)[above=3pt]{$\vdots$};
\node at (9r)[above=3pt]{$\vdots$};
\node at (2)[below=-1pt]{$\vdots$};
\node at (2r)[below=-1pt]{$\vdots$};
%--------
\draw[-](2)--(4r);
\draw[-](3)--(5r);
\draw[-](4)--(6r);
\draw[-](5)--(7r);
\draw[-](6)--(8r);
\draw[-](7)--(9r);

%--------
\node at (1,-0.1){\normalsize$g_{\tilde{\Delta},d}\times n$};
\end{tikzpicture}
\quad
\begin{tikzpicture}
[scale=0.43]
%--------
\node[fill,draw,circle,scale=0.3,left](9) at (0,9){};
\node[right](9.1,0) at (-.9,9.25){$7$};
\node[fill,draw,circle,scale=0.3,right](9r) at (2,9){};
\node[right](9.1) at (2.1,9.25){$7$};

\node[fill,draw,circle,scale=0.3,left](8) at (0,8){};
\node[right](8.1,0) at (-.9,8.25){$6$};
\node[fill,draw,circle,scale=0.3,right](8r) at (2,8){};
\node[right](8.1) at (2.1,8.25){$6$};

\node[fill,draw,circle,scale=0.3,left](7) at (0,7){};
\node[right](7.1,0) at (-.9,7.25){$5$};
\node[fill,draw,circle,scale=0.3,right](7r) at (2,7){};
\node[right](7.1) at (2.1,7.25){$5$};

\node[fill,draw,circle,scale=0.3,left](6) at (0,6){};
\node[right](6.1,0) at (-.9,6.25){$4$};
\node[fill,draw,circle,scale=0.3,right](6r) at (2,6){};
\node[right](6.1) at (2.1,6.25){$4$};

\node[fill,draw,circle,scale=0.3,left](5) at (0,5){};
\node[right](9.1,0) at (-.9,5.25){$3$};
\node[fill,draw,circle,scale=0.3,right](5r) at (2,5){};
\node[right](5.1) at (2.1,5.25){$3$};

\node[fill,draw,circle,red,scale=0.3,left](4) at (0,4){};
\node[right](8.1,0) at (-.9,4.25){$2$};
\node[fill,draw,circle,red,scale=0.3,right](4r) at (2,4){};
\node[right](4.1) at (2.1,4.25){$2$};

\node[fill,draw,circle,scale=0.3,left](3) at (0,3){};
\node[right](3.1,0) at (-.9,3.25){$1$};
\node[fill,draw,circle,scale=0.3,right](3r) at (2,3){};
\node[right](3.1) at (2.1,3.25){$1$};

\node[fill,draw,circle,red,scale=0.3,left](2) at (0,2){};
\node[right](2.1,0) at (-.9,2.25){$0$};
\node[fill,draw,circle,red,scale=0.3,right](2r) at (2,2){};
\node[right](2.1) at (2.1,2.25){$0$};
%---------
\node at (9)[above=3pt]{$\vdots$};
\node at (9r)[above=3pt]{$\vdots$};
\node at (2)[below=-1pt]{$\vdots$};
\node at (2r)[below=-1pt]{$\vdots$};
%--------
\draw[-](2)--(4r);
\draw[-](3)--(4r);
\draw[-](4)--(6r);
\draw[-](5)--(6r);
\draw[-](6)--(8r);
\draw[-](7)--(8r);

%--------
\node at (1,-0.1){\normalsize$f_{\Delta,e}$};
\end{tikzpicture}
}
\end{center}
\caption{
Shortening an element of $\m F_n(\mathbb{Z})$
%Elements of  $\m F_n(\mathbb{Z})$ and $\m F_1(\mathbb{Z})$
}
\label{f: figures}
%\caption{}
\label{f:omega bar}
\end{figure}
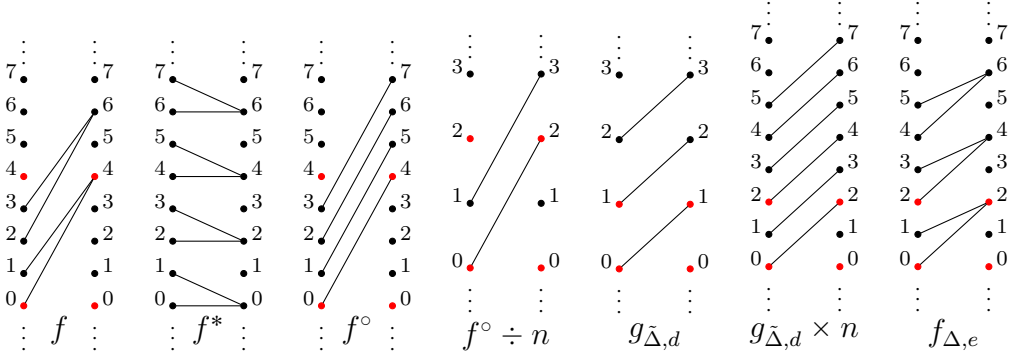

\begin{lemma}\label{l:bounded for F_N(Z)}
Given  a finite sub c-chain $\m \Delta$ of $\mathbb{Z}$, % (not necessarily convex).
there exists an $n$-short spacing embedding with domain $\m \Delta$ that transfers $n$-periodicity. %$\psi: \m \Delta \rightarrow (\mathbb{Z}, \leq_{\mathbb{Z}},\diagcov_{\mathbb{Z}})$.% of height at most $\nu(|\Delta|)$.% a there exists $g\in F_N(\mathbb{Z})$ such that  $g_i\psi(a)=\psi f_i(a)$ for all $a,f(a)\in\Delta$.
\end{lemma}

\begin{proof}
%Consider first all the possible intersection of automorphisms of $\mathbb{Z}$ and $\Delta\times\Delta$, observe that there are finite such intersections. Let us consider for each of these intersections a function that yields that intersection, say $f_1,\dots,f_n\in F_n(\mathbb{Z})$. By Lemma~\ref{l:fof* decomp} for each $f_i$ there exist $f_i^{\circ}\in F_1(\mathbb{Z})$ and $f_i^{*}\in F_n(\mathbb{Z})$ with $f_i^{*}(0)=0$, such that $f_i=f_i^{\circ}\circ f_i^{*}$. For convenience, 
For $\widetilde{\Delta}:=\{Qx:x\in\Delta\}\cup\{Qx+1:x\in\Delta\}\cup\{Qx-1:x\in\Delta\}$
 %\begin{align*}
 %\widetilde{\Delta}&:=\{Sx:x\in\Delta\}\cup\{Sx+n:x\in\Delta\}\cup\{Sx-n:x\in%\Delta\}
% \end{align*}
 we have $|\widetilde{\Delta}|\leq 3|\Delta|$; let $\widetilde{\mathbf{\Delta}}$ be the sub-chain of $\mathbb{Z}$ with base set $\widetilde{\Delta}$. By Lemma~\ref{l: bounded spacing embedding in Aut(Z)}, there exists a spacing embedding $\overline{\cdot}:\widetilde{\Delta}\rightarrow\mathbb{Z}$ that transfers $1$-periodicity and has height at most $\rho(|\widetilde{\Delta}|)$.
 %We want to build an $n$-short spacing embedding with domain $\m \Delta$ that transfers $n$-periodicity. 
 Let $e:\Delta\rightarrow\mathbb{Z}$ be given by $e(x)=\overline{Qx}n+Rx$, for all $x\in \Delta$.
 %$e(kn+r)=\overline{kn}+r$ for all $k,r\in\mathbb{Z}$ where $0\leq r<n$ and $kn+r\in\Delta$.
 Observe that $\max(e[\Delta])\leq %(d+1)n\leq 
 (\rho(|\widetilde{\Delta}|)+1)n\leq 
 (\rho(3|{\Delta}|)+1)n =\nu(|{\Delta}|)$; hence $e$ is $n$-short. We will show that $e$ transfers $n$-periodicity. 
 
 If $f\in F_n(\mathbb{Z})$ and $f|_{\Delta\times\Delta}$ is non-empty, %\CHAT{N}{Why can we assume non-empty?}\CHAT{I}{I think this should be part of the transfers n-periodicity definition}
  by Lemma~\ref{l:fof* decomp} there exist $f^{\circ}\in F_1(\mathbb{Z})$ and $f^{*}\in F_n(\mathbb{Z})$ with $0<f^{*}(0)\leq n$, $Sf^{\circ}(0)=f^{\circ}(0)$, %\CHAT{N}{You mean $0\leq f^{*}(0)<n$?}\CHAT{I}{Yes} 
 and $f=f^{\circ}\circ f^{*}$. 
 Using the $n$-periodicity of $f^*$ and that $f^\circ(a)=f^\circ(0)+a$, for all $a$, we obtain
 $f(x)=(f^{\circ}\circ f^*)(Sx+Rx)=f^{\circ}(0)+f^*(Sx+Rx)=f^{\circ}(0)+Sx+f^*(Rx)
         =f^{\circ}(Sx)+f^*(Rx)$.
         % \begin{align*}
%     f(x)&=f(Sx+Rx)\\
%         &=(f^{\circ}\circ f^*)(Sx+Rx)\\
%         &=f^{\circ}(0)+f^*(Sx+Rx)\\         
%         &=f^{\circ}(0)+Sx+f^*(Rx)\\
%         %\text{\CHAT{N}{Why?}\CHAT{I}{periodicity}}\\
%         %&\text{\CHAT{N}{I dont see how perodicity applies, as $Sx$ is not a multiple of $n$,}}\\
%         &=f^{\circ}(Sx)+f^*(Rx)\\
%         &=Sf^{\circ}(Sx)+f^*(Rx)
% \end{align*}
Since $f^{\circ}(Sx)=Sx+f^\circ(0)=Sx+Sf^\circ(0)=S(x+f^\circ(0))=Sf^{\circ}(Sx)$, 
we get that $f^{\circ}(Sx)$ is a multiple of $n$, so $Sf(x)=f^{\circ}(Sx)$ or $Sf(x)=f^{\circ}(Sx)+n$, depending on whether $f^*(Rx)$ is smaller or bigger than $n$; note that indeed $f^*(Rx)\in [0, 2n)_\mathbb{Z}$, since $0 \leq Rx <0+n$ implies % $f^*(0) \leq f^*(Rx) <f^*(0)+n$, 
$0 \leq f^*(0) \leq f^*(Rx) <f^*(0)+n<n+n=2n$, by $n$-periodicity. 
Therefore, $f^{\circ}(Sx)=Sf(x)=n \cdot Qf(x)$ or $f^{\circ}(Sx)=Sf(x)-n=n \cdot Qf(x) -n$; hence
 $Qf^{\circ}(Sx)=Qf(x)$ or $Qf^{\circ}(Sx)= Qf(x) -1$.
In either case,  if $x,f(x)\in\Delta$, then  $Qx, Qf^{\circ}(Sx)\in\widetilde{\Delta}$. 

We define the function $g$ on $\mathbb{Z}$, by $g(a)=Qf^\circ(a n)$, for all $a \in \mathbb{Z}$; note that $g(Qx)=Qf^{\circ}(Sx)$, %$=\frac{f^{\circ}(Sx)}{n}$
for all $x\in \mathbb{Z}$ (denoted by $f^{\circ}\div n$ in the Example~\ref{e: 2}). Observe that 
%since $f(x) \in \Delta$ in both cases $g(Sx)=Sf^{\circ}(Sx)\in\widetilde{\Delta}$. Hence,
if $x,f(x)\in \Delta$, then  $Qx,g(Qx)\in\widetilde{\Delta}$;
 %and  $f^{\circ}(Sx)=Sf^{\circ}(Sx)=g(Sx)$
hence, since $f|_{\Delta\times\Delta}$ is non-empty so is $g|_{\widetilde{\Delta}\times\widetilde{\Delta}}$. Since $\overline{\cdot}:\widetilde{\Delta}\rightarrow\mathbb{Z}$  transfers $1$-periodicity, the counterpart of $g|_{\widetilde{\Delta}\times\widetilde{\Delta}}$  by $\overline{\cdot}$ can be extended to some $h\in F_1(\mathbb{Z})$. We further define the function $k\in F_1(\mathbb{Z})$ given by $k(x)=h(0)n+x$ (denoted by $h\times n$ in the Example~\ref{e: 2}). 

 Note that $k\circ f^*\in F_n(\mathbb{Z})$; we will show now that 
$k\circ f^*$ extends the counterpart of $f|_{\Delta\times\Delta}$ by $e$. Assume that we have $x,f(x)\in\Delta$. We  first  show that
$e f(x)=\overline{Qf^{\circ}(Sx)}n+f^{*}(Rx)$, i.e., 
$e(f^{\circ}(Sx)+f^{*}(Rx))=\overline{Qf^{\circ}(Sx)}n+f^{*}(Rx)$. When $0\leq f^{*}(Rx)<n$, then $R(f^{\circ}(Sx)+f^{*}(Rx))=Rf^*(Rx)=f^*(Rx)$ and $Q(f^{\circ}(Sx)+f^*(Rx))=Qf^{\circ}(Sx)$, so the equation holds.
%then $e(f^{\circ}(Sx)+f^{*}(Rx))=\overline{f^{\circ}(Sx)}+f^{*}(Rx)$.
If $n\leq f^{*}(Rx) <2n$%(recall that $f^*(Rx)\in [0, 2n)_\mathbb{Z}$, so that is the only remaining case)
, then 
%Assume now that $f^{*}(Rx)\geq n$. Since $f^{*}$ is $n$-periodic, $f^{*}(Rx)-f^{*}(0)\leq n$ hence $f^{*}(Rx)\leq f^{*}(0)+n<2n$. Therefore, 
$0\leq f^{*}(Rx)-n<n$, so $R(f^{\circ}(Sx)+f^{*}(Rx))=f^*(Rx)-n$ and %$S(f^{\circ}(Sx)+f^*(Rx))=f^{\circ}(Sx)+n$, hence 
$\overline{Q(f^{\circ}(Sx)+f^*(Rx))}=\overline{Qf^{\circ}(Sx)+1}=\overline{Qf^{\circ}(Sx)}+1$, since $\overline{\cdot}$ preseves coverings,
%\CHAT{N}{How do we know that $\overline{a+1}=\overline{a}+1$?}\CHAT{I}{Because $\overline{\cdot}$ is a spacing embedding}
so  the equation holds in this case, as well.
%\begin{align*}
%e{f(Sx+Rx)}&=e(f^{\circ}(Sx)+f^{*}(Rx))\\
%               &=e((f^{\circ}(Sx)+n)+(f^{*}(Rx)-n))\\
%               &=\overline{(f^{\circ}(Sx)+n)}+(f^{*}(Rx)-n))\\
%               &=\overline{f^{\circ}(Sx)}+n+(f^{*}(Rx)-n))\\
%               &=\overline{f^{\circ}(Sx)}+f^{*}(Rx)
%\end{align*}
%Then in both cases $e f(x)=\overline{f^{\circ}(Sx)}+f^{*}(Rx)$. Observe now that for all $x,f(x)\in\Delta$,
Therefore,
%\CHAT{N}{Below we seem to be using that since $Sx$ is a multiple of $n$, then also $\overline{Sx}$ is a multiple of $n$.}
$e{f(x)}=e (f^{\circ}(Sx)+f^*(Rx))=\overline{Q{f^{\circ}(Sx)}}n+f^*(Rx)=\overline{g(Qx)}n+f^*(Rx)=h(\overline{Qx})n+f^*(Rx)=h(0)n+\overline{Qx}n+f^*(Rx)
=k(0)+\overline{Qx}n+f^*(Rx)
=k(\overline{Qx}n+f^*(Rx))
=kf^*(\overline{Qx}n+Rx)
=kf^*(e x)$.
%\begin{align*}
%e{f(x)}&=e (f^{\circ}(Sx)+f^*(Rx))\\
%          &=\overline{Q{f^{\circ}(Sx)}}n+f^*(Rx)\\
%          &=\overline{g(Qx)}n+f^*(Rx)\\
%          &=h(\overline{Qx})n+f^*(Rx)\\
%          &=h(0)n+\overline{Qx}n+f^*(Rx)\\
%          &=k(0)+\overline{Qx}n+f^*(Rx)\\
%          &=k(\overline{Qx}n+f^*(Rx))\\
%          &=kf^*(\overline{Qx}n+Rx)\\
%          &=kf^*(e x)
%\end{align*}
%Therefore, $e$ is a $n$-short spacing embedding that transfers $n$-periodicity. 
\end{proof}

%Observe that this shows that any $n$-periodic compatible surjection is short, since there exists a $n$-short spacing embedding.

\begin{lemma}\label{l:short compatible surjection}
    Every  $n$-periodic compatible surjection (with respect to some  spacing embedding) is $n$-periodic with respect to an $n$-short spacing embedding.
\end{lemma}
\begin{proof}
    If $\varphi: \Delta_\varepsilon \rightarrow \mathbb{N}_d$ is an $n$-periodic compatible surjection for some equation $\varepsilon(x_1, \ldots, x_l)$, there exists a spacing embedding $e:\mathbb{N}_d\rightarrow\mathbb{Z}$ such that the partial functions $g_1,\ldots,g_l$ given by $\varphi$ are $n$-periodic respect to $e$. Hence, each counterpart $g_i^e$ is a partial function on $e[\mathbb{N}_d]\subseteq \mathbb{Z}$ that is $n$-periodic with respect to $id_\mathbb{Z}$, and by Lemma~\ref{l:extention of counterparts}  can be extended to a function in $F_n(\mathbb{Z})$. By Lemma~\ref{l:bounded for F_N(Z)}, there exists an $n$-short spacing embedding $e':e[\mathbb{N}_d]\rightarrow\mathbb{Z}$ that transfers $n$-periodicity, so $e'\circ e:\mathbb{N}_d\rightarrow\mathbb{Z}$ is an $n$-short spacing embedding (since it has the same image as $e'$, which is $n$-short) and each $g_i$ has an $n$-periodic counterpart by $e'\circ e$ (since $e'$ transfers $n$-periodicity). Hence, $\varphi$ is $n$-periodic with respect to  the $n$-short spacing embedding $e'\circ e$.
\end{proof}

% \subsection{From an $n$-diagram to $\m F_n(\mathbb{Z})$.}
 \subsection{Decidability.}

 Finally, we are ready obtain decidability results for $\mathsf{V}(\m{F}_n(\mathbb{Z}))$ and  $\mathsf{DLP}$.

%\CHAT{N}{We need to complete the rest of the section}\CHAT{I}{Wasn't this proved in Lemma 3.4.? Do we need the conerse here?}
%Now we will use the result in the previous section to link the failure of an equation on $\m{F}_n(\mathbb{Z})$ to the existence of short $n$-periodic compatible surjections.
\begin{lemma}\label{l: failure in F_n, n-diagram failure }
An equation fails in $\m{F}_n(\mathbb{Z})$ iff it fails in an $n$-short, $n$-periodic compatible surjection.
\end{lemma}
\begin{proof}

   % Suppose $\varepsilon$ in the form $1\leq w_{1}\vee\ldots\vee w_{k}$, fails on $\m{F}_n(\mathbb{Z})$, then there exists a homomorphism $\varphi:\m{Tm}\rightarrow\m{F_n}(\mathbb{Z})$ and $p\in\mathbb{Z}$ such that $\varphi(1)(p) > \varphi( w_{1})(p), \ldots  , \varphi(w_{k})(p)$. Consider the diagram induced by $\varphi$, since $Im(\varphi)\subseteq{F_n}(\mathbb{Z})$ and $\Delta_{\varepsilon}$ is finite, there exists a spacing embedding with respect to which the diagram is $n$-periodic. Now, by Lemma~\ref{l:bounded for F_N(Z)}, there exists a short spacing embedding for the diagram. Hence, the diagram is short and $\varepsilon$ fails in it.
   Suppose than an equation $\varepsilon$ fails in $\m{F}_n(\mathbb{Z})$, then by Theorem~\ref{t: DLPG to comsur}, there exists an $n$-periodic compatible surjection $\varphi$ in which $\varepsilon$ fails. By Lemma~\ref{l:short compatible surjection}, there exists an $n$-short spacing embedding with respect to which $\varphi$ is also $n$-periodic. Hence $\varepsilon$ fails in an $n$-short, $n$-periodic compatible surjection. The converse follows from Lemma~\ref{l: nperdiagram2FnZ}.
 \end{proof}

%We would to that using the fact that for all $\varepsilon$ there exists at most $2^{|\Delta_{\varepsilon}|^3\nu(|\Delta_{\varepsilon}|)}$ short $n$-periodic diagrams up to isomorphisms.
\begin{theorem}\label{t: decidability of V(F_n(Z))}
    The equational theory of the variety $\mathsf{V}(\m{F}_n(\mathbb{Z}))$ is decidable.
\end{theorem}
\begin{proof}
    An equation $\varepsilon$ fails in $\mathsf{V}(\m{F}_n(\mathbb{Z}))$ iff it fails in $\m{F}_n(\mathbb{Z})$ iff (by Theorem~\ref{l: failure in F_n, n-diagram failure }) $\varepsilon$ fails in an $n$-short, $n$-periodic compatible surjection. There exist finitely many compatible surjections with domain $\Delta_{\varepsilon}$ (since all of them have a range of the form $\mathbb{N}_q$, where $q \leq |\Delta_\varepsilon|$). Also, for each such compatible surjection, there are finitely many  $n$-short spacing embeddings under which the compatible surjection is $n$-periodic, as the image of each $n$-short spacing embedding is contained in $\mathbb{N}_d$, where $d \leq \nu(q)$. Therefore, by checking all of these finitely-many situations we obtain an algorithm that decides every equation. 
\end{proof}

Theorem~\ref{t: decidability of V(F_n(Z))} and Theorem~\ref{t: DLP FnZ} %{t: joinnper} 
yield the following result, which provides a proof of the decidability of $\mathsf{DLP}$ that is different to the one given in \cite{GG}. 

 \begin{corollary}
      The equational theory of the variety $\mathsf{DLP}$ is decidable.
 \end{corollary}

 %\begin{proof}
 %An equation $\varepsilon$ fails in $\mathsf{DLP}$ iff it fails in a diagram iff (by Lemma~\ref{l: diag2nperdiag}) it  fails in an $n$-periodic diagram, where $n=2^{|\varepsilon|}|\varepsilon|^4$ and $|\varepsilon|$ is the length of $\varepsilon$. By \CHAT{N}{????} his is equivalent to the demand that  $\varepsilon$ fails in $\m{F}_n(\mathbb{Z})$ for $n=2^{|\varepsilon|}|\varepsilon|^4$, which is decidable by Theorem~\ref{t: decidability of V(F_n(Z))}.
 %\end{proof}

   \section{Generation and decidability for \textnormal{$\mathsf{LP_n}$}.}\label{s: From F_N(Omega) to an N-Diagram}%\CHAT{N}{This section should be revised after the previous one is complete}

   In this section we will prove that, for all $n$, $\mathsf{LP_n}=\mathsf{V}(\m F_n(\mathbb{Q} \overrightarrow{\times}\mathbb{Z}))$ and that the equational theory of  $\mathsf{LP_n}$ is decidable. A lot of the detailed work we did in the previous sections will come in handy.

   As the failure of an equation in an ($n$-periodic) diagram corresponds to the failure in $\m F_n(\mathbb{Z})$ and the latter does not generate $\mathsf{LP_n}$,  a new notion of diagram is required for $\m F_n(\m J \overrightarrow{\times}\mathbb{Z})$, capturing the natural partiction induced by the lexicographic product ordering on the chain. 
   %Working toward proving that , i.e., that an equation fails in $\mathsf{LP_n}$ iff it fails in $\m F_n(\mathbb{Q} \overrightarrow{\times}\mathbb{Z})$, in this section we prove that if an equation fails in $\mathsf{LP_n}$ then it fails in an \emph{$n$-diagram}. 
   %Since, given the first representation theorem, we want to capture failure in $\m F_n(\mathbb{Q} \overrightarrow{\times}\mathbb{Z})$, the relevant finite set of points will be a subset of $\mathbb{Q} \overrightarrow{\times}\mathbb{Z}$. As it was shown in Section 3, this partition created by the first coordinate "partitions" also the functions, therefore it is extremely important, even in the diagrams. For reasons that would become clear in Theorem~\ref{t:LP_n to diagrams} it would be convenient to make the set uniform by considering a lexicographic order $\m J_0\overrightarrow{\times}\Delta_0$ for some finite chain $\m J$ and $\Delta_0\subseteq \mathbb{Z}$.
For  $\Delta\subseteq J\times \mathbb{Z}$, where $J$ is a set, we define the equivalence  relation $\equiv$ on $\Delta$  by $(i_1,x)\equiv(i_2, y)$ iff $i_1=i_2$; note that generalizes to subsets the equivalence relation we defined on $J\times \mathbb{Z}$ in Section~\ref{s: FnOmega}.  %\CHAT{I}{It was defined for JxZ in page 10, generalize and define only once?}\CHAT{N}{As it would feel a bit artificial to define it more generally on p. 10 and also since by now the reader may have forgotten the meanins, I think it is ok to define it here for $\Delta$.}

A  \emph{partition diagram} consists of a diagram of the form $(\m{J\overrightarrow{\times}\Delta},\diagcov,g_1,\ldots,g_s)$, where 
%is a if: $\m{\Delta}=\m{J\overrightarrow{\times}\Delta_0}$ where
$\m J$ is a chain and $\m \Delta$ is a finite subchain of $\mathbb{Z}$ such that 
$x \diagcov y \Rightarrow x \equiv y$
and
$x\equiv y$ iff $g_i(x)\equiv g_i(y)$, for all $i\in \{1,\ldots,s\}$ and $x,y\in Dom(g_i)$, together with the exact partition induced by $\m J$ and $\m \Delta$; so if $\m{J\overrightarrow{\times}\Delta}$ is partitioned differently then we get a different partition diagram. For each 
partial function $g$ of the diagram
%$g\in\{g_1,\ldots,g_s\}$ 
we  define the partial function $\widetilde{g}: J \rightarrow J$ where $\widetilde{g}(j)=k$ iff there exist $x,y\in \mathbb{Z}$ such that $(j,x)\in Dom(g)$ and $g(j,x)=(k,y)$. Note that that $\widetilde{g}$ is well-defined  because if  
%$x_1,x_2\in \mathbb{Z}$, 
%$(j,x_1),(j,x_2)\in Dom(g)$, 
$g(j,x_1)=(k_1,y_1)$ and $g(j,x_2)=(k_2,y_2)$,
then $(j,x_1)\equiv(j,x_2)$ implies $(k_1,y_1)=g(j,x_1)\equiv g(j,x_2)=(k_2,y_2)$, hence $k_1=k_2$; moreover, $\widetilde{g}$ is one-to-one. We also define $\overline{g_j}:\Delta\rightarrow \{j\}\times\Delta\stackrel{g}{\rightarrow}\{\tilde{g}(j)\}\times\Delta\rightarrow\Delta$, for every $j\in J$, where the first and last maps are the obvious bijections. We say that an equation \emph{fails} in a partition diagram, if it fails in the underlying diagram.

A partition diagram $(\m J\overrightarrow{\times} \m \Delta,\diagcov,g_1,\ldots,g_s)$ is said to be \emph{$n$-periodic} if there exists a spacing embedding $e$ on $\m \Delta$ such that  $\overline{g}_j$ is $n$-periodic with respect to $e$, for all $j\in J$ and $g\in\{g_1,\ldots,g_s\}$. A $n$-periodic partition diagram is said to be \emph{$n$-short} if the spacing embedding witnessing $n$-periodicity is $n$-short. Note that we insist that all  parts of the partition in a partition diagram have the same size, as this facilitates the definition of $n$-periodicity: the spacing embedding $e$ and the partial functions  $\overline{g}_j$, for $j \in J$, are all defined on $\m \Delta$, not on  $\m J\overrightarrow{\times} \m \Delta$.

\begin{theorem}\label{t:LP_n to diagrams}
    If an equation fails in $\mathsf{LP_n}$, where $n\in\mathbb{Z}$, it also fails in an $n$-short $n$-periodic partition diagram of size at most $|\Delta_{\varepsilon}|^2$.
\end{theorem}
  \begin{proof}
      If an equation $\varepsilon=\varepsilon(x_1,\ldots,x_l)$ in intensional  form $1\leq w_{1}\vee\ldots\vee w_{k}$ fails in $\mathsf{LP_n}$, then it fails in some $n$-periodic $\ell$-pregroup. By Theorem~\ref{t: emb to F_N(Omega)} it fails in   $\m F_n(\m{\Omega})$, where $\m{\Omega}=\m{J}\overrightarrow{\times}\mathbb{Z}$ for some chain $\m{J}$. So, there exists $f=(f_1,\ldots,f_l)$, where $f_1,\cdots,f_l \in F_n(\Omega)$, and $p\in\Omega$  with $w_{1}^{\m F_n(\m \Omega)}(f)(p), \ldots, w_{k}^{\m F_n(\m \Omega)}(f)(p)>p$. Let $\psi:\m {Ti} \ra \m F_n(\m \Omega)$ be the homomorphism witnessing the failure of $\varepsilon$, where $\psi(x_i)=f_i$ and let $\psi^\pm:\m {Ti}^\pm \ra \m F_n(\Omega)^\pm$ be the extension of $\psi$. Also, we consider the compatible surjection $\psi_p:\Delta_{\varepsilon}\rightarrow\psi_p[\Delta_{\varepsilon}]$ given in Theorem~3.6 of \cite{GG}, where the order of $\psi_p[\Delta_{\varepsilon}]$ is induced by $\m{\Omega}$ and
      $$\psi_p(u):=\psi^\pm(u)(p)=u^{\m F_n(\Omega)^\pm}(f)(p).$$
      Let $\m J_0$ be the subchain of $\m J$ with $J_0=\{j\in J:\exists x\in\mathbb{Z}, (j,x)\in \psi_p[\Delta_{\varepsilon}]\}$
      and $\m \Delta_0$ the subchain of $\mathbb{Z}$ with $\Delta_0=\{x\in \mathbb{Z}:\exists j\in J, (j,x)\in \psi_p[\Delta_{\varepsilon}]\}$.   We consider the diagram $\m \Delta=(\m J_0\overrightarrow{\times}\m \Delta_0,\diagcov,g_1,\ldots,g_s)$,   where  $g_i$ is the restriction of $f_i$ to $\Delta \times \Delta$.
      %the partial functions determined by the compatible surjection, observe that this are also partial functions over $\Delta$ since $\psi_p[\Delta_{\varepsilon}]\subseteq\Delta$. 
      Then, for every $j \in J_0$ and $g\in\{g_1,\ldots,g_s\}$, the partial function $\overline{g}_j$ is the restriction of the corresponding $\overline{f}_j \in F_n(\mathbb{Z})$, given in Theorem~\ref{t:F(JxZ)}, to $\Delta_0$. 
      %  $f\in\{f_1,\ldots,f_n\}$ is an $n$-periodic function, by Theorem~\ref{t:F(JxZ)}, we can consider $\overline{f}_j\in F_n(\mathbb{Z})$ for all $j\in J_0$ and moreover $\overline{g}_j=\overline{f}_j$ is a restriction of such a function. 
      By Lemma~\ref{l:bounded for F_N(Z)}, there exists an $n$-short spacing embedding with domain $\m \Delta_0$ that transfers $n$-periodicity. 
      Hence, $\m {\Delta}$ is an $n$-short $n$-periodic partition diagram of size at most $|\Delta_{\varepsilon}|^2$, since $|\Delta|=|J_0||\Delta_0|\leq |\Delta_\varepsilon|^2$.

      Consider now the intensional homomorphism $\varphi: \m {Ti} \rightarrow \m {Pf}(\m \Delta)$ 
      %, from the algebra $\m {Ti}$ of all intentional terms,
      extending the assignment $\varphi(x_i)=g_i$ and observe that $p \in \Delta$. Since each $g_i$ is a restriction of $f_i\in F_n(\Omega)$, for all $1 \leq i\leq l$, we have  $\varphi(u)(p)=u^{\m F_n(\m \Omega)^\pm}(f)(p)=\psi_p(u)$, for all $u\in FS$, so  $\varphi(1)(p) > \varphi( w_{1})(p), \ldots  , \varphi(w_{k})(p)$, %and $\varphi(x_i)=f_i$
      thus witnessing the failure of $\varepsilon$ in the partition diagram.
  \end{proof}

The next lemma and theorem move from a failure in an $n$-periodic partition diagram to a failure in $\m{F}_n(\mathbb{Q}\overrightarrow{\times}\mathbb{Z})$.
%On the other hand, if we have a failure in an $n$-periodic partition diagram, we would like to extend it to a failure in $\mathsf{LP_n}$. Following the construction used so far we would like to find a failure in $\m F_n(\m\Omega)$ for some chain $\m\Omega$, observe that by the definition of partition diagrams it is natural to consider $\m\Omega=\m J\times \mathbb{Z}$ for a chain $\m J$, however the question becomes what is a general good choice for $\m J$. While, the choice of $\mathbb{Z}$ is enough for the second coordinate the role of $\m J$ would be required to allow for any 1-1 partial function to be extended to a function in $F(\m{J})$; therefore the correct choice for $\m \Omega$ happen to be $\mathbb{Q}\overrightarrow{\times}\mathbb{Z}$.

%\CHAT{I}{new}
\begin{lemma}\label{l: calculating g^{[m]} in partition diagrams}
    %Given an $n$-periodic partition diagram $\m{\Delta}=(\m J_0 \overrightarrow{\times} \Delta_0,\diagcov,g_1,\ldots,g_l)$,
    If $g$ is a partial function in an $n$-periodic partition diagram, then
    $g^{(m)}(j,a)=(\tilde{g}^{(-1)^m}(j),\overline{g}_{\tilde{g}^{(-1)^m}(j)}^{[m]}(a))$, for all $m\in\mathbb{Z}$
    and $(j,a)\in Dom(g^{[m]})$.
    %for some $g\in\{g_1,\ldots,g_n\}$
\end{lemma}

\begin{proof}
    We will prove the claim inductively for positive $m$, as the verification for negative $m$'s is similar; 
    %Suppose $g\in\{g_1,\ldots,g_n\}$, and 
    assume the claim is true for $m\in\mathbb{Z}^+$ and show it for $m+1$. If $(j,a)\in Dom(g^{[m+1]})$, there exist $(i_1,b),(i_2,c)\in Dom(g^{[m]})$ such that $(i_1,b)\diagcov(i_2,c)$, $g^{[m]}(i_1,b)<(j,a)\leq g^{[m]}(i_2,c)$ and $g^{[m+1]}(j,a)=(i_2,c)$, by Lemma~\ref{l: g^[ell] order preserving}. Since $(i_1,b)\diagcov(i_1,c)$, we get $i_1=i_2$ and $c=b+1$. Therefore, $(\tilde{g}^{(-1)^m}(i_2),\overline{g}_{i_2}^{[m]}(b))<(j,a)\leq (\tilde{g}^{(-1)^m}(i_2),\overline{g}_{i_2}(b+1))$, hence $j=\tilde{g}^{(-1)^m}(i_2)$ and $\overline{g}_{i_2}^{[m]}(b)<a\leq\overline{g}_{i_2}^{[m]}(b+1)$. Hence, $i_2=(\tilde{g}^{(-1)^m})^{-1}(j)$ and $c=b+1=\overline{g}_{i_2}^{[m+1]}(a)$, so $g^{[m+1]}(j,a)=(i_2,c)=(\tilde{g}^{(-1)^{m+1}}(j),g_{\tilde{g}^{(-1)^{m+1}}(j)}^{[m+1]}(a))$.
\end{proof}
  
  \begin{theorem}\label{t:diagram implies F_n(QxZ)}
     For every $n$, if an equation fails in an $n$-periodic partition diagram, it also fails in $\m{F}_n(\mathbb{Q}\overrightarrow{\times}\mathbb{Z})$. 
  \end{theorem}
  
  \begin{proof}
      If the equation $\varepsilon=\varepsilon(x_1,\ldots,x_l)$ in intensional form $1\leq w_{1}\vee\ldots\vee w_{k}$ fails in an $n$-periodic partition diagram $\m{\Delta}=(\m J_0 \overrightarrow{\times} \Delta_0,\diagcov,g_1,\ldots,g_l)$, there exist $p\in\Delta$ and an intensional homomorphism $\varphi: \m {Ti} \rightarrow \m {Pf}(\m \Delta)$ %, from the algebra $\m {Ti}$ of all intensional terms, and a point $p \in \Delta$ 
      such that $\varphi(x_i)=g_i$, for all $1 \leq i\leq l$, and   $\varphi(1)(p) > \varphi( w_{1})(p), \ldots  , \varphi(w_{k})(p)$. For each $g\in\{g_1,\ldots,g_l\}$ we consider the partial function $\widetilde{g}$ on $J_0$ and,  for all $j\in J_0$, the partial function $\overline{g}_j$ on $\Delta_0$.  By Lemma~\ref{l:bounded for F_N(Z)} there exists an $n$-short $n$-periodic spacing embedding  $e:\m \Delta_0\rightarrow\mathbb{Z}$ that transfers $n$-periodicity, so, for all $j \in J_0$,  the counterpart of $\overline{g}_j$ with respect to $e$ extends to a function $\overline{f}_j\in F_n(\mathbb{Z})$. %, that exists since $e$ transfers $n$-periodicity. 
      Observe that since $e$ is a spacing embedding and $\overline{f}_j$ extends $\overline{g}_j$, by (repeated applications of) Lemma~\ref{l: g^[ell] order preserving}  we get that $\overline{f}_j^{(m)}$ extends $\overline{g}_j^{[m]}$. 
      %\CHAT{I}{it is lemma 3.5 in previos paper} \CHAT{N}{Is this because of Lemma~\ref{l:e, composition and residuals}, or something else?}
      Also, for $j\notin J_0$, we define $\overline{f}_j=Id_{\mathbb{Z}}$.
      Let $\iota_0:\m J_0\rightarrow\mathbb{Q}$ be any order preserving injection and let $\widetilde{f}\in F_1(\mathbb{Q})$ be any extension of $\iota_0 \widetilde{g}\iota_o^{-1}$. We define $f:\mathbb{Q}\overrightarrow{\times}\mathbb{Z}\rightarrow\mathbb{Q}\overrightarrow{\times}\mathbb{Z}$ by $f(j,a)=(\tilde{f}(j),\overline{f}_j(a))$; by Theorem~\ref{t:F(JxZ)}, $f\in \m{F}_n(\mathbb{Q}\overrightarrow{\times}\mathbb{Z})$.
      Let $\iota:\Delta\rightarrow\mathbb{Q}\overrightarrow{\times}\mathbb{Z}$ be given by $\iota(j,x)=(\iota_0(j), e(x))$. 
      %And let for each $g\in\{g_1,\ldots,g_n\}$, $f:\mathbb{Q}\overrightarrow{\times}\mathbb{Z}\rightarrow\mathbb{Q}\overrightarrow{\times}\mathbb{Z}$ given by $f(j,a)=(\tilde{f}(j),\overline{f}_j(a))$ where $\widetilde{f}\in F_1(\mathbb{Q})$ an extension of $\widetilde{g}$ and $\overline{f}_j\in F_n(\mathbb{Z})$ an extension of the counterpart of $\overline{g}_j$ respect to $e$ that exists since $e$ transfers periodicity. 

      Let $\psi:\m{Ti}\rightarrow\m{F}_n(\mathbb{Q}\overrightarrow{\times}\mathbb{Z})$ be the intensional homomorphism  extending the assignment $\psi(x_i)=f_i$ for all $i\in\{1,\ldots,l\}$; we will show $\psi(u)(p)=\iota\varphi(u)(p)$, for all $u\in FS$. First we show 
       $\iota g^{(m)}(j,a)=f^{(m)}\iota(j,a)$,  for all $(j,a)\in Dom(g_i^{[m]})$. 
       %\CHAT{I}{Do we need to add a lemma showing how to compute $g^{[m]}$}\CHAT{N}{I added that $\diagcov$ is contained inside each component. I think the reason you forgot to add it (but it is important) is that you did not yet write that lemma, where it will be needed.} 
      We have $\iota g^{(m)}(j,a)=\iota(\tilde{g}^{(-1)^m}(j),\overline{g}_{\tilde{g}^{(-1)^m}(j)}^{[m]}(a))=(\iota_0\tilde{g}^{(-1)^m}(j),e\overline{g}_{\tilde{g}^{(-1)^m}(j)}^{[m]}(a))=(\iota_0\tilde{g}^{(-1)^m}(j),(\overline{g}_{\tilde{g}^{(-1)^m}(j)}^{[m]})^e e(a))$,
      by Lemma~\ref{l: calculating g^{[m]} in partition diagrams}  and the definition of the counterpart by $\overline{\cdot}$.
      %, we obtain $(\iota_0\tilde{g}^{(-1)^m}(j),e\overline{g}_{\tilde{g}^{(-1)^m}(j)}^{[m]}(a))=(\iota_0\tilde{g}^{(-1)^m}(j),(\overline{g}_{\tilde{g}^{(-1)^m}(j)}^{[m]})^e e(a))$. 
      Also by Lemma~\ref{l:e, composition and residuals}, 
      $(\overline{g}_{\tilde{g}^{(-1)^m}(j)}^{[m]})^e(e(a))=(\overline{g}_{\tilde{g}^{(-1)^m}(j)}^e)^{[m]}(e(a))$, and
       $(\iota_0\tilde{g}^{(-1)^m}(j),(\overline{g}_{\tilde{g}^{(-1)^m}(j)}^e)^{[m]}(e(a)))=(\tilde{f}(\iota_0 (j)),\overline{f}_{\tilde{f}^{(-1)^m}(j)}^{(m)}(e(a)))=f^{(m)}\iota(j,a)$,
       by the  definition of $f$,
      Hence, $\iota g^{(m)}(j,a)=f^{(m)}\iota(j,a)$.
%\CHAT{I}{new}

      We now show that $\psi(u)(p)=\iota\varphi(u)(p)$, for all $u\in FS$, by induction on the structure of $u$. If $u,x_i^{(m)}u\in FS$ and  $\psi(u)(p)=\iota\varphi(u)(p)$, then we have
      $     \psi(x_i^{(m)}u)(\iota(p))=f_i^{(m)}[\psi(u)(\iota(p)]
      =f_i^{(m)}[\iota\varphi(u)(p)]
      =\iota g_i^{[m]}\iota^{-1}[\iota\varphi(u)(p)]
      =\iota g_i^{[m]}\varphi(u)(p)
      =\iota \varphi(x_i^{(m)}u)(p)$.
%      \begin{align*}
%          \psi(x_i^{(m)}u)(\iota(p))&=f_i^{(m)}[\psi(u)(\iota(p)]\\
%                                    &=f_i^{(m)}[\iota\varphi(u)(p)]\\
%                                    &=\iota g_i^{[m]}\iota^{-1}[\iota\varphi(u)(p)]\\
%                                    &=\iota g_i^{[m]}\varphi(u)(p)\\
%                                    &=\iota \varphi(x_i^{(m)}u)(p)
%      \end{align*}  
      Therefore, $\iota\varphi(1)(p) > \iota\varphi( w_{1})(p), \ldots  , \iota\varphi(w_{k})(p)$,
      by the order preservation of $\iota$, so $\psi(1)(\iota(p)) > \psi( w_{1})(\iota(p)), \ldots  , \psi(w_{k})(\iota(p))$; i.e., $\varepsilon$ fails in $\m{F}_n(\mathbb{Q}\overrightarrow{\times}\mathbb{Z})$.
  \end{proof}

  \begin{theorem}\label{t: LPnfailure}
      An equation fails in $\mathsf{LP_n}$ iff it fails in  an $n$-short $n$-periodic partition diagram of size at most $|\Delta_{\varepsilon}|^2$ 
       iff it fails in  
       $\m{F}_n(\mathbb{Q}\overrightarrow{\times}\mathbb{Z})$.
  \end{theorem}
  \begin{proof}
  If an equation $\varepsilon$  fails in $\mathsf{LP_n}$, then by Theorem~\ref{t:LP_n to diagrams} it  fails in  an $n$-short $n$-periodic partition diagram of size at most $|\Delta_{\varepsilon}|^2$. By  Theorem~\ref{t:diagram implies F_n(QxZ)}, it further fails in $\m{F}_n(\mathbb{Q}\overrightarrow{\times}\mathbb{Z})$. Since the latter is $n$-periodic, the equation  fails in $\mathsf{LP_n}$.
  \end{proof}

  \begin{corollary}
       The variety of $\mathsf{LP_n}$ is generated by $\m{F}_n(\mathbb{Q}\overrightarrow{\times}\mathbb{Z})$.     
  \end{corollary}
 % \begin{proof}
 %     Since $\m{F}_n(\mathbb{Q}\overrightarrow{\times}\mathbb{Z})$ is in the variety they generate a sub-variety of $\mathsf{LP_n}$. On the other hand, by Theorem~\ref{t:LP_n to diagrams}, every equation that fails in $\mathsf{LP_n}$ fails in some $n$-periodic partition diagram and by Theorem~\ref{t:diagram implies F_n(QxZ)}, it fails in $\m{F_n(\mathbb{Q}\overrightarrow{\times}\mathbb{Z})}$, so the variety $\mathsf{LP_n}$ is generated by $\m{F}_n(\mathbb{Q}\overrightarrow{\times}\mathbb{Z})$.
 % \end{proof}

% This theorem is enough to show the decidability of the variety.
  
  \begin{corollary}
       The equational theory of the variety of $\mathsf{LP_n}$ is decidable.   
  \end{corollary}

  \begin{proof}
      In view of Theorem~\ref{t: LPnfailure}, given an  equation $\varepsilon$ it is enough to verify that there are finitely-many partition diagrams of size at most $|\Delta_{\varepsilon}|^2$ and that there is a procedure for verifying whether a partition diagram is $n$-periodic with respect to some $n$-short spacing embedding.
      
      First we argue that there are only finitely-many partition diagrams of size at most $|\Delta_{\varepsilon}|^2$.  If $\m \Delta$ is the underlying c-chain, there are at most  $|\Delta_{\varepsilon}|^2$-many possible sizes for  $|\Delta|$ and for each such size there are at most  $(|\Delta_{\varepsilon}|^2)^2$-many possible choices of placements of $\diagcov$ covererings in $\m \Delta$. Also, there are at most  $|\Delta_{\varepsilon}|^2$-many possible       partitions. Finally, there are at most $2^{|\Delta|}$-many possible partial functions on $\Delta$ and at most ${|\varepsilon|}$-many possible variables in $\varepsilon$; thus there are at most $(2^{|\Delta|})^{|\varepsilon|}$-many possible partial functions in $\m \Delta$, i.e., at most $(2^{|\Delta_{\varepsilon}|^2})^{|\varepsilon|}$-many. Therefore, there are at most $|\Delta_{\varepsilon}|^{16}\cdot (2^{|\Delta_{\varepsilon}|^2})^{|\varepsilon|}$-many possible partition diagrams of size at most $|\Delta_{\varepsilon}|^2$.
      Given that $|\Delta_{\varepsilon}| \leq 2^{|\varepsilon|}|\varepsilon|^4$, %=2^{\ell}\ell^4$, 
      as computed in \cite{GG}, we get that the number of such diagrams is bounded by a function of $|\varepsilon|$.
      %at most $(2^{\ell}\ell^4)^{16}\cdot (2^{(2^{\ell}\ell^4)^2})^{2^{\ell}\ell^4}$-many, where $\ell:=|\varepsilon|$, which is triply exponential in $\ell$.
      
      Also, every $n$-short spacing embedding on $\m \Delta$ has image whose convexification in $\mathbb{Z}$ is contained in $\mathbb{N}_d$, where $d \leq \nu(|\Delta|)$, hence $d \leq \nu(|\Delta_{\varepsilon}|^2)$. Therefore, there are at most $d^{|\Delta|}$-many $n$-short spacing embeddings on $\m \Delta$, hence at most $\nu(|\Delta_{\varepsilon}|^2)^{|\Delta_{\varepsilon}|^2}$-many. Given that $\rho(a):=2a^3a!+a+1$ and $\nu(a)=(\rho(3a)+1)n$, the number of such spacing embeddings is bounded by a function of $|\varepsilon|$.
      %at most $((2(3(2^{\ell}\ell^4)^2)^3(3(2^{\ell}\ell^4)^2)!+3(2^{\ell}\ell^4)^2+2) n)^{(2^{\ell}\ell^4)^2}$.

      Given such a spacing embedding, by Lemma~\ref{l:check periodicity} we need to perform checks for the $n$-periodicity condition $x\leq_{\mathbb{Z}} y+kn \Rightarrow  g^e(x)\leq_{\mathbb{Z}} g^e(y)+ kn$ 
      for all $x,y \in Dom (g^e)$ and $\lceil d/n \rceil$. Therefore, the number of checks that we need to perform at most $|\Delta|^2\cdot \lceil d/n \rceil\leq |\Delta_{\varepsilon}|^2 \cdot \nu(|\Delta_{\varepsilon}|^2)/n$, hence bounded by a function of $|\varepsilon|$. %\leq |\Delta_{\varepsilon}|^2 \cdot (\rho(3|\Delta_{\varepsilon}|^2)+1)\leq (2^{\ell}\ell^4)^2 \cdot (2(3(2^{\ell}\ell^4)^2)^3(3(2^{\ell}\ell^4)^2)!+3(2^{\ell}\ell^4)^2+2)$.
      %
       % Observe now that up to isomorphisms for $\m{J}_0$, there exist finitely many $n$-short $n$-periodic partition diagrams of size at most $|\Delta_{\varepsilon}|^2$, since there is finitely many such diagrams and fixed one of these diagrams there is only finitely many possible $n$-short $n$-periodic spacing embeddings.       Hence, the variety is decidable.      \CHAT{N}{We only need to explain why there are only finitely many $n$-short $n$-periodic partition diagram of size at most $|\Delta_{\varepsilon}|^2$ }\CHAT{I}{Is this enough or do we need more detail?}
  \end{proof}

%\nocite{*}

\end{document}